\documentclass[11pt]{article}
\pdfoutput = 1

\usepackage[T1]{fontenc}                  
\usepackage[utf8]{inputenc}              

\usepackage[svgnames,x11names,hyperref]{xcolor} 

\usepackage{amsmath, amsfonts, amssymb}  
\usepackage{mathrsfs}                    
\usepackage{stmaryrd}                    
\usepackage{mathtools}                   
\usepackage{amsthm}                      
\usepackage{nccmath}                     
\usepackage{scalerel}                    
\usepackage{dsfont}                      
\usepackage[squaren,cdot]{SIunits}
\usepackage{physics}                     
\usepackage{array, multirow, makecell} 
\usepackage{stackengine}                

\AtBeginEnvironment{tabular}{\normalsize}
\usepackage{tikz}                        
\usepackage{graphicx}                    
    \usetikzlibrary{arrows}              %
\usepackage{grffile}  
\usepackage[font=small]{caption}         
\usepackage{subcaption} 
\usepackage{placeins} 

\usepackage{fancyhdr}                    
\usepackage{afterpage}
\usepackage[a4paper, left=3cm, right=3cm, top=3cm, bottom=3.2cm,
twoside]{geometry}                       
\usepackage{enumitem}                    

\usepackage{comment}                     
\usepackage[english]{babel}              
\usepackage[babel=true]{csquotes}        
\usepackage{hyperref}                    
\usepackage{datetime}
\usepackage{cleveref}                    

\makeatletter
\def\namedlabel#1#2{\begingroup #2 \def\@currentlabel{\textnormal{#2}}
  \phantomsection\label{#1}\endgroup}
\def\blfootnote{\gdef\@thefnmark{}\@footnotetext}
\makeatother
\usepackage{titlesec}
\titleformat*{\section}{\centering \large\bfseries\scshape}
\titleformat*{\subsection}{\normalsize \bfseries}
\titleformat{\subsubsection}[runin]{\itshape}{
  \thesubsubsection}{1em}{}
\usepackage[nottoc,numbib]{tocbibind}

\newtheorem{prop}{Proposition}
\newtheorem{theorem}{Theorem}

\newtheorem{lemma}{Lemma}

\newtheorem{definition}{Definition}

\newenvironment{keywords}{\textbf{Keywords: }}{}
\newenvironment{msc}{\textbf{MSC 2020: }}{}
\newenvironment{acknowledgements}{
  \textbf{Acknowledgements. }}{}

\newcommand{\ttfrac}[2]{\scaleobj{.83}{
    \frac{\scaleobj{1}{#1}}{\scaleobj{1}{#2}}}}

\newcommand{\msum}[1]{\displaystyle \scaleobj{.98}{\,
    \sum_{\scaleobj{.9}{#1}}}} %
\newcommand{\mmsum}[2]{\displaystyle \scaleobj{.98}{\,
    \sum_{\scaleobj{.9}{#1}}^{\scaleobj{.9}{#2}}}}

\newcommand{\eint}[2]{\scaleobj{0.66}{
    \mint_{\scaleobj{1.5}{#1}}^{\scaleobj{1.5}{#2}}}}
\newcommand{\eiint}[1]{\scaleobj{0.7}{ \miint_{\scaleobj{1.3}{#1}}}}

\newcommand{\texp}[1]{\hspace{-.5pt} \scaleobj{0.87}{{#1}}}
\newcommand{\ttexp}[1]{\hspace{-.5pt} \scaleobj{0.75}{{#1}}}

\newcommand{\XRindex}{\hspace{-.5pt} \scaleobj{0.85}{X_{\! R}}}
\newcommand{\XRnorm[1]}{\lN {#1} \rN_{\XRindex}}

\newcommand{\tNormIndex}[1]{\hspace{-.5pt} \scaleobj{0.65}{{#1}}}
\newcommand{\NormIndex}[1]{\hspace{-.5pt} \scaleobj{0.87}{{#1}}}
\newcommand{\Norm}[2]{\lN {#1} \rN_{\NormIndex{{#2}}}}
\newcommand{\tNorm}[2]{\lN {#1} \rN_{\tNormIndex{{#2}}}}
\newcommand{\bigNorm}[2]{\big\lVert {#1} \big\rVert_{\NormIndex{{#2}}}}
\newcommand{\BigNorm}[2]{\Big\lVert {#1} \Big\rVert_{\NormIndex{{#2}}}}

\newcommand{\tCinterval}[2]{\scaleobj{0.72}{[{#1}, {#2}]}}

\newcommand{\lc}{\left\lbrack}
\newcommand{\rc}{\right\rbrack}

\newcommand{\lp}{\left(}
\def\rp{\right)}
\newcommand{\lb}{\left\{}
\newcommand{\rb}{\right\}}
\newcommand{\lN}{\left\lVert}
\newcommand{\rN}{\right\rVert}

\newcommand{\la}{\langle}
\newcommand{\ra}{\rangle}

\newcommand{\R}{\mathbb{R}}

\newcommand{\N}{\mathbb{N}}

\newcommand{\ep}{\varepsilon}

\newcommand{\nae}{\textrm{~a.e.~}}
\newcommand{\ie}{\textrm{~i.e.~}}
\newcommand{\st}{\textrm{~s.t.~}}

\newcommand{\mint}{\medint{\int}}
\newcommand{\miint}{\medint{\iint}}

\newcommand{\p}{\partial}
\newcommand{\mdvx}{\tfrac{\dd }{\dd{x}}}
\newcommand{\mpdvt}{\tfrac{\p }{\p t}}
\newcommand{\mpdvx}{\tfrac{\p }{\p x}}

\newcommand{\sC}{\mathscr{C}}

\newcommand{\cV}{\mathcal{V}}
\newcommand{\cS}{\mathcal{S}}
\newcommand{\cC}{\mathcal{C}}

\newcommand{\cG}{\mathcal{G}}
\newcommand{\cX}{\mathcal{X}}
\newcommand{\cB}{\mathcal{B}}
\newcommand{\cP}{\mathcal{P}}
\newcommand{\cI}{\mathcal{I}}
\newcommand{\cF}{\mathcal{F}}

\DeclareMathOperator*{\infess}{ess\,inf\,}
\DeclareMathOperator{\e}{e}
\DeclareMathOperator{\sgn}{sgn}
\DeclareMathOperator{\supp}{supp}

\hypersetup{%
  linktocpage=false, 
  breaklinks=true, 
  colorlinks=true, 
  urlcolor=SlateGray, 
  citecolor=Green, 
  linkcolor=Navy, 
  anchorcolor=SlateGray, 
  bookmarksopen=false, 
  pdftitle={rat_tournus_KRM_2023}, %
  pdfauthor={Anaïs RAT, Magali Tournus}, %
}

\title{Growth-fragmentation model for a population presenting
  heterogeneity in growth rate: Malthus parameter and long-time
  behavior}

\date{June 2, 2023}

\author{ Ana\"{i}s Rat\footnote{Aix Marseille Univ, CNRS, I2M,
    Centrale Marseille, Marseille, France.}$^{~,\hspace{-4pt}}$
  \footnote{Sorbonne Université, Inria, CNRS, Université de Paris,
    Laboratoire Jacques-Louis Lions, 4 place Jussieu, 75005 Paris, France.}
  \and Magali Tournus\footnotemark[1]}

\begin{document}

\maketitle

\blfootnote{\hspace{-7pt} Email:
  \href{mailto:anais.rat@centrale-marseille.fr}{anais.rat@centrale-marseille.fr},
  \href{mailto:mtournus@math.cnrs.fr}{mtournus@math.cnrs.fr}.}

\begin{abstract}
  The goal of the present paper is to explore the long-time behavior
  of the growth-fragmentation equation formulated in the case of equal
  mitosis and variability in growth rate, under fairly general
  assumptions on the coefficients.
  The first results concern the monotonicity of the Malthus parameter
  with respect to the coefficients.
  Existence of a solution to the associated eigenproblem is then
  stated in the case of a finite set of growth rates thanks to
  the Kre\u{\i}n-Rutman theorem and a series of estimates on
  the moments.
  Providing additionally enough mixing in the 
  population expressed in terms of irreducibility of the transition
  kernel, uniqueness of the eigenelements and asynchronous exponential
  growth are stated through entropy methods.
  Notably, convergence in shape to the steady state holds in the case
  of individual exponential growth where, in the absence of
  variability, the solution is known to exhibit oscillations at large
  times.
  We eventually perform a few numerical approximations to illustrate
  our results and discuss our mixing condition.
\end{abstract}

\begin{keywords} Growth-fragmentation equation, structured population,
  heterogeneity in growth rate, eigenproblem, Malthus parameter,
  long-time behavior, general relative entropy, Kre\u{\i}n-Rutman's
  theorem.
\end{keywords}

\begin{msc}
  35B40, 35Q92, 45C05, 45K05, 47A75, 92D25.
\end{msc}


\section{Introduction}

Among the class of structured population models,
\emph{growth-fragmentation equations} have raised throughout the last
decades a wide literature.
Their first formulation traces back to 1967 by three independent
groups of biophysicists, Bell and Anderson~\cite{bell_cell_1967},
Sinko and Streifer~\cite{sinko_new_1967}, and Fredrickson, Ramkrishna
and Tsuchiya~\cite{fredrickson_statistics_1967}.
Their need was to add/sub- stitute size, or several other
physiological variables~\cite{fredrickson_statistics_1967}, to the age
\smash{structuring~variable of} the already known renewal equation in
order to account for and learn from the finer~structure of the data
made available by technical progress (notably electronic \emph{Coulter
  counters} allowing for accurate cell counting and
sizing~\cite{gregg_electrical_1965, anderson_cell_1967}).
Yet, mathematical tools were insufficient to tackle the problem other
than numerically and it was not until the 1980s and the development of
semi-group theory 
that the first theoretical results were
obtained~\cite{diekmann_stability_1984, greiner_growth_1988,
  heijmans_stable_1984}.
Since then, growth-fragmentation equations
have been extensively studied under weaker assumptions
thanks to new tools from both deterministic and probabilistic
approaches.
Such interest is due to the variety of the domains of application and
to the mathematical complexity and richness of the questions raised
(well-posedness, long-time asymptotic behavior and coefficient
estimation through inverse problem, for the most classical ones).
For further information on growth-fragmentation equations we refer to
the book~\cite{metz_dynamics_1986} of Metz and Diekmann (1986) or the
review papers~\cite{arino_survey_1995, gyllenberg_mathematical_2007}
of Arino (1995) and Gyllenberg (2007).

\subsection{The classical growth-fragmentation model}

We consider a population of cells that grow in size according to some
\emph{deterministic growth rate} $\tau$ and divide with a certain
size-dependent \emph{probability per unit of time} $\gamma$. We assume
a constant environment so that coefficients $\tau$ and $\gamma$ only
depend on the individual variables, usually called \emph{structuring
  variables}.

In this subsection, the only structuring variable is the \emph{size}
$x$. We call it size although~$x$ can be volume, length, dry mass,
protein content, etc. as long as it is conserved through division.
Finally, we assume \emph{equal mitosis} considering that dividing
cells split into two cells of equal size. We obtain the classical
\emph{growth-fragmentation} equation
\begin{equation*} \label{pb:GFt}
  \tag*{\textnormal{(GF$_{\! t}$)}} \left\{
    \begin{aligned}
      & \mpdvt n(t,x)+ \mpdvx \big( \tau(x) n(t,x) \big)
        + \gamma(x) n(t,x) = 4 \gamma(2x)n(t,2x),\\
      & \tau(0) n(t,0) = 0, \qquad n(0,x) = n^{in}(x),
    \end{aligned}
  \right.\
\end{equation*}
where~$n(t,x)$ stands for the density of cells of size $x>0$ at time
$t\geq 0$. The $4$ in factor of the right-hand term (which stands for
new-born cells) is the product of two factors $2$ that arise from
getting \emph{two} cells with birth size in $(x,x+\dd{x})$ out of the
division of one cell with size in the interval $(2x, 2x + 2 \dd x)$
\emph{twice} as big.

Although non-conservative, this \emph{population balance equation}
comes from the combination of different conservation laws:
integrating~\ref{pb:GFt} against $1$ and $x$ respectively, one
gets~that
\begin{itemize}[leftmargin=.7cm,parsep=0cm,itemsep=0.1cm,topsep=0cm]
\item the \emph{total number of cells} is left unchanged by growth but
  increased by fragmentation
  \begin{equation*}
    \frac{\dd}{\dd{t}} \int_{0}^\infty n(t,x) \dd{x}
    = \int_{0}^\infty \gamma(x)n(t,x) \dd{x},
  \end{equation*}
\item and conversely, fragmentation conserves the \emph{total mass},
  besides increased by growth
  \begin{equation*}
    \frac{\dd}{\dd{t}} \int_{0}^\infty x n(t,x) \dd{x}
    = \int_{0}^{\infty} \tau(x) n(t,x) \dd{x}.
  \end{equation*}
\end{itemize}

A remarkable feature of many reproducing populations observed before
crowding or resource limitation occurs, therefore expected to be
captured by the model, is exponential growth coupled with
asynchronicity --or \emph{asynchronous exponential growth
  (A.E.G.)}~\cite{webb_operator-theoretic_1987}.

\emph{Asynchrony} is the property of a system to forget the shape of
its initial distribution at large times and asymptotically stabilize
in the sense that the proportion of individuals in a given
\emph{cohort} --a sub-population sharing same traits-- becomes
constant as time progresses.
Biologists also refer to it as the \emph{desynchronization
  effect}~\cite{chiorino_desynchronization_2001} since no matter how
synchronized the initial population is, i.e. how narrowly distributed
(Dirac distribution included, in the case of \emph{clonal
  populations}), it progressively desynchronizes as it aligns with
some stable distribution.

Mathematically, the A.E.G. corresponds to the existence of a
\emph{stationary profile} $N$, independent of the initial state, and
positive constants $C$, $\lambda$ such that:
\begin{equation*}
  n(t,x)\underset{t \rightarrow + \infty}{\sim} C \e^{\lambda t}N(x),
\end{equation*}
where the only memory of the initial state is a weighted average
contained in $C$. The asymptotic exponential growth rate $\lambda$ is
called the \emph{Malthus parameter} or sometimes \emph{fitness}.

We refer to Mischler and Scher's article~\cite{mischler_spectral_2016}
for a study and review of the long-time behavior.
A good balance between the growth and fragmentation rates is required
to ensure the A.E.G: if fragmentation dominates growth around size
zero the density goes to a Dirac distribution at size
zero~\cite{doumic_time_2016} and conversely, it dilutes to infinity if
growth dominates fragmentation at large sizes,
see~\cite{doumic_jauffret_eigenelements_2010} for some examples of
such non-existence of a steady~profile.

Another situation where the A.E.G fails, more surprisingly although
already anticipated in 1967 by Bell and
Anderson~\cite{bell_cell_1967}, is the case of ideal bacterial growth:
equal mitosis and linear growth rate $\tau (x) =vx$ for some positive
$v$ (or any growth rate s.t.  $2 \tau(x) = \tau(2x)$).
In this setting indeed, because the equation lacks dissipativity, it
loses its regularizing effect and keeps stronger memory of the initial
state. We still have exponential growth but convergence in shape
towards a time-periodic limit preserving the (otherwise vanishing)
singular part of measure solution. For instance, starting from a Dirac
mass in $x$, the distribution at any larger time $t$ is a sum of Dirac
masses supported on a subset of
$\{x \e^{v t} 2^{-n} \}_{n \in \N}$~\cite{gabriel_periodic_2022}.
An intuitive explanation is that under individual exponential growth,
cells with same size, no matter when they respectively divide, give
birth to cells of same size.
The property of having the same size is thus passed on through
generations together with the properties of the initial
population~\cite{bell_cell_1967, diekmann_stability_1984}.
See~\cite{greiner_growth_1988} for the first proof of convergence
(1988), most notably improved
in~\cite{brunt_cell_2018,doumic_explicit_2018,bernard_cyclic_2019},
or~\cite{gabriel_periodic_2022} for the latest results and a nice
review.

In his 1995 survey~\cite{arino_survey_1995}, Arino lists the
modifications of this limit case that have been made to re-establish
{the} A.E.G. With no surprise, it concerns one of the two
conditions --equal mitosis or $\tau(2x)=2\tau(x)$-- and somehow
consists in adding variability: allowing for
\begin{itemize}[leftmargin=.7cm,parsep=0cm,itemsep=0.1cm,topsep=0cm]
\item different sizes at birth, that is
  \emph{asymmetric}/\emph{unequal division}, one daughter cell
  inheriting a fraction $r \neq \frac{1}{2}$ of her mother's mass, the
  other one the remaining $1-r$, see Heijmans
  (1984)~\cite{heijmans_stable_1984} for the first proof,
\item different growth rates, subdividing the population into
  \emph{quiescent} non-growing (tumor) cells and \emph{proliferating}
  growing cells~\cite{gyllenberg_nonlinear_1990},
\end{itemize}
or, he suggests, even both combined which has only been studied very
recently by Cloez, de Saporta and Roget~\cite{cloez_long-time_2021}.

The first option has been well studied, especially through the larger
model of \emph{general fragmentation} that accounts for stochastic
mother/daughter size ratios at
division~\cite{bernard_asynchronous_2020, mischler_spectral_2016}, but
few studies have focused on other forms of variability. In the present
paper, motivated by biological evidence
(see~\cite{kiviet_stochasticity_2014, gangwe_nana_division-based_2018}
and the numerous references therein), we investigate the second
option: variability in growth rate within the size-structured setting,
more suited to cell division than the age-structured one (at least for
\emph{E. Coli} as indicated in~\cite{robert_division_2014}).

\subsection{A model accounting for cell-to-cell variability in growth
  rate}

Up to our knowledge, no theoretical study on the asymptotics of a size
structured model encompassing more than the two ways of growing
mentioned above is available.
A model {allowing a} continuous set of growth rates was formulated
through piecewise deterministic Markov branching tree by Doumic,
Hoffmann, Krell {and} Robert~\cite{doumic_statistical_2015} in order to
estimate the division rate more accurately.
The model was then used by Olivier~\cite{olivier_how_2017} to quantify
the variation of the Malthus parameter with respect to variability in
the aging and growth rates of age and size-structured populations,
respectively. Theoretical results are obtained in a deterministic
framework for the age-structured model, but again, no theoretical
study is carried out in the size-structured case, she relies on the
stochastic approach of Doumic et al. to get numerical results later
detailed in Subsection~\ref{ssec:monotonicity_V}.

To account for such cell-to-cell variability one needs to describe
cells, by their size still, but also by a new structuring variable: an
individual \emph{feature} or \emph{trait} $v$, attributed at birth and
conserved all life long, that determines cells' own growth rates
\smash{$\tau(v,\cdot)$}. It remains to define a rule for the
transmission of this feature over generations. Denoting by $\cV$ the
\emph{set of admissible features}, we introduce a kernel $\kappa$, the
\emph{variability kernel}, supported on $\cV \times \cV$ such that
{$\kappa(v, \dd v')$ prescribes the distribution of the features
  transmitted by cells of feature~$v$};
in particular it must {satisfy}
\begin{equation*}
  \int_\cV \kappa(v, \dd v') = 1, \quad \forall v \in \cV.
\end{equation*}

Within this new setting, the density $n(t,v,x)$ of cells of feature
$v \in \cV$ and size $x>0$ at time $t \geq 0$ evolves as
\begin{equation*}\label{pb:GFtv}
  \tag*{\textnormal{(GF$_{\! t,v}$)}}
  \left\{
    \begin{aligned}
      & \mpdvt n(t,v,x)+ \mpdvx \bigl( \tau(v,x) n(t,v,x) \bigr)
      = \cF (n ) (t,v,x),\\
      & \tau(v',0) n(t,v',0) = 0, \qquad n(0,v,x) = n^{in} (v,x),
    \end{aligned}
  \right.
\end{equation*}
with the fragmentation operator $\cF$ acting on a function $f$ through
\begin{equation*}
  \cF f (v,x) \coloneqq - \gamma(v,x) n(t,v,x) + 4 \mint_{\cV}
  \gamma(v',2x)n(t,v',2x) \kappa(v',v) \dd{v'}.
\end{equation*}

Let us clarify why the division rate depends on $v$ although
variability is only assumed on growth.
When assuming division to be only triggered by size, we actually aim
at defining a \emph{division rate per unit of size}, say $\beta$,
function of the size only such that the {probability $\beta(x)\dd x$,
  to divide before reaching size $x+ \dd x$ when having reached size
  $x$}, is common to all cells. A straightforward dimensional analysis
however indicates that the division rate $\gamma$ of the equation is a
\emph{rate per unit of time}.  In {the} absence of variability in
growth rate, $\gamma(x)\dd t$ is thus the probability for any cell
reaching size $x$ to divide in at most $\dd t$. Relating $\gamma$ and
$\beta$ is then easy given cells' growth rate: any cell of size $x$ at
time $t$, therefore instantaneously growing at speed $\tau(x)$, will
take a time \smash{$\dd{t} = \frac{\dd{x}}{\tau(x)}$} at most to grow
of $\dd x$ at most, {\Cref{fig:division_rate}}. This imposes the
natural relation $\gamma(x) \coloneqq \tau(x)\beta(x)$, as described
for example in~\cite{hall_steady_1991, diekmann_growth_1983}.
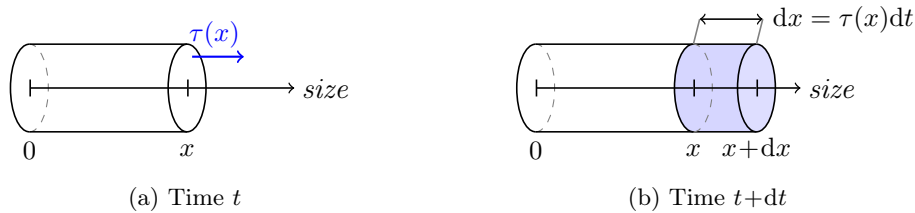
\begin{figure}[h!]
  \centering
  \begin{subfigure}{0.45\textwidth}
    \centering
    \begin{tikzpicture}[scale=.83]
      \draw[dashed,color=gray] (0,-.7) arc (-90:90:0.3 and .7); %
      \draw[semithick] (0,-.7) arc (270:90:0.3 and .7); 
      \draw[semithick] (0,-.7) -- (2.5,-.7); 
      \draw[semithick] (0,1.4-.7) -- (2.5,1.4-.7); 
      \draw[semithick] (2.5,0) ellipse (0.3 and 0.7); 
      \draw (4.7,0) node {\textit{size}}; %
      \draw[|-,semithick] (0,0) -- (2.5,0); %
      \draw[|->,semithick] (2.5,0) -- (4.2,0); %
      \draw (0,-1) node {\small$0$}; %
      \draw (2.5,-1) node {\small $x$}; %
      \draw[->,thick,color=blue] (2.58,.5) -- (3.4,.5); %
      \draw (3.5,.85) node[color=blue, left] {\small $\tau(x)$}; %
      \draw (3.5,-.96) node {\small \vphantom{$x \! +\! \dd{x}$}};
    \end{tikzpicture}
    \caption{Time $t$}
  \end{subfigure}
  \begin{subfigure}{0.45\textwidth}
    \centering \hspace{.2cm}
    \begin{tikzpicture}[scale=.83]
      \draw[dashed,color=gray] (0,-.7) arc (-90:90:0.3 and .7); %
      \fill[blue!16] (2.5,0) ellipse (0.3 and 0.7); 
      \fill[blue!16] (2.5,-.7) -- (3.5,-.7) -- (3.5,.7) -- (2.5,.7) --
      cycle; %
      \draw[semithick] (0,-.7) arc (270:90:0.3 and .7); 
      \draw[semithick] (0,-.7) -- (3.5,-.7); 
      \draw[semithick] (0,.7) -- (3.5,.7); 
      \draw[dashed,color=gray] (2.5,-.7) arc (-90:90:0.3 and .7); %
      \draw[semithick] (2.5,-.7) arc (270:90:0.3 and .7); 
      \draw[semithick,fill=blue!13] (3.5,0) ellipse (0.3 and 0.7); %
      \draw (4.7,0) node {\textit{size}}; %
      \draw[|-,semithick] (0,0) -- (2.5,0); \draw[|-,semithick]
      (2.5,0) -- (3.5,0); %
      \draw[|->,semithick] (3.5,0) -- (4.2,0); %
      \draw (0,-1) node {\small$0$}; %
      \draw (2.5,-1) node {\small$x$}; %
      \draw (3.5,-.96) node {\small $x \! +\! \dd x$}; %
      \draw[<->,semithick] (2.6,1.1) -- (3.6,1.1); %
      \draw[semithick,color=gray] (2.5,.7) -- (2.6,1.1); %
      \draw[semithick,color=gray] (3.5,.7) -- (3.6,1.1); %
      \draw (3.6,1.1) node[right] {\small $\dd{x} = \tau(x) \dd t$};
    \end{tikzpicture}
    \caption{Time $t \! +\! \dd{t}$}
  \end{subfigure}
  \caption{Scheme of the growth mechanism of a cell with growth rate
    $\tau = \tau(x)$.}
  \label{fig:division_rate}
\end{figure}

Returning to the model with variability, $\gamma$ should {depend} on
$v$ as well but only through~$\tau$, as a consequence of the relation
\smash{$\gamma \coloneqq \tau \beta$} ensuring
\smash{$\gamma(v, x) \dd t = \beta(x) \dd x$}. So one should not be
mistaken: only growth is subjected to variability, division keeps
being triggered by size only, in the sense that $\beta=\beta(x)$.
Note that this modeling assumption is made for the size/size-increment
equation (or equivalent) in both
biological~\cite{taheri-araghi_cell-size_2015} and
mathematical~\cite{hall_steady_1991, gabriel_steady_2019,
  doumic_estimating_2020} papers, and mimicked by
Oliver~\cite{olivier_how_2017} for the size-structured model with
variability.

It is clear that model~\ref{pb:GFt} can be formally obtained from
model~\ref{pb:GFtv} by
\begin{itemize}[leftmargin=.7cm,parsep=0cm,itemsep=0.1cm,topsep=0cm]
\item choosing an initial state composed of cells of a unique given
  feature $\Bar{v}$,
\item and imposing to every new-born cell to have feature $\Bar{v}$ by
  setting $\kappa(v, \cdot) = \delta_{\Bar{v}}$ for all $v$,
\end{itemize}  
after which the support of the distribution of features among cells,
initially reduced to $v=\Bar{v}$, will conserve through time since
there is no transport in feature. This way, we have
$n(t,\dd v,x) = \tilde{n}(t,x) \delta_{\Bar{v}}(\dd v)$ and finding a
solution $\tilde{n}$ to~model~\ref{pb:GFt} is equivalent to
finding a solution $n$ to this \enquote{reset}
model~\ref{pb:GFtv}.

Studying the long-time asymptotic behavior of~\ref{pb:GFtv}
requires to study its eigenproblem, whose solution can bee seen as a
stationary solution to the problem on $n \e^{\lambda t}$ (solution
to~\ref{pb:GFtv} rescaled by the exponential rate inherent to
its unrestricted growth):
\begin{equation*}
  \label{pb:GFv}
  \tag*{\textnormal{(GF$_{\! v}$)}}
  \left\{
    \begin{aligned}
      &\mpdvx \big( \tau N \big) + \lambda N = \cF N, \qquad
        \tau(\cdot,0)N(\cdot,0) = 0, \quad
        \mint_{\cV} \! \mint_0^\infty N = 1,
      \quad N \geq 0,\\
      &-\tau\mpdvx \phi + \lambda \phi = \cF^* \phi, \qquad
      \mint_{\cV}\! \mint_0^\infty N \phi = 1, \qquad \phi \geq 0,
    \end{aligned}
  \right.
\end{equation*}
with $\lambda$ the Malthus parameter previously mentioned and $\cF^*$
the dual operator of~$\cF$:
\begin{equation}\label{eq:SV_adjoint_opp}
  \cF^* h  (v,x) \coloneqq -\gamma (v,x) h (v,x) + 2\gamma(v,x)
  \mint_{\cV} \kappa(v,v') h \Big( v',\mfrac{x}{2} \Big) \dd{v'}.
\end{equation}

Note that the direct equation is defined in the weak sense: \emph{weak
  solutions} to the direct problem of~\ref{pb:GFv} are functions
$N \in L^1 ( \cV \times \R_+)$ such that for all
$\varphi \in \sC^\infty_c \big( \cV \times [0, +\infty) \big)$
\begin{multline*}
  \mint_{\cV} \! \mint_0^\infty \! N(v,x) \Big( - \tau(v,x) \mpdvx
  \varphi(v,x) + \big(\lambda + \gamma(v,x) \big) \varphi(v,x) \Big)
  \dd x \dd v \\
  = 4 \mint_{\cV} \! \mint_0^\infty \!  \Big( \mint_{\cV}\gamma(v',2x)
  N(v', 2x) \kappa(v',v) \dd{v'} \Big)\varphi(v,x) \, \dd v \dd x = 0.
\end{multline*}
As for the adjoint equation, it should hold a.e. for
$\phi~\in~L_1 \bigl(\cV; W_{loc}^{1,\infty} \lp 0, + \infty \rp
\bigr)$.

\subsection{Malthus parameter}

Characterizing the variation of the Malthus parameter $\lambda$ with
respect to (w.r.t.) a variation in the coefficients of the problem
is a question of both biological and mathematical interest. In the
context of adaptive dynamics, it arises as soon as one seeks to
question the optimality of the coefficients in the Darwinian sense of
maximizing the growth speed of the population --for example to explain
observations~\cite{michel_optimal_2006} or {to} optimize
therapeutic strategy like
cancer~\cite{clairambault_circadian_2006,billy_synchronisation_2014}
or antibiotic treatments~\cite{balaban_bacterial_2004}.
Such concerns also come to mind for anyone that wishes to gain insight
into the mathematical model, through a deeper understanding of the
interplay between the coefficients and the eigenvalue
{$\lambda$}~\cite{calvez_self-similarity_2012}. Though, relatively
few mathematical studies have addressed it (see
Subsection~\ref{ssec:monotonicity} for a brief review).

What interests us in the following, is the influence of monotonous
variations of the growth or division rate. A natural question is
whether or not the resulting variation of $\lambda$ is~monotonous. One
could think for example, that increasing the division rate should
increase the instantaneous number of newborn cells and in turn, the
exponential growth rate $\lambda$ of the overall population. No matter
how intuitive this can be it is not necessarily true as proved
in~\cite{calvez_self-similarity_2012}.

In this article, Calvez, Doumic and Gabriel, multiply by a factor
$\alpha$ the growth rate (or the division rate but the problem is
proved equivalent) and study the limit of the corresponding solution
$\lambda_\alpha$ when $\alpha$ goes to zero or infinity.  At the
limit, the problem becomes pure fragmentation, or pure transport
respectively: the mass tends to a Dirac {distribution at} zero (or
goes to infinity resp.) so, at the end, the only quantity that matters
to the asymptotic exponential growth speed of the population is the
value of the division rate at infinitely small (large resp.)
sizes. This formal reasoning is summarized by their {Theorem~1:}
$\lim_{\alpha \rightarrow L} \lambda_\alpha = \lim_{x \rightarrow L}
\gamma(x)$, for $L$ equal to $0$ or $+\infty$, which especially
implies the possible non-monotonicity of the Malthus parameter (for
example if $\gamma$, non-monotonous, cancels in zero and infinity).

The question then, which still remains open, is to find necessary and
sufficient conditions on the coefficients for the Malthus parameter to
depend {on them} in a monotonous way.

\subsection{Outline of the paper}

Our objective is to bring insight in the effect of variability between
individuals on the dynamics of the whole population: can such
cell-level heterogeneity be responsible for any population-level
strategy? in other words, is the population with variability growing
faster? for which families of growth rates, which transmission laws?
are, among others, related biological questions that we would like to
answer. From a mathematical point of view, how does variability
impacts the asymptotic behavior of the system? does it introduce
enough mixing between individuals to recover the A.E.G. property in
the case of individual exponential growth?
The main results of the present paper are the following:
\begin{itemize}[leftmargin=.7cm, parsep=0cm,itemsep=0.1cm,topsep=0cm]
  
\item \Cref{cor:monotonicity} (no variability) tells us that the
  modeling assumption $\gamma \coloneqq \tau \beta$ is sufficient to
  ensure monotonicity of the Malthus parameter with respect to the
  growth rate,
  
\item \Cref{thm:monotonicity_V} (variability) provides bounds for the
  Malthus parameter $\lambda$ solution to~\ref{pb:GFv}. The
  lower (upper) bound is the Malthus parameter $\lambda_1$
  ($\lambda_2$ resp.) associated with the problem with no variability
  and any growth rate $\tau_1 \leq \tau(v, \cdot)$ (resp.
  $\tau(v, \cdot) \leq \tau_2$),
  
\item \Cref{thm:existence_eig} states existence of eigenelements
  $(\lambda, N ,\phi)$ solution to~\ref{pb:GFv},
  
\item \Cref{prop:GRE_principle} ensures that~\ref{pb:GFtv}
  verifies a GRE inequality,
  
\item from which we get uniqueness of eigenelements
  (\Cref{thm:uniqueness_eig}) and long-time convergence of {the}
  renormalized solutions to~\ref{pb:GFtv} to the steady state
  $N$ in $L^1( \phi)$-norm (\Cref{thm:conv}). Proofs strongly rely on
  the GRE principle and the localization of the support of the
  eigenvectors (\Cref{prop:positivity}). An interesting consequence is
  that providing sufficient mixing in feature, encoded in a condition
  of irreducibility on the variability kernel $\kappa$, convergence to
  a steady profile holds in the special case of linear growth rate.
\end{itemize}

The paper is organized as follows. \Cref{sec:main_results} exposes our
set of assumptions and main results. Proofs are given all along the
section, except the longer proof of~\Cref{thm:existence_eig},
postponed to \Cref{sec:proof}.
\Cref{sec:illustration} illustrates the monotonicity and convergence
results in the case of linear growth rate, through theoretical and
numerical study of simple cases. Numerical approximations also allow
to study weaker mixing conditions than irreducibility.

\section{Main results}
\label{sec:main_results}

All along the paper, we adopt the following notation: 
\begin{equation*}
  \cS = \cV \times (0, +\infty),
\end{equation*}
introduce the space of functions $L^p$-integrable on a neighborhood of
zero
\begin{equation*}
  L_0^ p \coloneqq \bigl\{ f : \exists a>0, \;f \in L^p(0,a) \bigr\},
  \qquad 1 \leq p \leq +\infty,
\end{equation*}
and the set of non-negative functions with at most
polynomial growth or decay at infinity:
\begin{equation*}
  \cP_\infty \coloneqq \Big\{ f \geq 0, f \in
  L^\infty_{loc}(0, +\infty) 
  : \exists \nu, \omega \geq 0,
  ~~ \limsup_{x \rightarrow + \infty } \bigl( x^{-\nu} f(x) \bigr) < \infty,
  ~~ \liminf_{x \rightarrow + \infty } \bigl( x^{\omega} f(x) \bigr) >0 \Big\}.
\end{equation*}

\subsection{Assumptions}

We formulate the following hypotheses:
\begin{description}
\item[\namedlabel{as:V}{(H$\cV_{\mathscr{D}}$)}] The \emph{set of
    features} $\cV$ is finite:
  $\exists M \in \N^* : \cV = \{v_1,v_2, \dots, v_M\}$.
  
  To lighten notations we introduce $\, \cI \coloneqq \{1,\ldots, M\}$
  and use indifferently the notations~$f(v_i, \cdot)$ or $f_i$, for
  any $i$ in $\cI$, whether one considers $f$ as a map on $\cV$ to
  some functional space $X$ or as a vector in $X^M$. In particular we
  write equivalently \smash{$\ell^1\big(\cI; L^1(0, + \infty) \big)$}
  and $L^1(\cS)$.

\item[\namedlabel{as:tau}{(H$\tau$)}] The \emph{growth rate} $\tau$
  satisfies: there exist $\nu_0, \omega_0 \geq 0$, s.t. for all $v$ in
  $\cV$,
  \begin{description}[leftmargin=-.5cm, parsep=0cm, topsep=0cm]
  \item[\namedlabel{as:tau_positive}{(H$\tau_{_{pos}}$)}]
    $\forall K \subset (0, +\infty)$ compact,
    $\exists m_K >0 : \tau(v,\cdot) \geq m_K \,$ a.e. on
    $K$
    
  \item[\namedlabel{as:tau_in_0}{(H$\tau_{_0}$)}]
    \smash{$\frac{x^{\nu_0}}{\tau(v,\cdot)} \in L_0^1$} and
    \smash{$x^{\omega_0} \tau(v, \cdot) \in L_{0}^\infty$}

  \item[\namedlabel{as:tau_in_inf}{(H$\tau_{\infty}$)}]
    $\tau(v, \cdot) \in \cP_\infty$
  \end{description}

\item[\namedlabel{as:gamma}{(H$\gamma$)}] The \emph{division rate per
    unit of time} $\gamma$ satisfies: for all $v$ in $\cV$,
  $\gamma(v, \cdot) \in \cP_\infty$.
  
\item[\namedlabel{as:tau-gamma}{(H$\gamma$\textnormal{-}$\tau$)}]
  \vspace{.1cm} There exists a function $\beta$, depending only on the
  size variable and satisfying~\ref{as:beta}, such that
  $\gamma(v, x) = \beta (x) \tau(v,x)$,
  $~\forall v \in \cV, ~\nae x \in (0, +\infty)$.

\item[\namedlabel{as:beta}{(H$\beta$)}] The fragmentation/growth ratio
  $\beta$, or \emph{division rate per unit of size}, satisfies
  \begin{description}[leftmargin=1cm, parsep=0cm, topsep=0cm]
    
  \item[\namedlabel{as:beta_support}{(H$\beta_{\supp}$)}]
    $\beta$ is supported on $[b, + \infty)$ for some $b \geq 0$

  \item[\namedlabel{as:beta_L0}{(H$\beta_{_0}$)}]
    \smash{$\beta \in L^1_0$}
   
  \item[\namedlabel{as:beta_in_inf}{(H$\beta_{\infty}$)}]
    \smash{$\lim\limits_{x \rightarrow + \infty} x \beta(x)=+\infty$}
  \end{description}

\item[\namedlabel{as:kappa}{(H$\kappa$)}] Under~\ref{as:V}, the
  \emph{variability kernel} $\kappa : \cV \times \cV \rightarrow \R_+$
  can be seen as a square matrix of size~$M$, in which case
  $\kappa = (\kappa_{ij}) \in \mathcal{M}_{\texp{M}}(\R_+)$ satisfies
  \begin{description}[leftmargin=1cm, parsep=0cm, topsep=0cm]
  \item[\namedlabel{as:kappa_cv}{(H$\kappa_{prob}$)}]
    $\kappa$ is a \emph{stochastic (or probability or transition)
      matrix}:
    \begin{equation*}
      \kappa_{ij} \geq 0,~~\forall (i,j) \in \cI^2,
      \qquad \; \mmsum{j=1}{M} \kappa_{ij}= 1,~~\forall i \in \cI.
    \end{equation*} 

  \item[\namedlabel{as:kappa_irr}{(H$\kappa_{irr}$)}]
    $\kappa$ is \emph{irreducible}:
    $\forall (i,j) \in \cI^2, ~ \exists m \in \N : \kappa^{(m)}_{ij}
    \coloneqq ( \kappa^m )_{ij} > 0.$
  \end{description}

\item[\namedlabel{as:tau_hetero}{(H$\tau_{var}$)}]
  $\forall (v',v) \in \supp(\kappa) : v' \neq v \; \,
  {\Longrightarrow} \; \; 2 \tau(v,x) \neq \tau(v',2x) \;\nae x \in
  \big( \tfrac{b}{2}, +\infty \big)$.
\end{description}

A few remarks on these assumptions:
\begin{itemize}[leftmargin=0cm, wide=0cm, topsep=0cm]
\item \textit{Set of features.} Taking a finite set of
  features~\ref{as:V} is sufficient to model most biological
  situations where variability manifests as clear physiological
  differences between distinct subpopulations (see
  e.g.~\cite{cloez_long-time_2021} which accounts for the difference
  in growth rates between the \emph{old pole} and \emph{new pole
    cells}, or~\cite{balaban_bacterial_2004} which studies
  \emph{persistent} cells, resistant to antibiotic treatments but
  having reduced growth rate contrary to the rest of the population).
  Likewise, one can resume from the continuous vision of variability,
  of the type of~\cite{doumic_statistical_2015, olivier_how_2017}
  \begin{description} \centering
  \item[\namedlabel{as:V_continuous}{(H$\cV_{\sC}$)}] ~~$\cV$ is a
    compact interval of $(0,+\infty)$,
  \end{description}
  to the discrete one, simply by partitioning $\cV$ in characteristic
  intervals $\cV_i$ and considering any cell with feature in $\cV_i$
  as a cell with feature $\bar{v}_i$, a certain average on~$\cV_i$.

\item \textit{Growth and fragmentation rates}.
  Assumption~\ref{as:beta_L0} together with \ref{as:tau_positive},
  \ref{as:tau-gamma} and~\ref{as:gamma} implies that $\beta$ is
  $L^1_{loc}(\R_+)$.
  Hypotheses~\ref{as:tau},~\ref{as:gamma} and~\ref{as:beta} altogether
  is nothing but having, for any $v$ in $\cV$, that $\tau(v,\cdot)$,
  $\gamma(v,\cdot)$ and their ratio satisfy the assumptions made by
  Doumic and Gabriel in~\cite{doumic_jauffret_eigenelements_2010} (to
  ensure existence of a solution to~\ref{pb:GF}) with the
  additional requirement that $\gamma$ is $L^\infty(\R_+^*)$.
  Under~\ref{as:V}, this straightforward adaptation of their
  assumptions is sufficient to state existence but we prove it under
  the additional modeling assumption~\ref{as:tau-gamma} which brings
  small simplifications.
  
  The additional assumption~\ref{as:tau_hetero} is crucial to
  characterize the functions canceling the dissipation of entropy,
  needed to establish uniqueness of eigenelements and convergence. It
  can be interpreted as follows: every time a mother of feature $v$
  gives birth to cells with same features $v' \neq v$, the daughters
  cannot gain together the exact mass that their mother would have
  gained if it had not divided; or equivalently from a macroscopic
  point of view: the mass that would have grown the fraction
  $\kappa(v', v)$ of all dividing cells of feature $v'$ cannot be the
  total mass gained by their daughters born with feature~$v \neq v'$.
  
\item \textit{Variability kernel.}  For $m > 0$,
  \smash{$\kappa^{(m)}_{ij}$} represents the probability for a cell
  descending from $m$ {generations} of a cell with feature $v_i$
  to have feature~$v_j$. Therefore~\ref{as:kappa_irr} means that any
  cell has a non-zero probability to transmit any feature in a finite
  number of division ($M$ at most).
  Besides, $\kappa$ can be associated to the directed graph having
  $\{v_1, \ldots v_M\}$ as vertices and the pairs $(v_i, v_j)$ such
  that $\kappa_{ij} >0$ as oriented edges. Then,~\ref{as:kappa_irr} is
  equivalent to the \emph{graph of $\kappa$} being \emph{strongly
    connected} (every vertex is reachable from every other vertex).

  The first assumption on $\kappa$ easily translates to the continuous
  setting~\ref{as:V_continuous} (this is $\kappa$ being a
  \emph{Markov} or \emph{probability} \emph{kernel}\footnote{
    $\kappa(v, \cdot)$ is a probability measure for every $v$ (i.e.
    $v \mapsto \int_\cV \varphi(v') \kappa(v, \dd{v'}) $, for
    $\varphi \in \sC(\cV)$, is Lebesgue measurable and
    $\int_\cV \kappa(v, \dd{v'}) =1$, $v \in \cV$) and
    $v \mapsto \kappa(v, A)$ is a measurable function for all
    $A \in \mathcal{B}(\cV)$.} as assumed
  in~\cite{doumic_statistical_2015, olivier_how_2017}). It is however
  more complicated to generalize~\ref{as:kappa_irr}, crucial to
  establish the positivity result of \Cref{prop:positivity}, to the
  continuous setting (see~\cite{meyn_markov_1993} for a
  generalization). Additionally, irreducibility alone could not be
  appropriate to extend our results to larger $\cV$: for countable
  infinite sets $\cV$, we might additionally need $\kappa$
  \emph{recurrent}~\cite{norris_markov_1997}; as for continuous sets,
  irreducibility could be less relevant as suggest the recent
  articles~\cite{cloez_irreducibility_2020,
    bansaye_non-conservative_2022}. Their probabilistic approach via
  Harris's theorem proves particularly suited to obtain existence of a
  steady state and long-time convergence at exponential rate in the
  setting of measure solutions; see also~\cite{canizo_spectral_2021}
  for a good introduction to the Harris theorem and its application to
  growth-fragmentation in different settings including equal mitosis.
\end{itemize}

\subsection{Variation of the Malthus parameter with respect to
  coefficients}
\label{ssec:monotonicity}

\subsubsection{In the absence of variability.}
\label{ssec:monotonicity_no_V}

When there is no variability in growth rate, the eigenproblem
associated to the Cauchy problem~\ref{pb:GFt} writes as
follows: for $x>0$,
\begin{equation*}\label{pb:GF}
  \tag*{\textnormal{(GF)}}
  \left\{
    \begin{aligned}
      &\tfrac{\dd}{\dd{x}} \big( \tau(x) N(x) \big)
      + \big( \lambda + \gamma (x) \big) N(x)
      = 4 \gamma(2x) N(2x), \quad \tau(0) N(0) = 0,
      \quad \mint N = 1, \quad  N \geq 0,\\
      &\!-\!  \tau(x) \tfrac{\dd}{\dd{x}} \phi(x)
      + \big( \lambda + \gamma (x) \big) \phi(x)
      = 2 \gamma (x) \phi \Big(\mfrac{x}{2} \Big),
      \qquad \mint N \phi = 1, \qquad \phi \geq 0,
    \end{aligned}
  \right.
\end{equation*}
and the only coefficients of the problem are $\tau$ and $\gamma$. In
the general fragmentation model, note that another coefficient is the
fragmentation kernel (that is
$\frac{1}{2} ( \delta_{x=ry} + \delta_{x= (1-r)y} )$ in the case of
\emph{general mitosis}, with \smash{$r =\frac{1}{2}$} in our case of
equal mitosis).

Among existing studies on the variation of $\lambda$ w.r.t variations
of the coefficients, Michel~\cite{michel_optimal_2006} focuses on the
influence of the asymmetry in birth size between daughters (given by
the parameter~$r$ just mentioned) from two different points of view
--either by differentiation of $\lambda$ with respect to the asymmetry
parameter $r$ or by a min-max principle providing a handful expression
of $\lambda$. Monotonicity of $\lambda$ w.r.t $r$ is proved in the
case of constant growth rate and compactly supported division rate
that satisfies conditions proved sufficient for the adjoint vector to
be concave or convex. In particular, depending on the form of the
division rate, symmetric division can be detrimental to the growth of
the overall population.

A corpus of articles by Clairambault et al. investigates the influence
on $\lambda$ of the time periodic dependence of coefficients (division
and death rates) induced by the circadian rhythm. However, they rather
consider the renewal (age-structured)
equation~\cite{gaubert_discrete_2015} or a system of renewal equations
accounting for different phases of the cell
cycle~\cite{clairambault_circadian_2006,billy_synchronisation_2014}.

Campillo, Champagnat and Fritsch~\cite{campillo_variations_2017}
provide sufficient conditions on $\tau$, $\gamma$ and $\gamma / \tau$
for the monotonicity of $\lambda$ w.r.t. monotonous variations of the
coefficients. The result is derived in a probabilistic framework by
coupling techniques and eventually extended to the deterministic
framework thanks to a relation between the survival probability of the
stochastic model and the Malthus parameter solution to the
deterministic problem.

With only deterministic tools and straightforward arguments, we
propose another set of conditions.
\begin{lemma}\label{thm:monotonicity}
  Set $\cV \coloneqq \{v_1, v_2 \}$ and take $\tau$, $\gamma$
  satisfying ~\ref{as:tau},~\ref{as:gamma} such that
  $\frac{\tau(v_i, \cdot)}{\gamma(v_i, \cdot)}$
  satisfies~\ref{as:beta} for $i \in \{1,2 \}$. Denote by
  $(\lambda_{i}, N_i, \phi_i)$, $\lambda_i >0$, the weak solution to
  the eigenproblem~\ref{pb:GF} with coefficients
  $\tau_i\coloneqq \tau(v_i,\cdot)$ and
  $\gamma_i\coloneqq \gamma (v_i ,\cdot)$.  Assume the following
  \begin{equation*}
    \textit{i)~} \tau_1 \leq \tau_2,
    \qquad \textit{ii)~} \Big (\mfrac{\gamma_2}{\tau_2}
    - \mfrac{\gamma_1}{\tau_1} \Big) \Big( \mfrac{1}{2} \phi_i
    - \phi_i \big( \mfrac{\cdot}{2} \big) \Big) \leq 0
    \quad i=1 \text{~or~} 2,
    \qquad \text{a.e. on~~} (0, +\infty),
  \end{equation*}
  with notation
  $\phi \big( \mfrac{ \cdot}{2} \big) : x \mapsto \phi \big(
  \mfrac{x}{2} \big)$. Then we have
  \begin{equation*}
    \lambda_{1} \leq \lambda_{2}.
  \end{equation*}
\end{lemma}

\begin{proof} The proof relies on the duality relation between
  eigenvectors. Existence and uniqueness of positive eigenelements is
  guaranteed by~\cite[Theorem
  1]{doumic_jauffret_eigenelements_2010}. To simplify notations we
  introduce the bracket $\la f,g \ra \coloneqq \int_0^\infty f
  g$. Without loss of generality, assume that \textit{ii)} holds for
  $i=2$ (otherwise we can still swap 1 and 2, and reverse the
  inequalities in the following). Using equation~\ref{pb:GF} on
  $N_1$ and $\phi_2$ successively, we have:
  \begin{equation*}
    \begin{aligned}
      \lambda_1 \big\la N_1 , \phi_2 \big\ra
      &= \Big\la -\tfrac{\dd}{\dd{x}} \big( \tau_1 N_1 \big)
      -\gamma_1 N_1 + 4 \gamma_1 (2 \cdot) N_1 (2 \cdot), \phi_2 \Big\ra\\
      &= \Big\la N_1, \tau_1\tfrac{\dd}{\dd{x}} \phi_2
      - \gamma_1 \phi_2 + 2 \gamma_1 \phi_2
      \big( \mfrac{ \cdot}{2} \big) \Big\ra\\
      &= \Big\la N_1, \mfrac{\tau_1}{\tau_2} \Big( \lambda_2\phi_2
      + \gamma_2 \phi_2
      - 2 \gamma_2 \phi_2 \big( \mfrac{ \cdot}{2} \big) \Big)
      - \gamma_1 \phi_2
      + 2 \gamma_1 \phi_2 \big( \mfrac{\cdot}{2} \big) \Big\ra\\
      &= \lambda_2 \Big\la N_1 , \mfrac{\tau_1}{\tau_2} \phi_2 \Big\ra
      + \Big\la \tau_1 N_1, \Big( \mfrac{\gamma_2}{\tau_2}
      - \mfrac{\gamma_1}{\tau_1} \Big)
      \Big( \phi_2 - 2 \phi_2 \big( \mfrac{ \cdot}{2} \big) \Big) \Big\ra,
    \end{aligned}
  \end{equation*}
  leading, thanks to \textit{i)}, to the following inequality
  \begin{equation*}
    (\lambda_1-\lambda_2) \big\la N_1 , \phi_2 \big\ra
    \leq \Big\la \tau_1 N_1, \Big( \mfrac{\gamma_2}{\tau_2}
    - \mfrac{\gamma_1}{\tau_1} \Big) \Big( \phi_2 - 2 \phi_2
    \big( \mfrac{ \cdot}{2} \big) \Big) \Big\ra,
  \end{equation*}
  which makes clear that $\lambda_1 \leq \lambda_2$ as soon as
  \textit{ii)} is verified.
\end{proof}

In particular when~\ref{as:tau-gamma} is satisfied, \textit{ii)} is an
equality and \Cref{thm:monotonicity} directly brings monotonicity of
the Malthus parameter with respect to the growth rate as stated below.
\begin{theorem}[Monotonicity of the Malthus
  parameter]\label{cor:monotonicity}
  Let us take notations of \Cref{thm:monotonicity} and make the same
  assumptions on $\tau$ and $\gamma$. Assume also
  that~\ref{as:tau-gamma} holds, then
  \begin{equation*}
    \tau_1 \leq \tau_2 \text{~~a.e. on~~} (0, +\infty)
    \quad \Longrightarrow  \quad \lambda_1 \leq \lambda_2.
  \end{equation*}
  If in addition $\tau_1 \not\equiv \tau_2$, then the inequality
    on the Malthus parameters is strict.
\end{theorem}

It should be noticed that another set of assumptions can be derived,
namely:
\begin{equation*}
  \textit{i)~} \gamma_1 \leq \gamma_2,
  \qquad \textit{ii)~} \Big( \mfrac{\tau_1}{\tau_2}
  - \mfrac{\gamma_1}{\gamma_2} \Big) \tfrac{\dd}{\dd{x}} \phi_2
  \leq 0 \quad i=1 \textit{~or~} 2,
  \qquad \textit{a.e. on~~} (0, +\infty),
\end{equation*}
since the third equality in the proof can be replaced by
\begin{equation*}
  \lambda_1 \big\la N_1 , \phi_2 \big\ra
  = \lambda_2 \Big\la N_1 , \mfrac{\gamma_1}{\gamma_2} \phi_2 \Big\ra
  + \Big\la  N_1, \Big( \mfrac{\tau_1}{\tau_2} - 
  \mfrac{\gamma_1}{\gamma_2} \Big)
  \tau_2 \tfrac{\dd}{\dd{x}} \phi_2  \Big\ra.
\end{equation*}

We make a few remarks on \Cref{thm:monotonicity} before
continuing on the problem with variability.
\begin{itemize}[leftmargin=.7cm,parsep=0cm,itemsep=0.1cm,topsep=0cm]

\item \Cref{thm:monotonicity} and \Cref{cor:monotonicity} remain valid
  for general fragmentation kernels verifying the general assumptions
  of~\cite{doumic_jauffret_eigenelements_2010}.

\item Assumption~\textit{ii)} depends on the adjoint eigenvector which
  makes it hardly interpretable. Although not very satisfying, we kept
  it for several reasons. First, it needs to be satisfied for only one
  $i$, therefore in practice one can obtain monotonicity between a
  problem with explicit solution $(\lambda, N, \phi)$ and another
  problem, with no explicit solution but lower (or greater) growth
  rate and greater (lower resp.) fragmentation/growth ratio wherever
  \smash{$ \phi \big( \frac{ \cdot}{2} \big) \leq \frac{1}{2} \phi$}
  --or simply lower (or greater) growth rate when
  \smash{$\phi \big( \frac{ \cdot}{2} \big) = \frac{1}{2} \phi$}
  (which is true for most of the known solutions, see e.g. the
  explicit solutions given
  in~\cite{doumic_jauffret_eigenelements_2010} for uniform
  fragmentation, or in~\cite{perthame_exponential_2005,
    bernard_cyclic_2019} for mitosis).
  
  Second, it is interesting to see that we retrieve conditions similar
  to those of~\cite{campillo_variations_2017}, but through a different
  and much quicker approach allowing for more general
  coefficients. Assuming $\tau_1 \leq \tau_2$ and
  {$\frac{\gamma_2}{\tau_2} \leq \frac{\gamma_1}{\tau_1}$} as
  in~\cite{campillo_variations_2017}, we require
  \smash{$ \phi_i \big( \frac{ \cdot}{2} \big) \leq \frac{1}{2}
    \phi_i$} where they ask for $\gamma$ to be monotonous not only in
  the $v$ variable but also in the $x$ variable.

\item For our lemma to be interpretable in terms of coefficients only,
  the problem comes down to finding assumptions on the coefficients
  allowing to characterize the sign of $\frac{\dd}{\dd{x}} \phi$ (or
  $\phi \big( \frac{ \cdot}{2} \big) -\frac{1}{2} \phi$, implied
  because of~\ref{pb:GF}) in the spirit
  of~\cite{michel_optimal_2006} that derives conditions on $\gamma$
  for $\phi$ to be concave or convex.
\end{itemize}

\subsubsection{In {the} presence of variability.}
\label{ssec:monotonicity_V}

When it comes to the question of the impact of variability on the
Malthus parameter even less studies are available although attracting
a lot of attention from biologists, see~\cite{levien_non-genetic_2021}
for a recent review in the biological community.

Among the most notable studies, Cloez, de Saporta and
Roget~\cite{cloez_long-time_2021} add to the morphological asymmetry
between daughters studied by Michel~\cite{michel_optimal_2006} a
{physiological} asymmetry, attributing different growth rates to
daughters depending on whether or not they have bigger size. In
particular, in the case $\tau$ linear (and
\smash{$\lim_{x \rightarrow L} \gamma(x)=L$}, for $L= 0$ or $+\infty$)
they explicitly compute the partial derivatives of $\lambda$ w.r.t.
the asymmetry parameters and find that in {the} presence of
morphological asymmetry, physiological asymmetry maximizes $\lambda$
if largest cell growth faster.

Rather than two different growth rates in the whole population,
Olivier~\cite{olivier_how_2017} allows for a continuous set of growth
rates but equal mitosis, this is model~\ref{pb:GFtv}
under~\ref{as:V_continuous}. Her numerical results indicate that
variability in growth rate do not favor the global growth of the
population --more precisely she exhibits, under specific coefficients,
monotonicity of the Malthus parameter w.r.t. the coefficient of
variation (CV) of growth rates-- in short, w.r.t a parameter $\alpha$
characterizing the level of variability in the sense that the set of
features $\cV_\alpha = \alpha \cV + (1-\alpha) \bar{v}$ is more or less
spread out depending on~$\alpha$.

Relying on the ideas used to treat the case with no variability we can
compare the eigenvalue of problem~\ref{pb:GFv} to eigenvalues
solution to the problem with no variability.

\begin{theorem}[Monotonicity of the Malthus parameter with
  variability]\label{thm:monotonicity_V}
  Consider $(\lambda, N ,\phi)$, with $\lambda>0$, a weak solution
  to~\ref{pb:GFv} {with kernel $\kappa$
    satisfying~\ref{as:kappa_cv}}, and coefficient~$\tau$ and $\gamma$
  satisfying 
  \ref{as:tau-gamma} {s.t. their ratio $\beta=\beta(x)$ satisfies
    \ref{as:beta}}. Let $\tau_1$ and~$\tau_2$ satisfy~\ref{as:tau} and
  be such that
  \begin{equation*}
    \tau_1 \leq \tau( v, \cdot) \leq \tau_2
    \qquad \text{a.e. on~} (0,+\infty),
    \qquad \forall v \in \cV,
  \end{equation*}
  and define $\gamma_i \coloneqq \beta \tau_ i$, for $i=1, 2 $. Assume
  that $ \gamma_i$ satisfies \ref{as:gamma} and denote by
  $(\lambda_i, N_i, \phi_i)$, the solution to the
  eigenproblem~\ref{pb:GF} with no variability and coefficients
  $\tau_i$ and $\gamma_i$. Then
  \begin{equation*}
    \lambda_1 \leq \lambda \leq \lambda_2,
  \end{equation*}
  {with strict inequality if in addition
    $\tau(v, \cdot) \not\equiv \tau_1 , \tau_2$ for some $v \in \cV$}.
\end{theorem}

\begin{proof} Proceeding similarly to the case with no variability and
  but with duality bracket
  $\la f, g \ra = \int \! \! \! \int_{\cS} fg$, and using
  \ref{as:kappa_cv}, we find:
  \begin{equation}\label{eq:proof_monotonicity_V}
    \begin{aligned}
      \lambda \big\la N, \phi_i \big\ra
      &= \lambda_i \Big\la N, \mfrac{\tau}{\tau_i} \phi_i \Big\ra
      + \Big\la \tau N,
      \Big( \mfrac{\gamma_i}{\tau_i} - \mfrac{\gamma}{\tau} \Big)
      \Big(\phi_i- 2 \phi_i \big( \mfrac{\cdot}{2} \big) \Big)\Big\ra,
      \quad i \in \{1, 2 \}.
    \end{aligned}
  \end{equation}
  From~\ref{as:tau-gamma} and the definition of the $\gamma_i$, the
  second term of the right-hand side cancels and the inequality on
  $\lambda$ is a direct consequence of the inequality on $\tau$.
\end{proof}

In particular, a population where~\ref{as:tau-gamma} holds and $\tau$
is monotonous in $v \in [v_{min}, v_{max}]$, asymptotically grows
faster, respectively slower, than the population where all cells have
feature $v_{min}$ or $v_{max}$, respectively.

Note that \Cref{thm:monotonicity_V}, which assumes existence of a
solution to~\ref{pb:GFv}, holds true in the continuous
formulation~\ref{as:V_continuous}. Besides, a finer condition of the
type of \Cref{thm:monotonicity} can be formulated
from~\eqref{eq:proof_monotonicity_V}.

\subsection{Eigenvalue problem}

In the following, we elaborate on previously established results
adapted to our case to prove existence of a solution
to~\ref{pb:GFv} and GRE related results
(Subsections~\ref{ssec:GRE} and~\ref{ssec:conv}).

\begin{theorem}[Existence of eigenelements]\label{thm:existence_eig}
  Under assumptions~\ref{as:V},~\ref{as:tau},
  \ref{as:gamma},\ref{as:tau-gamma}, \ref{as:beta}
  and~\ref{as:kappa_cv}, there exists a weak solution
  $(\lambda,N,\phi)$ to the eigenproblem~\ref{pb:GFv} with
  $\lambda >0$, and for every $i \in \cI$ we have:
  \begin{gather*}
    x^\alpha \tau_i N_i \in L^p(\R_+),
    \quad {\forall \alpha \in \R,}
    \quad 1 \leq p \leq + \infty,
    \qquad x^\alpha \tau_i N_i \in W^{1,1}(\R_+),
    \quad \forall \alpha \geq  0,\\[3pt]
    \exists k >0 : \mfrac{\phi_i}{1+x^k} \in L^\infty(\R_+), \qquad
    \tau_i \mpdvx \phi_i \in L_{loc}^\infty(\R_+).
  \end{gather*}
\end{theorem}
In particular, note that $\tau_i N_i$ and $\phi_i$ are continuous on
$(0, +\infty)$.

The proof is postponed to \Cref{sec:proof}. It follows the proof
of~\cite[Theorem 1]{doumic_jauffret_eigenelements_2010} in which Doumic
and Gabriel state, under similar assumptions on the coefficients, the
existence of eigenelements to the problem with no variability but
general fragmentation. As
in~\cite{doumic_jauffret_eigenelements_2010}, we can also derive a
priori results on the localization of the support of the eigenvectors.

\begin{prop}[Positivity]\label{prop:positivity}
  {Consider $(\lambda, N, \phi)$ a weak solution
    to~\ref{pb:GFv} with $\lambda >0$.}
  Assume~\ref{as:V}, \ref{as:tau_positive}, \ref{as:tau-gamma},
  \ref{as:beta_support} and~\ref{as:kappa_irr} satisfied. Then, for
  all $i \in \cI$
  \begin{gather*}
    \supp N_i= \big[\tfrac{b}{2}, +\infty \big), \qquad \tau_i(x) N_i
    (x) >0,
    \quad \forall x \in \big(\tfrac{b}{2}, +\infty \big),\\
    \phi_i(x) >0, \quad \forall x > 0,
  \end{gather*}
  with $b$ defined by~\ref{as:beta_support}. Additionally, for any
  $j \in \cI$ such that $\frac{1}{\tau_j} \in L^1_0$ we have
  \begin{equation*}
    \phi_j (0) >0.
  \end{equation*}
\end{prop}

Note that the support of the eigenvectors being of the form
$V \times X$, for $V$ and $X$ subsets of $\cV$ and $[0, + \infty)$,
respectively, is a consequence of the support of $\tau_i$ and
$\gamma_i$ being independent of $i$ (see~\ref{as:tau_positive}
and~\ref{as:tau-gamma}). The support of the direct eigenvector $N$ is
quite natural considering the following remarks.
\begin{itemize}[leftmargin=.7cm,parsep=0cm,itemsep=0.1cm,topsep=0cm]
\item \textit{Along the $x$-axis}. Under equal mitosis the smallest
  size a new-born cell can have, regardless its feature,
  is~$\frac{b}{2}$ since $b$ is the smallest size at which division
  can occur. Besides if $b>0$, any cell of feature $v_i$ and size
  $x < b$ (and $x>0$ otherwise it stays with size $x=0$ from the
  boundary condition $\tau_i(0) n_i(t,0) =0$) reaches size $b$ in
  finite time (namely
  \smash{$\int_x^{b} \scaleobj{.88}{\frac{1}{\tau_i} }$}) so all
  individuals will have at least size $\frac{b}{2}$ after some time
  (or keep size $0$). Likewise, the probability
  \smash{$\e^{-\int_s^x \beta(y) \dd{y}}$} to survive from size $s$ up
  to size~$x$ without dividing is positive for all $x$ so we expect
  arbitrarily large cells to be found in the population at large
  times.

\item \textit{Along the $v$-axis.} Because of the irreducibility
  condition~\ref{as:kappa_irr} any cell can transmit any feature after
  a finite number of divisions, so the population should occupy the
  whole set of features (and stay in it) in finite time.
\end{itemize}

As for the adjoint vector, the integrability condition on some
$\tau_j$ for $\phi_j(0)>0$ can be explained interpreting the equation
on $\psi$ (see~\ref{pb:GFtv_ad} later) as a backward equation
(where in a sense cells are shrinking and aggregating over
time). Then, we expect to find cells of feature $v_j$ reaching size 0
at large times, meaning $\phi_j(0)>0$, if and only if the time
\smash{$\int_0^x \! \frac{1}{\tau_j} \dd{x}$} to reach size $0$ from
size $x>0$ (in {the} absence of \enquote{aggregation}) is finite
--that is $\frac{1}{\tau_j} \in L_0^1$.

\begin{proof} We introduce for all $i$ in $\cI$, the sets
  \begin{equation}\label{def:Iarrow}
    I_{_\leftarrow}(i) \coloneqq \bigl\{ j \in \cI : \kappa_{ji} >0
    \bigr\}, \qquad
    I_{_\rightarrow}(i) \coloneqq \bigl\{ j \in \cI : \kappa_{ij} >0
    \bigr\}, 
  \end{equation}
  of the indices of all the features of cells that can have a daughter
  or a mother, respectively, with feature $v_i$. Note, that for every
  $i \in \cI$, none of them is empty thanks to~\ref{as:kappa_irr}.

  \textit{Direct eigenvector.} Integrating on $[0,x]$ the equation
  of~\ref{pb:GFv} satisfied by $N\geq 0$ brings: for all $i \in \cI$,
  for {all} $x$ in $(0, +\infty)$,
  \begin{equation}\label{eq:proof_positivity}
    \tau_i(x) N_i(x)
    \leq 2 \mint_0^{2x} \beta (s) \msum{j \in \cI}
    \big( \tau_j(s) N_j(s) \kappa_{ji}  \big) \dd{s},
  \end{equation}
  which makes clear that $\tau N \geq 0$ cancels on
  \smash{$\cV \times \big[0, \frac{b}{2} \big]$} since $\beta$ cancels
  on $[0,b]$ (see~\ref{as:beta_support}). It remains to prove that
  $\tau N$ is positive elsewhere. For some positive $x_0$ we define
  \begin{equation*}
    F_i(x) \coloneqq \tau_i(x) N_i(x) \e^{\Lambda_i(x)},
    \quad \Lambda_i(x) \coloneqq \mint_{x_0}^x \mfrac{\lambda
      + \gamma_i(s)}{\tau_i(s)} \dd{s},
    \quad \forall i \in \cI, \quad {\forall} x \in [x_0,+\infty ),
  \end{equation*}
  and obtain, through equation~\ref{pb:GFv} again, that for all
  $i$ and for {all} $x$ in $[x_0,+\infty )$
  \begin{equation*}
    \mdvx F_i(x) = 4  \msum{j \in \cI}
    \big( \gamma_j(2x) N_j(2x) \kappa_{ji} \big)
    \e^{\Lambda_i(x)} \, \geq \,  0.
  \end{equation*}
  This tells us for all $i$ in $\cI$, that if
  $\,\tau_i( \cdot) N_i( \cdot)$ becomes positive at some
  $x_1 \geq x_0$, then it stays positive on $[ x_1, + \infty)$.  Let
  thus introduce the smallest size of individuals with feature $v_i$
  at equilibrium
  \begin{equation*}
    \ell_i \coloneqq \inf_{(0,\infty)} \lb x : \tau_i(x) N_i(x) >0 \rb,
    \quad i \in \cI,
  \end{equation*}
  and prove that all the $\ell_i$, $i\in \cI$, are equal to
  \smash{$\frac{b}{2}$}. We already know that they {are} greater
  than \smash{$\frac{b}{2}$} so it only remains to prove that $\ell_i$
  is lower than \smash{$\frac{b}{2}$} for every $i$.

  For all $i \in \cI$, given that $x \mapsto \tau_i(x) N_i(x)$ cancels
  on $[0,\ell_i]$ and because $\tau_i$ can be bounded from below by a
  positive constant a.e. on any compact sets of $(0, +\infty)$ from
  \ref{as:tau_positive}, $N_i$ cancels a.e. on
  $(0,\ell_i]$. Integrating the direct equation of~\ref{pb:GFv}
  on~$(0,\ell_i)$ hence brings:
  \begin{equation*}
    0= \msum{j \in \cI} \mint_0^{\ell_i}  \gamma_j( 2x) N_j(2x)
    \kappa_{ji}  \dd{x}
    = \frac{1}{2} \msum{j \in I_{_{_\leftarrow}} \! (i)} 
    \mint_{\max(b, \ell_j)}^{2 \ell_i}  \beta(x) \tau_j(x) N_j(x) \kappa_{ji}  
    \dd{x}, \qquad  \forall i \in \cI,
  \end{equation*}
  with \emph{positive integrand} at the right-hand side which tells us
  that necessarily:
  \begin{equation}\label{eq:proof_suppN_in_x}
    \forall i \in \cI,
    \qquad  2 \ell_i \leq \max \bigl(b, \ell_j \bigr),
    \quad \forall j \in I_{_\leftarrow}(i).
  \end{equation}

  Now consider any $i_1$, $i_2$ in $\cI$.
  Assumption~\ref{as:kappa_irr} indicates that there exists
  $m = m_{i_1 i_2} \leq M$ such that $\kappa^m_{i_1 i_2} > 0$, meaning
  that we can find a path connecting $v_{i_1}$ to $v_{i_2}$ in the
  directed graph associated to $\kappa$, namely that there
  is~$(v_{p_1}, \ldots, v_{p_m}) \in \cV^{m}$ such that:
  \begin{equation*}
    \left\{
      \begin{aligned}
        &v_{p_1} = v_{i_1},
        \quad v_{p_m} = v_{i_2}, \\
        &\kappa(v_{p_i}, v_{p_{i+1}}) = \kappa_{p_i p_{i+1}} > 0, \;
        \ie \; p_{i+1} \in I ( p_{i} ),
        \quad \forall i \in \{1,\ldots, m-1\}.
    \end{aligned}
    \right. 
  \end{equation*}

  Applying~\eqref{eq:proof_suppN_in_x} iteratively to the $p_i$, for
  $i$ from $1$ to $m$, we obtain: $\forall i \in \{1, \ldots, m-1 \}$
  \begin{equation}\label{eq:proof_ineq_on_l}
    \ell_{i_1} \leq \mfrac{1}{2} \max \bigl(b, \mfrac{\ell_{p_i}}{2^i} \bigr)
    \leq \mfrac{1}{2} \max \bigl(b, \mfrac{\ell_{p_{i+1}}}{2^{i+1}} \bigr)
    \leq \mfrac{1}{2} \max \bigl(b, \mfrac{\ell_{i_2}}{2^m} \bigr).
  \end{equation}
  
  The inequality tells us that as soon as $v_{i_1}$ and $v_{i_2}$
  communicate ($v_{i_1}$ leads to $v_{i_2}$ and $v_{i_2}$ to $v_{i_1}$
  in the graph of $\kappa$), then $\ell_{i_1}$ and $\ell_{i_2}$ are
  both either infinite or inferior or equal to $\frac{b}{2}$ --it
  suffices to apply~\eqref{eq:proof_ineq_on_l} to $({i_2},{i_1})$ and
  $({i_2},{i_1})$, distinguishing between cases
  $\ell_{i_2} \leq 2^m b$ or $\ell_{i_2} > 2^mb$, and $\ell_{i_1}$
  lower or greater than $2^{m_{i_2 i_1}}b$, respectively. However the
  irreducibility condition~\ref{as:kappa_irr} implies that all the
  features communicate, so we deduce that all the $\ell_i$,
  $i \in \cI$, are either infinite --this would imply that all the
  $N_i$ are null-functions on $\R_+$ which is immediately excluded by
  the normalization condition on $N$-- or lower or equal
  to~$\frac{b}{2}$.

  \textit{Adjoint eigenvector.} Following the same idea we define
  $G_i \coloneqq \phi_i \e^{-\Lambda_i}$ on~$(x_0,+\infty)$, for $x_0$
  positive, and compute for all $i \in \cI$ and for {all}
  $x \in (x_0,+\infty)$:
  \begin{equation}\label{eq:proof_dG}
    \mdvx G_i(x) = -2 \beta  (x) \msum{j \in \cI} \kappa_{ij}
    \phi_j \Big( \mfrac{x}{2} \Big) \e^{-\Lambda_i(x)} \leq 0,
  \end{equation}
  so, as soon as $\phi_i$ vanishes it remains null, plus $G_i$ is
  constant on $(0,b]$ for every $i \in \cI$.

  We prove that for all $i \in \cI$, $\phi_i$ never reaches
  zero. Again, we introduce:
  \begin{equation*}
    k_i \coloneqq \sup_{[0,+\infty]} \bigl\{ x : \phi_i(x) >0
    \bigr\}, \quad i \in \cI.
  \end{equation*}
  By definition, $\phi_i$ is zero on $[k_i, +\infty)$ for any $i$ in
  $\cI$, therefore so are $G_i$ and \smash{$\mdvx G_i$}.  Thus
  integrating~\eqref{eq:proof_dG} on $[k_i, +\infty)$ yields after
  straightforward computations, similar to those done for the direct
  eigenvector, that:
  \begin{equation*}
    \msum{j \in I_{_{_\rightarrow}} \!(i)} \kappa_{ij}
    \mint_{\max \bigl( \frac{b}{2}, \frac{k_i}{2} \bigr)}^{k_j}
    \beta (2x) \phi_j(x) \e^{-\Lambda_i(2x)} \dd{x} = 0,
    \qquad i \in \cI,
  \end{equation*}
  with \emph{positive integrand} which leads to:
  \begin{equation}\label{eq:proof_max_phi}
    \forall i \in \cI, \qquad  k_j \leq  \max \Bigl( \mfrac{b}{2},
    \mfrac{k_i}{2} \Bigr), \quad \forall j \in I_{_\rightarrow}(i).
  \end{equation}
  Same arguments based on~\eqref{eq:proof_max_phi}
  and~\ref{as:kappa_irr} hold and allow us to deduce that the $k_i$,
  $i \in \cI$, are either all infinite or all inferior or equal to
  $\frac{b}{2}$. The first option is contradicted by the normalization
  condition on $\phi$. As for the second one, it proves impossible
  unless all $k_i$ are zero. Indeed if $0<k_i \leq \frac{b}{2}$ for
  some $i \in \cI$, we take $x_0 \in (0, k_i)$ and get $G_i$ constant
  on $[x_0, k_i]$ (constant on $[x_0, b]$) equal to $G_i( k_i)=0$
  which contradicts the definition of $k_i$.

  Assume now that \smash{$\frac{1}{\tau_j} \in L_0^1$} for some
  $j \in \cI$. Then we can take $x_0 = 0$ in the definition of $G_i$
  and $\phi_j(0) = 0$ would imply that $\phi_j$ is zero everywhere,
  i.e.  $k_i = +\infty$ which we said is impossible, so
  $\phi_j(0) >0$.
\end{proof}

A few remarks on the proof:
\begin{itemize}[leftmargin=.7cm,parsep=0cm,itemsep=0.1cm,topsep=0cm]
\item This proof also works in the case where the graph associated to
  $\kappa$ is not strongly connected but decomposes in strongly
  connected components partitioned, say in~$(\cV_1, \ldots, \cV_p)$
  associated to a partition $(\cI_1, \ldots, \cI_p)$ of $\cI$. In this
  case, inequality~\eqref{eq:proof_ineq_on_l} implies that the
  property of $i \mapsto \ell_i$ to be either infinite or upper
  bounded by \smash{$\frac{b}{2}$} is valid on each of the $\cI_i$,
  and the normalization condition on $N$ then
  yields~\smash{$\ell \leq \frac{b}{2}$} on at least one $\cI_i$, but
  $\ell$ possibly infinite elsewhere (meaning $N$ null outside
  {of}~\smash{$\cV_i \times \big[\frac{b}{2}, +\infty \big)$}).

\item Likewise, equation~\eqref{eq:proof_positivity} tells us that if
  there was a feature $v_i$ that somehow could not be transmitted over
  generations, this is $I_{_\leftarrow}(i)=\emptyset$, then
  $\tau_i N_i$ would be zero on~$\R_+$.
\end{itemize}

\subsection{General relative entropy}
\label{ssec:GRE}

The GRE principle (see~\cite{perthame_transport_2007,
  michel_general_2004, michel_general_2005}) holds in the variability
setting and can be stated under assumption~\ref{as:V_continuous} of
$\cV$ continuous, so we first keep integral notation for the $v$
variable. Likewise, we start with a general abstract setting which
does not involve the eigenelements yet but a solution to the dual
equation given, for $\cF^*$ defined by~\eqref{eq:SV_adjoint_opp}, as:
\begin{equation*}\label{pb:GFtv_ad}
  \tag*{\textnormal{(GF$^*_{\! t,v}$)}}
  \mpdvt \psi(t,v,x)+ \tau(v,x) \mpdvx\psi(t,v,x) = -\cF^* \psi(t,v,x).
\end{equation*}

We adopt the compact notation $u_t(v,x) \coloneqq u(t,v,x)$.

\begin{definition}[GRE]
  We define the \emph{general relative entropy} of $n$ with respect to
  $p$ by
  \begin{equation*}
    E^H[n | p](t) \coloneqq \miint_\cS \psi_t(v,x) p_t(v,x) H \Bigl(
    \mfrac{n_t(v,x)}{p_t(v,x)} \Bigr) \, \dd v \dd x,
  \end{equation*}
  with $H \colon \R \rightarrow \R$ a convex function, $\psi$ weak
  solution to~\ref{pb:GFtv_ad} and $n$ and $p$ such that the
  integral is defined.
\end{definition}

\begin{prop}[GRE principle]\label{prop:GRE_principle}
  Assume~\ref{as:tau-gamma}, \ref{as:beta_support}
  and~\ref{as:kappa_cv} are satisfied and $n$ and $p$ are weak
  solutions to~\ref{pb:GFtv}, and $\psi$
  to~\ref{pb:GFtv_ad}, with $p \geq 0$ and $\psi$ such that
  $\supp( \psi_t) \cap \cV \times \big[\tfrac{b}{2}, +\infty \big)
  \subset \supp(p_t)$ and $\psi_t p_t \in L^1(\cS)$ for all positive
  $t$. Then, the \emph{entropy dissipation} $D^H[n | p]$ of $n$
  w.r.t. $p$, defined by
  \begin{equation*}
    \tfrac{\dd}{\dd{t}} E^H[n | p](t)
    \coloneqq -D^H[n | p](t),
  \end{equation*}
  can be expressed, for any $H \colon \R \rightarrow \R$ convex, as
  \begin{multline*}
    D^H[n | p](t) =-
    4 \miint_\cS \psi_t(v,x) \mint_\cV \kappa(v',v) \gamma(v', 2x) p_t(v', 2x)\\
    \! \times \! \lp H' \Big(\mfrac{n_t(v,x)}{p_t(v,x)} \Big) \lc
    \mfrac{n_t(v',2x)}{p_t(v',2x)} - \mfrac{ n_t(v,x)}{p_t(v,x)} \rc -
    \Big[ H \Big( \mfrac{n_t(v',2x)}{p_t(v',2x)} \Big) -H \Big(
    \mfrac{n_t(v,x)}{p_t(v,x)} \Big) \Big] \rp \, \dd v' \dd v \dd x.
  \end{multline*}

  In particular, for $H$ convex and $\psi \geq 0$ the general relative
  entropy is decreasing since $D^H[n | p ](t) \geq 0$, and for
  $H$ strictly convex and $\psi$ positive we have
  \begin{equation*}
    D^H[n | p ](t) = 0 \quad \Longleftrightarrow \quad
    \mfrac{n_t(v',2x)}{p_t(v',2x)} = \mfrac{ n_t(v,x)}{p_t(v,x)},
    \quad (v',v,x) \in \supp \bigl(\kappa(v',v) \gamma(v', 2x) \bigr).
  \end{equation*}
\end{prop}

\begin{proof} The standard computation, see
  e.g.~\cite[p.93]{perthame_transport_2007}, is not affected by the
  additional variable $v$, we only need to make sure quantities are
  well defined since $p$ here is not necessarily positive everywhere
  --we rely on $\supp( \gamma) \coloneqq \cV \times [b, +\infty)$
  from~\ref{as:tau-gamma} and~\ref{as:beta_support}.

  Using equation~\ref{pb:GFtv} verified by both $n$ and $p$ we
  obtain
  \begin{equation}\label{eq:proof_GRE_n/p}
    \begin{aligned}
      \mpdvt \Big( \mfrac{n_t(v,x)}{p_t(v,x)} \Big) + \tau(v,x) \mpdvx
      \Big( \mfrac{n_t(v,x)}{p_t(v,x)} \Big) &= \Big[ \mfrac{1}{p_t^2}
      \Big( p_t \big( \mpdvt n_t + \mpdvx (\tau n_t) \big)
      -n_t \big( \mpdvt p_t + \mpdvx (\tau p_t) \big) \Big)\Big](v,x)\\
      &=4\mint_{\cV} \mfrac{p_t(v',2x)}{p_t(v,x)} \kappa (v',v)
      \gamma(v',2x) \lc \mfrac{n_t(v',2x)}{p_t(v',2x)}
      - \mfrac{ n_t(v,x)}{p_t(v,x)} \rc \dd{v'}.\\
    \end{aligned}
  \end{equation}
  Multiplying by $H'\big( \frac{n_t(v,x)}{p_t(v,x)} \big) $, we find
  \begin{equation}\label{proof:eq_A}
    \begin{aligned}
      \mpdvt H\big( \tfrac{n_t}{p_t} \big) +\tau \mpdvx H\big(
      \tfrac{n_t}{p_t} \big) &=4 \mint_{\cV}
      \mfrac{p_t(v',2x)}{p_t(v,x)} \kappa (v',v) \gamma(v',2x) H'\Big(
      \mfrac{n_t(v,x)}{p_t(v,x)} \Big)\lc
      \mfrac{n_t(v',2x)}{p_t(v',2x)}
      - \mfrac{ n_t(v,x)}{p_t(v,x)} \rc \dd{v'}\\
      &\hspace{-.1cm}\coloneqq A_t(v,x).
    \end{aligned}
  \end{equation}
  Besides, using the equations on $\psi$ and $p$ we have
  \begin{equation*}
    \begin{aligned}
      \mpdvt \big( \psi_t p_t \big)+\mpdvx \big(\tau\psi_t p_t \big) =
      \big( \mpdvt \psi_t + \tau \mpdvx \psi_t \big) p_t + \big(
      \mpdvt p_t + \mpdvx (\tau p_t) \big) \psi_t = -\cF^*(\psi_t) p_t
      + \cF (p_t) \psi_t,
    \end{aligned}
  \end{equation*}
  and thus 
  \begin{equation*}
    \begin{aligned}
      \mpdvt \Big( \psi_t p_t H\big( \tfrac{n_t}{p_t} \big)
      \Big)+\mpdvx \Big(\tau\psi_t p_t H\big( \tfrac{n_t}{p_t} \big)
      \Big) = \Big( \cF(p_t) \psi_t -\cF^*(\psi_t) p_t \Big) H\big(
      \tfrac{n_t}{p_t} \big) + \psi_t p_t A_t.
    \end{aligned}
  \end{equation*}
  It then remains to integrate for $(v,x)$ in $\cS$ to get an
  expression of the entropy dissipation. \textcolor{black}{The
    $x$-derivative term cancels} and we obtain
  \begin{equation}\label{eq:proof_dissip}
    \tfrac{\dd}{\dd{t}} \miint_\cS \psi_t \hspace{.1pt}p_t H \big(
    \tfrac{n_t}{p_t} \big)
    \, \dd v \dd x = \underbrace{\miint_\cS \Big( \cF (p_t) \psi_t
      -\cF^*(\psi_t) p_t \Big) H\big( \tfrac{n_t}{p_t} \big)\, \dd v
      \dd x}_{\coloneqq I_t} + \miint_\cS \psi_t p_t A_t \, \dd v \dd x
  \end{equation}
  that is {well defined} from the assumption on the support of
  $\psi$, $p$ and the expression of the support of $\gamma$.  The
  first part $I_t$ of the dissipation is handled through a change of
  variables in its second term in order to factorize by
  $\cF (p_t)(v,x) \psi_t(v,x)$:
  \begin{equation*}
    I_t = 4 \miint_\cS \! \mint_\cV \kappa(v',v) \psi_t(v,x)
    \gamma(v',2x) p_t(v', 2x) \Big[ H \Big( \mfrac{n_t(v,x)}{p_t(v,x)}
    \Big) -H \Big( \mfrac{n_t(v',2x)}{p_t(v',2x)} \Big) \Big] \, \dd
    v' \dd v \dd x,
  \end{equation*}
  which, combined with~\eqref{proof:eq_A} and~\eqref{eq:proof_dissip},
  brings the expected formulation.
  
  The inequality in the case $H$ convex (strictly convex) follows from
  the non-negativity (positivity) of $\phi$ and $p$, $\gamma$ and
  $\kappa$, and a classical inequality of convexity.
\end{proof}

As direct consequence of \Cref{prop:GRE_principle} we have the
following classical properties.
\begin{prop}\label{prop:contraction}
  Assume hypothesis of \Cref{thm:existence_eig} satisfied and let
  $(\lambda,N,\phi)$ be a solution to~\ref{pb:GFv}. Assume
  $n \in \sC \big(\R_+, L^1(\cS; \phi \, \dd v \dd x) \big)$ is a
  solution to~\ref{pb:GFtv}. Then, setting
  $\, \tilde{n}\coloneqq n \e^{-\lambda t}$ we have:
  $\forall t \geq 0$
  \begin{enumerate}[label=\roman*),leftmargin=*, parsep=0cm,
    itemsep=.2cm, topsep=0cm]
  \item (Conservation law).
    $\miint_\cS \tilde{n}_t(v,x) \phi(v,x) \, \dd v \dd x = \miint_\cS
    n^{in}(v,x) \phi(v,x) \, \dd v \dd x$.
  \item (Contraction principle).
    $\miint_\cS \big\lvert \tilde{n}_t(v,x) \big\rvert \phi(v,x) \,
    \dd v \dd x \leq \miint_\cS \big\lvert n^{in}(v,x) \big\rvert
    \phi(v,x) \, \dd v \dd x$.

    More precisely,
    $s \mapsto \lp \iint_\cS \abs{ \tilde{n}_s } \phi \rp$ is a
    decreasing function of time.
    
  \item \vspace{.2cm}(Maximum Principle).
    $\abs{ n^{in} (v,x) } \leq C N(v,x) \Longrightarrow \abs{
      \tilde{n}_t (v,x) } \leq C N(v,x)$.
  \end{enumerate}
\end{prop}

\begin{proof}
  From \Cref{prop:positivity}, $\phi$ is positive on $\cS$ and
  $\supp(N) = \cV \times [ \frac{b}{2},
  +\infty)$. \Cref{prop:GRE_principle} thus applies to
  $p_t \coloneqq N \e^{\lambda t}$ and
  $\psi_t \coloneqq \phi \e^{-\lambda t}$, respective solutions
  to~\ref{pb:GFtv} and \ref{pb:GFtv_ad}, and for
  {well-chosen} \emph{convex} entropy functions we retrieve:
  \begin{enumerate}[label=~~\textit{\roman*)}, wide, parsep=0cm,
    itemsep=0.1cm, topsep=0cm]
  \item the conservation law, after taking $H(u)=u$ and $-H$,
    
  \item the contraction principle from choosing $H(u)=\abs{u}$,

  \item and the maximum principle from
    \smash{$H(u)=(\abs{u}-C)^{2}_+$}. Indeed, $\abs{n^{in}}\leq C N$
    then~implies \smash{$H(\frac{n^{in}}{N})=0$}. Therefore,
    $E^{\texp{H}}[n^{in}|N](0)=0$ and because the entropy decays (from
    $D^{\texp{H}}[\tilde{n}|N ](t) = D^{\texp{H}}[n|N \e^{\lambda t}
    ](t) \geq 0$) we have that for all times $t$,
    $E^{\texp{H}}[\tilde{n}|N](t) = \int \! \! \! \int_\cS \phi N H
    \big( \frac{\tilde{n}_t}{N} \big)$ is 0, and deduce the expected
    result.
  \end{enumerate}\vspace{-.4cm}
\end{proof}

Before going further we establish two useful results, also derived
from the GRE and later needed to establish convergence.
\begin{prop}[BV-regularity]\label{prop:BV_reg}
  Assume that the assumptions of \Cref{thm:existence_eig} are
  satisfied such that there exists a solution $(\lambda,N,\phi)$
  to~\ref{pb:GFv} with $\lambda >0$.  Assume that
  $n \in \sC \big(\R_+, L^1(\cS; \phi \, \dd v \dd x) \big)$ is
  a weak solution to~\ref{pb:GFtv} with initial data
  satisfying
  \begin{equation*}
    \abs{n^{in}} \leq C N, \qquad \mpdvx \big( \tau  n^{in} \big) 
    \in L^1 \big(\cS; \phi \, \dd v \dd x \big).
  \end{equation*}
  Then, for $\tilde{n} \coloneqq n \e^{- \lambda t}$ and $C(n^{in})$ a
  general constant depending on $n^{in}$, we have: $\forall t \geq 0$
  \begin{equation*}
    \mint_0^\infty \abs{\mpdvt \tilde{n}_i(t, x)} \phi_i(x) \dd{x}
    \leq C(n^{in}), \quad
    \mint_0^\infty \abs{\mpdvx \big( \tau_i(x) \tilde{n}_i(t,x) \big)} \phi_i(x) \dd{x}
    \leq C(n^{in}), \quad i \in \cI.
  \end{equation*}
\end{prop}

\begin{proof} \textit{Time derivative.} The renormalization
  of~\ref{pb:GFtv}, satisfied by $\tilde{n}$, writes as:
  \begin{equation} \label{eq:proof_dynamic_renormalized}
    \mpdvt \tilde{n}_i(t,x) %
    + \mpdvx\big( \tau_i(x) {\tilde{n}}_i(t,x) \big) %
    + \big( \lambda  + \gamma_i(x) \big) \tilde{n}_i(t,x)
    = 4 \msum{j \in \cI} \gamma_j(2x) \tilde{n}_j(t,2x) \kappa_{ji}.
  \end{equation}
  Differentiating it in time yields to the same equation verified by
  $q = \mpdvt \tilde{n}$, to which we can apply the contraction
  principle of \Cref{prop:contraction} and derive:
  \begin{equation*}
    I_t \coloneqq \msum{i \in \cI} \mint_0^\infty \abs{q_i(t,x)}
    \phi_i(x) \dd{x}
    \leq \msum{i \in \cI} \mint_0^\infty \abs{q_i(0,x)} \phi_i(x) \dd{x}.
  \end{equation*}
  Evaluating at $t=0 \,$ the equation verified by $\tilde{n}$, we have
  \begin{equation*}
    q_i(0,x) = - \mpdvx \big( \tau_i(x) n^{in}_i(x) \big)
    - \big(\lambda + \gamma_i(x)
    \big) n^{in}_i(x) +  4  \msum{j \in \cI} \gamma_j(2x)
    n^{in}_j(2x) \kappa_{ji},
  \end{equation*}
  whose absolute value can be controlled, bounding $\abs{n^{in}}$ by
  $C N$ and then replacing \hphantom{by}
  $\sum_j \gamma_j(2x) N_j(2x) \kappa_{ji}$ by the other terms of the
  equation on $N$, to obtain:
  \begin{equation*}
    I_t \leq \msum{i \in \cI} \mint_0^\infty \! \phi_i \Big(
    \abs{\mpdvx \big( \tau_i n^{in}_i \big)} + C \abs{\mpdvx \big(
      \tau_i
      N_i \big) } \Big)
    + 2 \lambda C + 2C \msum{i \in \cI} \mint_0^\infty \gamma_i
    N_i \phi_i.
  \end{equation*}
  The term in $n^{in}$ is controlled by assumption, the other ones
  thanks to the estimates on $(N, \phi)$ provided by
  \Cref{thm:existence_eig}, using in particular that
  \begin{equation*}
    \msum{i \in \cI} \mint_0^\infty \gamma_i(x)
    N_i(x) \phi_i(x) \dd{x}
    \leq \Norm{ \mfrac{\phi}{1+x^k}}{L^\infty(\cS)}
    \bigNorm{ (1+ x^k) \gamma N}{L^1(\cS)}.
  \end{equation*}

  \textit{Size derivative.} Starting again from
  \eqref{eq:proof_dynamic_renormalized} we use the estimate on the
  time derivative just obtained, the maximum principle of
  \Cref{prop:contraction} ($\abs{\tilde{n}_i(t,x)} \leq C N(x)$) and
  the estimates on $N$ to control the size derivative as expected.
\end{proof}

\begin{lemma}[GRE minimizers]\label{lemma:GRE_minimizer} Assume
  $n$,~$p$, $\psi$ and $H$ satisfy assumptions
  of~\Cref{prop:GRE_principle}, with $\psi$ positive and $H$ strictly
  convex. Assume~\ref{as:V}, \ref{as:kappa_irr}, \ref{as:tau_hetero},
  \ref{as:tau-gamma} and \ref{as:beta_support} satisfied. Then
  \begin{equation*}
    D^H [ n | p ] \equiv 0 \quad \Longleftrightarrow
    \quad \exists C \in \R \; : \; \mfrac{n_i(t,x)}{p_i (t,x)} =C,
    \quad  \forall t>0, \; \;
    \forall i \in \cI, \; \nae x \in \big( \tfrac{b}{2}, +\infty \big).
  \end{equation*}
\end{lemma}

\begin{proof}
  Before starting, we introduce $I_{_\leftarrow}$ and
  $I_{_\rightarrow}$ as defined by~\eqref{def:Iarrow} and formulate
  the \enquote{support of
    $(v', v, x) \mapsto \kappa(v',v) \gamma(v',2x)$} with discrete
  index notations:
  \begin{equation*} 
    \Delta
    \coloneqq \overline{ \ensurestackMath{\addstackgap[1.2pt]{
          \Big\{ (j, i, x) \in \cI^2 \times \R_+ :
          \kappa_{ji}  \gamma_j(2x) > 0 \Big\}}}},
  \end{equation*}
  which can be simplified, thanks to the assumption on the support of
  $\gamma$, into
  \begin{equation*}
    \Delta = \Big\{ (j, i) \in \cI^2 : 
    j \in I_{_\leftarrow}(i) \Big\}  \times
    \big[ \tfrac{b}{2}, +\infty \big) .
  \end{equation*}
  Also, note that because $I_{_\leftarrow}(i)$ is not empty for every
  $i \in \cI$ thanks to~\ref{as:kappa_irr} we have
  \begin{equation}\label{eq:def_delta0}
    \Big \{ (i, x) : \exists j \in \cI \st (j, i, x) \in \Delta \Big\}
    = \cI \times \big[ \tfrac{b}{2}, +\infty \big).
  \end{equation}

  Now, from~\Cref{prop:GRE_principle} we rather prove that the second
  assertion is equivalent to
  \begin{equation}\label{eq:proof_g/N}
    \mfrac{n_i(t,x)}{p_i(t,x)}=\mfrac{n_j(t,2x)}{p_j(t,2x)},
    \quad \forall t>0, \quad (j, i, x) \in \Delta.
  \end{equation}
  
  From now on we set \smash{$u_i \coloneqq {n_i / }{p_i}$} for all
  $i\in \cI$ to simplify notations. It is clear that
  $u_i(t, \cdot) \equiv C$ a.e. on $\big[\frac{b}{2}, +\infty \big)$
  for all $t>0$ and $i \in \cI$, implies~\eqref{eq:proof_g/N}. Let us
  prove the converse.

  Recalling~\eqref{eq:proof_GRE_n/p}, we use~\eqref{eq:proof_g/N}
  (which holds for any $i \in \cI$, a.e.
  $x \in \big[\frac{b}{2}, +\infty \big)$ from~\eqref{eq:def_delta0})
  and~get:
  \begin{equation}\label{eq:proof_carac_g/N}
    \mpdvt u _i(t,x) +\tau_i(x) \mpdvx u_i(t,x) =0,
    \qquad t>0, \quad (i, x)  \in \cI \times
    \big[\tfrac{b}{2}, +\infty \big).
  \end{equation}
  On the one hand,~\eqref{eq:proof_g/N} gives: for $t>0$,
  $(j, i, x) \in \Delta$
  \begin{equation*}
    \mpdvt \big( u_i(t,x) \big) +\tau_i(x) \mpdvx \big( u_i(t,x) \big)
    = \mpdvt \big( u_j(t,2x) \big) +\tau_i(x) \mpdvx \big(
    u_j(t,2x) \big),
  \end{equation*}
  which is
  \begin{equation*}
    \mpdvt u _i(t,x) +\tau_i(x) \mpdvx u_i(t,x)
    = \mpdvt u_j(t ,2x) + 2\tau_i(x) \mpdvx u_j(t,2x),
  \end{equation*}
  that, combined with~\eqref{eq:proof_carac_g/N}, eventually brings
  \begin{equation}\label{eq:proof_carac_g/N_ter}
    \mpdvt u_j(t, 2x) + 2\tau_i(x) \mpdvx u_j(t, 2x) =0,
    \qquad t>0, \quad (j, i, x) \in \Delta.
  \end{equation}
  On the other hand, \eqref{eq:proof_carac_g/N} being true for all
  $i \in \cI$ and $x \in \big[\frac{b}{2}, +\infty \big)$, we have
  \emph{a fortiori}
  \begin{equation}\label{eq:proof_carac_g/N_bis}
    \mpdvt u _j(t,2x) + \tau_j(2x) \mpdvx u_j(t,2x) = 0,
    \qquad t>0,
    \quad (j, x) \in \cI \times \big[\tfrac{b}{4}, +\infty \big).
  \end{equation}
  Taking the difference between
  equations~\eqref{eq:proof_carac_g/N_ter}
  and~\eqref{eq:proof_carac_g/N_bis} we find
  \begin{equation*} 
    \big( 2 \tau_i(x) - \tau_j(2x) \big) \mpdvx u_j(t,2x) = 0,
    \qquad t>0, \quad (j, i, x) \in \Delta.
  \end{equation*}
  Assumption \ref{as:tau_hetero} ensures that the first factor does
  not cancel on $\Delta$, thus
  \begin{equation*}
    \mpdvx u_i(t,x) = 2 \mpdvx u_j(t,2x) = 0,
    \qquad t>0, \quad (j, i, x) \in \Delta.
  \end{equation*}
  From~\eqref{eq:def_delta0}, this means that for all $i \in \cI$,
  $u_i$ is constant w.r.t. the size variable a.e. on the set
  \smash{$\big(\frac{b}{2}, +\infty \big)$}, which, injected
  in~\eqref{eq:proof_carac_g/N}, yields that it is also constant in
  the time variable:
  \begin{equation*}
    \mpdvx u_i(t,x) =0, \quad \mpdvt u_i(t,x) = 0,
    \qquad t>0, \quad
    (i,x) \in \cI \times \big(\tfrac{b}{2}, +\infty \big).
  \end{equation*}
  Therefore $u_i(t,x) = C_i$ for $t>0$, and
  $(i,x) \in \cI \times \big(\tfrac{b}{2}, +\infty \big)$, and we get
  by~\eqref{eq:proof_g/N} that:
  \begin{equation*}
    C_i = u_i(t,x) =u_j(t,2x) = C_j,
    \qquad  (j, i, x) \in \Delta,
  \end{equation*}
  which yields that for all $i \in \cI$, we have $C_j= C_i$ as soon as
  $j \in I_{_\leftarrow}(i)$. But then for all~$j$ in
  $I_{_\leftarrow}(i)$, the map $k \mapsto C_k$ is constant on
  $I_{_\leftarrow}(j)$, and so on, so that it is constant on sets of
  indices corresponding to any strongly connected component of the
  graph of~$\kappa$. We conclude thanks to~\ref{as:kappa_irr}: the
  graph of $\kappa $ is strongly connected, thus $k \mapsto C_k$
  constant on $\cI$.
\end{proof}

Another interesting consequence of the GRE principle is the uniqueness
of the direct eigenvector solution to~\ref{pb:GFv}. Uniqueness
for the whole eigenproblem can also be derived, see
\Cref{thm:uniqueness_eig} below whose proof is left to
\Cref{sec:proof}.

\begin{prop}[Uniqueness of eigenelements]\label{thm:uniqueness_eig}
  Assume assumptions of~\Cref{thm:existence_eig} are satisfied and
  that, in addition,~{\ref{as:tau_hetero}} and~\ref{as:kappa_irr}
  holds. Then, there exists a unique solution to the
  eigenproblem~\ref{pb:GFv}.
\end{prop}

\subsection{Long-time asymptotic behavior}
\label{ssec:conv}

After solving the eigenproblem, a natural wish is to characterize the
asymptotic behavior of a solution to the Cauchy
problem~\ref{pb:GFtv} for which we expect {the} A.E.G.:
\begin{equation*}
  n(t,v,x) \underset{t \rightarrow + \infty}{\sim}
  C \e^{\lambda t}N(v,x).
\end{equation*}
The GRE principle then proves very useful to get information on the
evolution of the distance between $n$ and $p = \e^{\lambda t}N(v,x)$
in {a} well-chosen $L^1$-norm (ponderated by $\phi$ by taking
$\psi = \e^{- \lambda t} \phi$ in the GRE inequality) since it
indicates that the entropy of $n$ w.r.t. $p$ decreases in time.
However we need, first to be sure that~\ref{pb:GFtv} admits a
weak solution in the suited space
$\sC \big(\R_+, L^1(\cS; \phi \, \dd v \dd x) \big)$ for any initial
condition allowing us to use the maximum principle of
\Cref{prop:contraction}-iii).

To state convergence we thus assume existence of a solution in
$\sC \big(\R_+, L^1(\cS; \phi \, \dd v \dd x) \big)$
to~\ref{pb:GFtv} for some initial data 
satisfying $\abs{n^{in}} \leq C N$.
When $\cV$ is finite~\ref{as:V}, a straightforward adaptation
of~\cite{bernard_asynchronous_2020} however guarantees the existence
of a weak solution to the Cauchy problem~\ref{pb:GFtv} in
$\sC \big(\R_+, L^1(\cS; \phi \, \dd v \dd x) \big)$ with initial
condition \smash{$n^{in} \in L^1 (\cS; \phi \, \dd v \dd x)$},
providing stronger assumptions on the coefficients. Introducing
\begin{equation*}
  \cP_{x_0, \omega_1} \coloneqq \big\{ f : \exists K_1 \geq K_0 > 0,\,
  \omega_0 \in (0, \omega_1], \; K_0 x^{\omega_0} \mathds{1}_{x \geq
    x_0} \leq f(x) \leq K_1 \max(1, x^{\omega_1}), x >0 \big\},
\end{equation*}
these assumptions {can be stated} as follows: for all $v \in \cV$:
\begin{itemize}[leftmargin=.7cm,parsep=0cm,itemsep=0.1cm,topsep=0cm]
\item $\tau(v,\cdot) : (0, +\infty) \rightarrow (0, +\infty)$ is
  $\sC^1$ and belong to $\cP_{1,1}$,
  
\item $\gamma(v, \cdot) : (0,+ \infty) \rightarrow [0, +\infty)$ is
  continuous with connected support (our \ref{as:beta_support}) and
  belongs to $\cP_{x_0, \omega}$ for some positive $x_0$, $\omega$,
  
\item either $\frac{1}{\tau(v, \cdot)}$ or
  $\frac{\gamma(v, \cdot)}{\tau(v, \cdot)}$ and
  $\frac{x^{\nu_0}}{\tau(v, \cdot)}$, for $\nu_0 \geq 0$, in
  $L^1_0$ (this is \ref{as:beta_L0} and \ref{as:tau_in_0} resp.).
\end{itemize}
Existence of a weak solution to~\ref{pb:GFtv} is besides
derived in~\cite{doumic_statistical_2015}, as the expectation of the
empirical measure of the underlying Markov process over smooth test
functions. The result is stated for the continuous compact
setting~\ref{as:V_continuous} without the modeling
assumption~\ref{as:tau-gamma} but continuous division rate $\gamma$
satisfying $\gamma(0) = 0$ and $\int^\infty x^{-1} \gamma = \infty$
(the analogous of~\ref{as:beta_in_inf}) and linear growth rate
$\tau : (v,x) \mapsto vx$.

\begin{theorem}[Convergence to equilibrium]\label{thm:conv}
  Assume~\ref{as:V},~\ref{as:tau},~\ref{as:gamma},~\ref{as:beta}
  and~\ref{as:kappa}. Take $(\lambda,N,\phi)$ the solution
  to~\ref{pb:GFv}, and assume
  $n \in \sC \big(\R_+, L^1(\cS; \phi \, \dd v \dd x) \big)$ is a
  solution to~\ref{pb:GFtv} with initial condition $n^{in}$
  verifying $\abs{n^{in}} \leq C N$, then we have
  \begin{equation*}
    \lim\limits_{t \rightarrow \infty} \msum{i \in \cI} \mint_0^\infty 
    \big\lvert n_i(t, x)
    \e^{-\lambda t}-\rho N_i(x) \big\rvert \phi_i(x) \dd{x} = 0,
    \quad  \rho \coloneqq \msum{i \in \cI} \mint_0^\infty  n^{in}_i(x)
    \phi_i(x) \dd{x}.
  \end{equation*}
\end{theorem}

\begin{proof}
  We follow the now classical proof given in Perthame's
  book~\cite[p.98]{perthame_transport_2007}. Note that one can always
  regularize the initial data so that it satisfies the assumptions
  of~\Cref{prop:BV_reg}, then contraction principle ensures that the
  convergence of the regularized problem implies converge of the
  initial problem. Therefore, it is enough to prove converge for
  $n^{in}$ satisfying the assumptions of~\Cref{prop:BV_reg}.

  \textit{First step.} For all $t \geq 0$, all $i \in \cI$ and for
  $x > 0$, we set
  \begin{equation*}
    h_i(t,x) \coloneqq n_i(t,x) \e^{-\lambda t}-\rho N_i(x)
  \end{equation*}
  and notice that $h \e^{\lambda t}$ satisfies
  {problem~\ref{pb:GFtv}} (linear and satisfied by both
$n$ and $N \e^{\lambda t}$) with regularized initial data. Thus
\Cref{prop:BV_reg} applies and ensures
  \begin{equation}\label{eq:BV_reg_h}
    \mint_0^\infty \abs{\mpdvt h_i(t,x)} \phi_i(x) \dd{x}
    \leq C(n^{in}), \qquad
    \mint_0^\infty \abs{\mpdvx h _i(t,x)} \phi_i(x) \dd{x}
    \leq C(n^{in}).
  \end{equation}
  Likewise, \Cref{prop:contraction} applies and brings first (Maximum
  principle and conservation law)
  \begin{equation}\label{eq:proof_int_gphi_zero}
    \abs{h} \leq C N,
    \qquad  \msum{i \in \cI} \mint_0^\infty h_i(t,x) \phi_i(x) \dd{x}
    = \msum{i \in \cI} \mint_0^\infty \big( n^{in}_i(x)
    -\rho N_i(x) \big) \phi_i(x)  \dd{x} = 0,
  \end{equation}
  and second (contraction principle) that
  \begin{equation*}
    R(t) \coloneqq \msum{i \in \cI} \mint_0^\infty
    \big\lvert n_i(t,x) \e^{-\lambda t}-\rho N_i(x)
    \big\rvert \phi_i(x) \dd{x}
  \end{equation*}
  decreases towards some non-negative constant, say $L$, as $t$ tend
  to infinity. We thus want to prove that $L$ is zero.
  
  \textit{Second step.} Instead of considering the convergence of $R$ as
  $t$ tends to infinity, we introduce the family of functions
  $(h^{\texp{(k)}})_{k \in \N}$, defined for all $t \geq 0$ and for
  $x>0$ by
  \begin{equation*}
    h^{\texp{(k)}}(t,x) \coloneqq h(t+k, x), \quad k\in\N,
  \end{equation*}
  and rather consider the
  $L^1([0, T] \times \R_+ ; \, \phi_i \dd{x})$-convergence of the
  sequences \smash{$(h_i^{\texp{(k)}})_{k \geq 0}$}, $i \in \cI$.

  Let us start with the
  \smash{$(\tau_i h_i^{\texp{(k)}})_{k \geq 0}$}, $i \in \cI$. Thanks
  to~\eqref{eq:BV_reg_h} and~\ref{as:tau_positive}, we have that for
  every compact set $K$ in~$(0, +\infty)$, the sequences of
  \smash{$\tau_i h_i^{\texp{(k)}}$}, for $i \in \cI$, are bounded in
  \smash{$BV ([0, T] \times K)$} (with weight~$\phi_i$).  By a
  diagonal argument, we can thus extract
  from~$(\tau h^{\texp{(k)}})_{k \in \N}$ a subsequence, denoted
  identically, s.t. for all compact $K \subset (0, +\infty)$ and
  every~$i \in \cI$, $(\tau_i h_i^{\texp{(k)}})_{k \in \N}$ converges
  strongly in $L^1([0,T] \times K; \, \phi_i \dd{x})$. Denote $s$ the
  limit in \smash{$\ell_1\big( \cI ; L^1([0, T] \times \R_+) \big)$.}
  The maximum principle ($\abs{h^{\texp{k}}} \leq C N$) besides
  ensures that \smash{$(h_i^{\texp{(k)}})_{k \geq 0}$} is
  equi-integrable and bounded in $L_1([0,T] \times \R_+)$. From the
  Dunford-Pettis theorem we can thus extract from
  \smash{$(h^{\texp{(k)}})_{k \geq 0}$} a subsequence that converges
  to some $g$, component-wise weakly in $L_1([0,T] \times \R_+)$. Note
  that for every $i \in \cI$, $\tau_i g_i$ and~$s_i$, the weak and
  strong limits of \smash{$(\tau_i h_i^{\texp{(k)}})_{k \geq 0}$},
  coincide on every compact sets of $[0, T] \times (0,
  +\infty)$. Using Fatou's lemma,~\ref{as:tau_positive} and the above
  results we conclude to the strong convergence of the
  \smash{$(h_i^{\texp{(k)}})_{k \geq 0}$}, $i \in \cI$, on
  $[0,T] \times \R_+$:
  \begin{multline*}
    \mint_0^\infty \abs{h_i^{(k)}(t,x) - g_i(t,x)} \dd{x} \leq 2 C
    \Big( \mint_0^\ep N_i (x) \dd{x} + \mint_X^\infty N_i
    (x) \dd{x}\Big) \\
    + \mfrac{1}{m_{\texp{[\ep, X]}}} \mint_\ep^X
    \abs{ \tau_i h_i^{(k)}(t,x) - s_i(t,x) } \dd{x}, \quad t \geq 0,
  \end{multline*}
  as small as possible for $k$, $X$ big and $\ep$ small. In
  addition, the strong limit $g$ satisfies the renormalization by
  $\e^{- \lambda t}$ of problem~\ref{pb:GFtv} and verifies
  \begin{equation*}
    \abs{g_i(t,x)} \leq C N_i(x), \quad \forall i \in \cI.
  \end{equation*}
  In particular the support of $g$ is included in the support of $N$,
  that is \smash{$\cV \times \big[\frac{b}{2}, +\infty \big)$} from
  \Cref{prop:positivity}, so that $g \equiv N \equiv 0$ on
  \smash{$\cV \times \big[0, \frac{b}{2} \big)$}.

  \textit{Third step.} We can now work on proving that $g$ is zero. The
  GRE principle stated in \Cref{prop:GRE_principle} applies to
  \smash{$n \coloneqq h \e^{\lambda t}$},
  \smash{$p \coloneqq N \e^{\lambda t}$},
  \smash{$\psi\coloneqq \phi \e^{-\lambda t}$} and
  \smash{$H(u) \coloneqq u^2$} convex,~bringing
  \begin{equation*}
    \begin{aligned}
      -4 \msum{i \in \cI} \mint_0^\infty \phi_i(x) \Big( \msum{j \in
        \cI} &\kappa_{ji} \gamma_j(2x) N_j(2x) \Big\lvert
      \mfrac{h_i(t,x)}{N_i(x)}
      -\mfrac{h_j(t,2x)}{N_j(2x)} \Big\rvert^2 \Big) \dd{x}\\
      &= \mfrac{\dd}{\dd{t}} \msum{i \in \cI} \mint_0^\infty
      \mfrac{\phi_i(x)}{N_i(x)} \big\lvert h_i(t,x) \big|^2 \dd{x} <
      +\infty.
    \end{aligned}
  \end{equation*}
  Therefore
  \begin{equation*}
    \begin{aligned}
      \mint_0^\infty \! &\msum{i, j \in \cI} \kappa_{ji} \phi_i(x)
      \gamma_j(2x) N_j(2x) \Big\lvert
      \mfrac{h^{\texp{(k)}}_i(t,x)}{N_i(x)}
      -\mfrac{h^{\texp{(k)}}_j(t,2x)}{N_j(2x)} \Big\rvert^2 \dd{x}\\
      &= \mint_k^\infty \! \msum{i, j \in \cI} \kappa_{ji} \phi_i(x)
      \gamma_j(2x) N_j(2x) \Big\lvert \mfrac{h_i(t,x)}{N_i(x)}
      -\mfrac{h_j(t,2x)}{N_j(2x)} \Big\rvert^2 \dd{x} \underset{k
        \rightarrow \infty}{\longrightarrow} 0.
    \end{aligned}
  \end{equation*}
  Passing to the strong limit the $h^{\texp{(k)}}$ in the first
  integral we get:
  \begin{equation*}
    \mint_0^\infty \msum{i, j \in \cI} \kappa_{ji} \phi_i(x) \gamma_j(2x)
    N_j(2x) \Big\lvert \mfrac{g_i(t,x)}{N_i(x)}
    -\mfrac{g_j(t,2x)}{N_j(2x)} \Big\rvert^2 \dd{x} = 0.
  \end{equation*}

  \textit{Fourth step.} We can finally apply~\Cref{lemma:GRE_minimizer}
  and get that \smash{$g_i(t,x) = CN_i(x)$} for all $i \in\cI$, and
  a.e. \smash{$x \in \big(\tfrac{b}{2}, +\infty \big)$} (and thus
  a.e. on $\R_+$ from the last remark of step 2). To conclude we use
  the normalization condition on $\phi$ and the conservation
  law~\eqref{eq:proof_int_gphi_zero} also valid on $g$ after passing
  to the limit in $k$:
  $ C = C \sum_i \int_\cS N_i \phi_i = \sum_i \int_{\cS} g_i(t) \phi_i
  = 0$ and thus $L=0$.
\end{proof}

\section{Illustration and discussion of the results}
\label{sec:illustration}

\subsection{Comparison of our result with existing literature}
\label{ssec:illustration_comparison}

The main message of Theorem \ref{thm:conv} is that {the
  heterogeneity in growth} rate, providing enough mixing in feature,
enables convergence towards a stationary profile. We provide here an
example that sheds light on how our result fits in the current
knowledge.

Let us consider the case where $\cV=\{v_1,v_2\}$ with $v_1<v_2$, take
$\tau(v,x) \coloneqq v x$, $\gamma: (v,x) \mapsto \tau(v,x) \beta(x)$
with $\beta$ such that~\ref{as:gamma} and~\ref{as:beta} are satisfied,
and define the two following variability kernels --a reducible one,
$\kappa^{red}$ standing for a population with absolutely no mixing in
feature, and an irreducible one, $\kappa^{irr}$ for a population with
mixing:
\begin{equation*}
  \kappa^{red} \coloneqq 
  \begin{pmatrix}
    1 & 0  \\
    0 & 1
  \end{pmatrix}, \qquad \kappa^{irr} \coloneqq
  \begin{pmatrix}
    0 & 1  \\
    1 & 0
  \end{pmatrix}.
\end{equation*}

In the non-mixing case, the Cauchy problem~\ref{pb:GFtv} is a
system of two {growth-fragmentation} equations that are not coupled
\begin{equation}\label{pb:GFtv_22NM} 
  \left\{
    \begin{aligned}
      & \mpdvt n_1(t,x)+ v_1\mpdvx \big( x n_1(t,x) \big) +
      \gamma_1(x) n_1(t,x) = 4 \gamma_1(2x) n_1(t,2x),\\
      & \mpdvt n_2(t,x)+ v_2\mpdvx \big( x n_2(t,x) \big) +
      \gamma_2(x) n_2(t,x) = 4 _2\gamma(2x) n_2(t,2x),\\
      & n_1(0,x) = n_1^{in} (x),\quad n_2(0,x) = n_2^{in} (x),
    \end{aligned}
  \right.
\end{equation}
whose associated eigenproblem, on the contrary, is coupled but only
through the eigenvalue~$\lambda^{red}$. Taken independently, each
equation is associated with a first positive eigenvalue associated
with eigenelements $(\lambda_i,N_i,\phi_i)$,
see~\cite{doumic_jauffret_eigenelements_2010}. Still, neither
$e^{-\lambda_1 t} {n}_1(t,x)$ nor
$e^{-\lambda_2 t} {n}_2(t,x)$ converge towards a stationary
profile as $t$ goes to infinity. Instead, we know
from~\cite{bernard_cyclic_2019} that they oscillate in the long time.

Besides, since Assumptions~\ref{as:V}, \ref{as:tau}, \ref{as:gamma},
\ref{as:tau-gamma}, \ref{as:beta}, \ref{as:kappa_cv} are satisfied,
\Cref{thm:existence_eig} applies and guarantees the existence of
eigenelements $(\lambda^{red}, N^{red},\phi^{red})$ associated to
system~\eqref{pb:GFtv_22NM}, with
\begin{equation*}
  0< \lambda_1 \leq \lambda^{red} \leq \lambda_2,
  \qquad N_i^{red} \geq 0,
  \quad \phi_i^{red} \geq 0,
  \quad i \in \{1,2\},
\end{equation*}
where the first inequalities are ensured by
\Cref{thm:monotonicity_V}. However, the irreducibility condition
\ref{as:kappa_irr} is not satisfied by $\kappa^{red}$, so uniqueness
is not ensured and \Cref{prop:positivity} does not apply to guarantee
that $N_1^{red}$ and $N_2^{red}$ are non-zero. Adapting the proof of
\Cref{prop:positivity} to each of the strongly connected components of
the graph of $\kappa$, simply reduced to $\{v_1 \}$ and $\{v_2 \}$,
shows that for each $i \in \{1,2\}$, $N_i^{red}$ is either zero
everywhere or positive on $(\frac{b}{2}, +\infty)$. The normalization
condition
\begin{equation*}
  \msum{i \in \{1,2\}} \mint_0^\infty \! N_i^{red}(x) \dd{x} = 1
\end{equation*}
implies that at least for one $i \in \{1,2\}$, $N_i^{red}$ is
non-zero.  Both $N_1^{red}$ and $N_2^{red}$ cannot be non-zero,
otherwise \Cref{lemma:GRE_minimizer} would work (in
equality~\eqref{eq:proof_g/N}, we need the existence of $i\neq j$
such that $N_i^{red} \neq 0$ and $N_j^{red} \neq 0$) and we could use
it to prove the long-time convergence of
system~\eqref{pb:GFtv_22NM} towards a stationary sate (fourth
step of \Cref{thm:conv}). That would
contradict~\cite{bernard_cyclic_2019}.  The only possibility is that
$N_1^{red} \equiv 0$ and $N_2^{red}(x)>0$ for $x \geq \frac{b}{2}$ or
the opposite and then $\lambda = \lambda_2$ or $\lambda = \lambda_1$
respectively.

In the mixing case, the equations of system~\ref{pb:GFtv} are
coupled through their source term:
\begin{equation*}
  \left\{
    \begin{aligned}
      & \mpdvt n_1(t,x)+ v_1\mpdvx \big( x n_1(t,x) \big) +
      \gamma_1(x) n_1(t,x) = 4 \gamma_2(2x) {n}_2(t,2x),\\
      & \mpdvt n_2(t,x)+ v_2\mpdvx \big( x n_2(t,x) \big) +
      \gamma_2(x) n_2(t,x) = 4 \gamma_1(2x) {n}_1(t,2x),\\
      & n_1(0,x) = n_1^{in} (x),\quad n_2(0,x) = n_2^{in} (x),
    \end{aligned}
  \right.
\end{equation*}
and the irreducibility condition~\ref{as:kappa_irr}, missing in the
non-mixing case, is now satisfied. We can thus apply successively
\Cref{thm:existence_eig}, \Cref{prop:positivity} and \Cref{thm:conv}
to get eigenelements $(\lambda^{irr}, N^{irr}, \phi^{irr})$ such that
\begin{equation*}
  \sum_{i=1}^2 \int_{\R^+} \big\lvert n_i(t,x)e^{-\lambda^{irr} t}
  - N_i^{irr}(x) \big\rvert \phi_i^{irr}{(x) \dd{x}}
  \underset{t \rightarrow \infty}{\longrightarrow} 0
\end{equation*}
with $\lambda^{irr} >0$, and for $i \in \{1,2\}$, $N_i^{irr}$ positive
on $(\frac{b}{2}, +\infty)$ and $\phi_i^{irr}>0$ on $(0,+\infty)$.

These two simple cases illustrate that the existence result
(\Cref{thm:existence_eig}) holds for every probability matrix
$\kappa$, in particular reducible ones. The irreducibility condition
on $\kappa$ comes into play to characterize the functions canceling
the dissipation of entropy (\Cref{lemma:GRE_minimizer}), which then
proves crucial to establish uniqueness of the steady state and
convergence towards it.

\subsection{Numerical illustration}
\label{ssec:numerics}

Similarly to the previous subsection, we focus here on the special
case of linear growth rates to illustrate the convergence result of
\Cref{thm:conv}. We numerically approximate and compare the long-time
asymptotics in {the} presence and absence of mixing in feature.

We choose $M=3$ different features, namely {$\cV = \{1, 2, 3\}$},
and approximate on the grid
\begin{equation*}
  \cS_N \coloneqq \cV \times \{ x_0, \ldots x_{2N} \},
  \qquad x_m \coloneqq 2^{\frac{-m-N}{k}}, \; m \in \{0, \ldots, 2N
  \}, \quad k = 200, \;  N = 2501,
\end{equation*}
the time-evolution of the following initial data (taken identical for
all features)
\begin{equation*}
  n^{in}: (v,x) \mapsto C x^a \e^{-b x^2},
  \qquad a=30, \quad b=60,
  \quad C  \, \st \lN n^{in} \rN_{L^1(\cV \times (0, x_{2N}))} = 1,
\end{equation*}
under the law given by the Cauchy problem~\ref{pb:GFtv} with
coefficients
\begin{equation*}
  \tau: (v,x) \mapsto v x,
  \qquad \gamma : (v,x) \mapsto x^2 \tau(v,x)
\end{equation*}
and variability kernel $\kappa^{red}$ or $\kappa^{irr}$, for the
non-mixing and mixing case respectively:
\begin{equation*}
  \kappa^{red} \coloneqq
  \begin{pmatrix}
    1 & 0 & 0 \\
    0 & 1 & 0 \\
    0 & 0 & 1 \\
  \end{pmatrix}, \qquad \kappa^{irr} \coloneqq
  \begin{pmatrix}
    0.7 & 0.2 & 0.1 \\
    0.5 & 0.4 & 0.1 \\
    0.3 & 0.3 & 0.4 \\
  \end{pmatrix}.
\end{equation*}
Our numerical scheme, available at
\href{https://github.com/anais-rat/growth-fragmentation}{%
  \small\texttt{https://github.com/anais-rat/growth-fragmentation}},
is an adaptation to the case $M>1$ of the scheme developed
in~\cite{bernard_cyclic_2019} to capture the oscillations that appear
in the case of equal mitosis, linear growth rate and no variability.
For $M=1$ the scheme remains identical, consisting in a splitting with
upwind scheme of $\text{CFL} = 1$ on a geometrical grid to avoid
diffusion. However, for $M>1$ the CFL is $1$ only for the fastest
individuals, those with feature $v_{max} = \max (\cV)$. It is
$\frac{v_i}{v_{max}} <1$ for the other individuals which yields
numerical diffusion in $v_i<v_{max}$.

We see on \Cref{fig:approx_n_nonmixing} that in the absence of mixing
the fastest subpopulation quickly dominates the others but does not
stabilize and keeps oscillating instead. However, as soon as mixing is
introduced (\Cref{fig:approx_n_mixing}), no subpopulation is
overwhelming the others anymore and all stabilize in shape to a steady
distribution.
\begin{figure}
  \centering
  \begin{subfigure}{0.49\textwidth}
    \centering%
    ~~~\includegraphics[height=.16\textheight]{%
      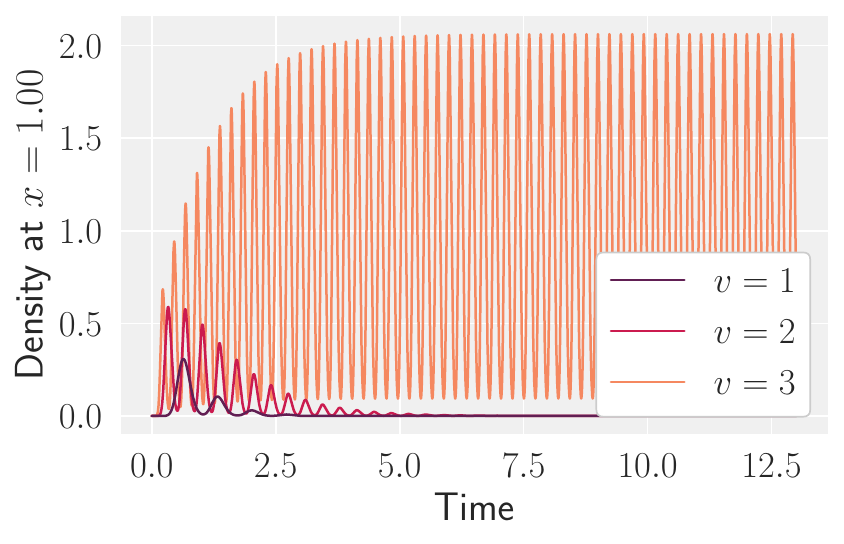}%
    
    \includegraphics[height=.7\textheight]{%
      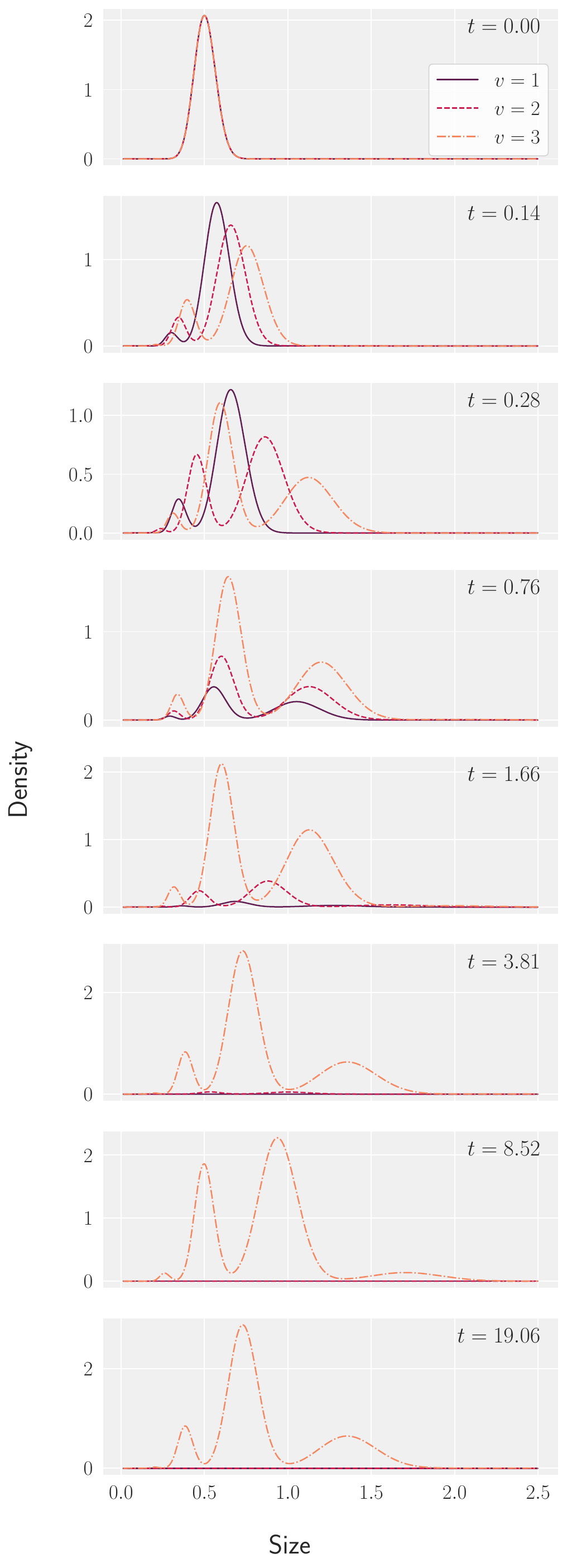}
    \caption{Non-mixing case}
    \label{fig:approx_n_nonmixing}
  \end{subfigure}
  \hfill
  \begin{subfigure}{0.49\textwidth}
    \centering%
    ~~~~\includegraphics[height=.16\textheight, %
    trim={0cm 0 0 0}, clip=true]{%
      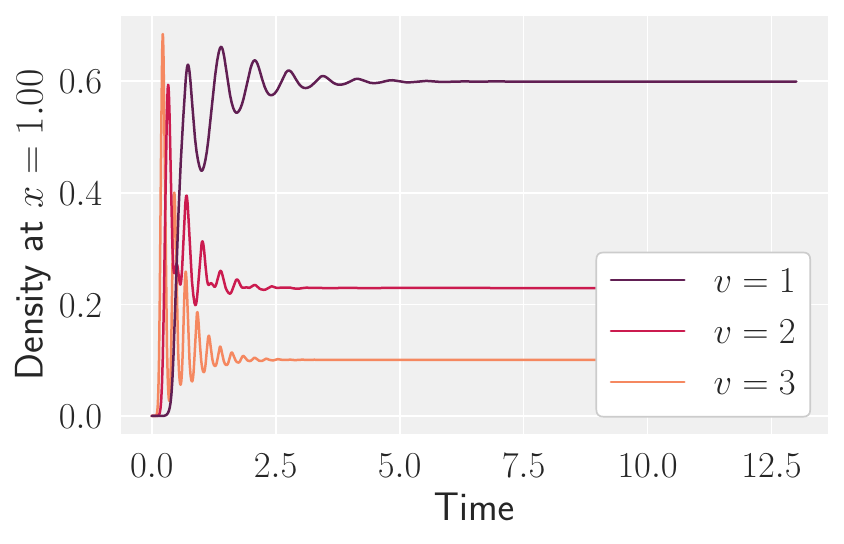}%
    
    \includegraphics[height=.7\textheight, %
    trim={0cm 0 0 0}, clip=true]{%
      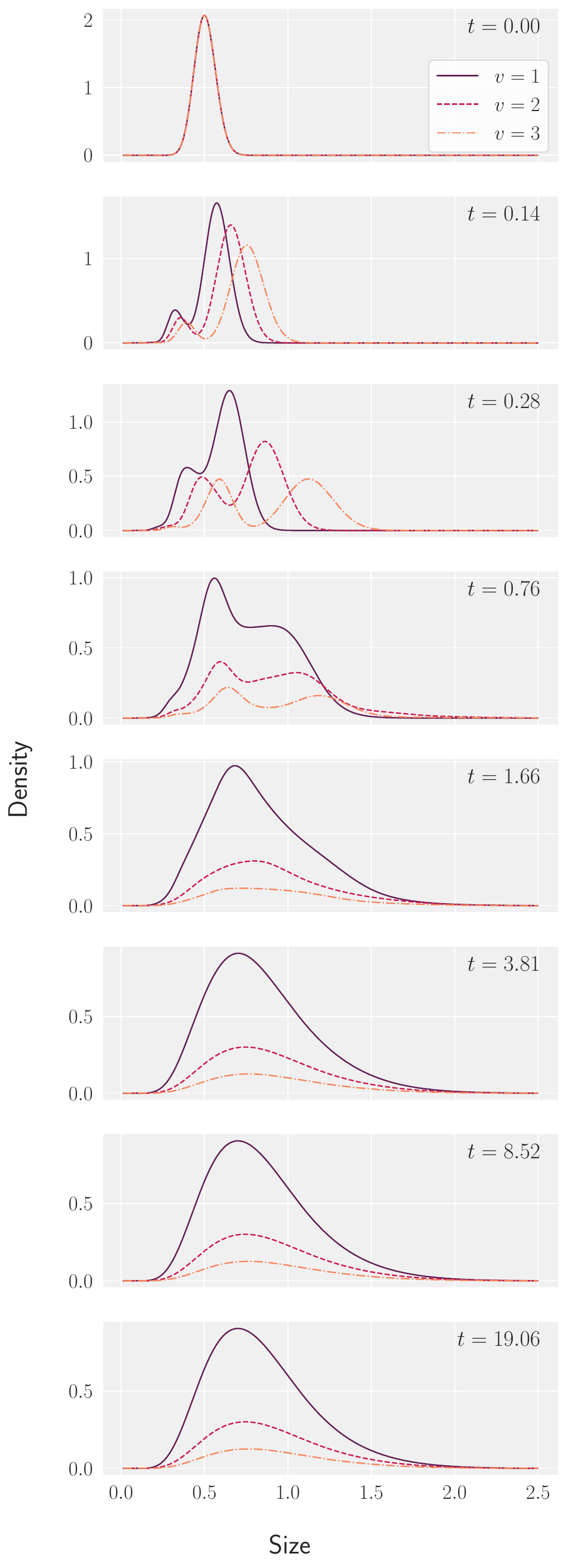}
    \caption{Mixing case}
    \label{fig:approx_n_mixing}
  \end{subfigure}
  \caption{Time evolution of the size distribution per feature, either
    the whole distribution at a few times (\textit{bottom}) or at many
    times but only at size $x=1$ (\textit{top}).  Obtained with
    coefficients $\tau (v,x) = vx$, $\gamma(v,x) = x^2 \tau(v,x) $ and
    the two variability kernels $\kappa^{red}$
    (\ref{fig:approx_n_nonmixing}) and $\kappa^{irr}$
    (\ref{fig:approx_n_mixing}).}
  \label{fig:approx_n}
\end{figure}

We besides approximate the Malthus parameter $\lambda$ solution to the
associated eigenproblem, by three different ways to control as much as
possible the approximation error (numerical values in
\Cref{table:malthus_par}). They are obtained by averaging over the
last times of computation different estimates of the instantaneous
growth rate of the population (see \Cref{fig:estimates_lambda}):
\begin{equation*}
  \lambda \approx \mfrac{ \tfrac{\dd}{\dd t}
    \big( \eiint{\cS} n_t \big)}{ \eiint{\cS} n_t } 
  \coloneqq \lambda_{n}(t), \quad
  \lambda = \mfrac{\eiint{\cS} \tau N}{\eiint{\cS} x N}
  \approx  \mfrac{\eiint{\cS} \tau n_t}{\eiint{\cS} x n_t}
  \coloneqq \lambda_{\tau}(t), \quad
  \lambda = \eiint{\cS} \gamma N
  \approx \mfrac{\eiint{\cS} \gamma n_t}{\eiint{\cS} n_t}
  \coloneqq \lambda_{\gamma}(t). 
\end{equation*}
\begin{minipage}{\textwidth}
  \centering \vspace{.4cm}%
  \includegraphics[height=.13\textheight]{%
    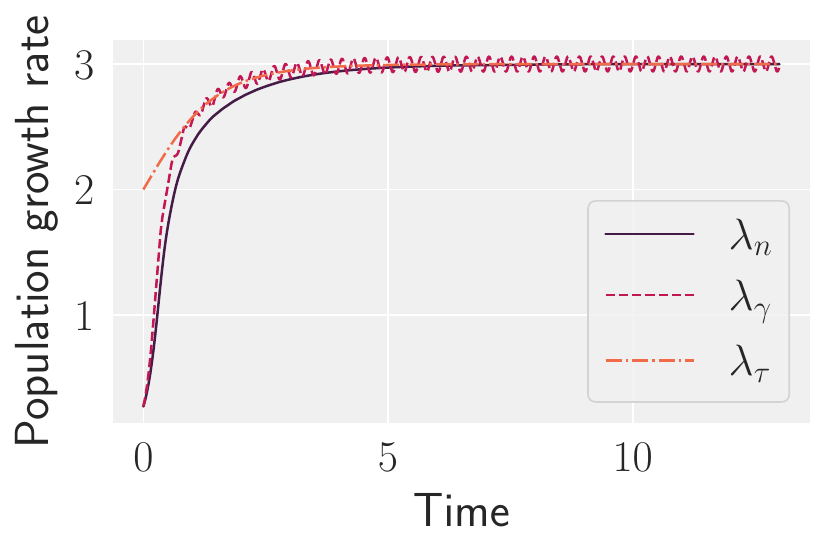}
  \includegraphics[height=.13\textheight, trim={1cm 0 0 0}, clip]{%
    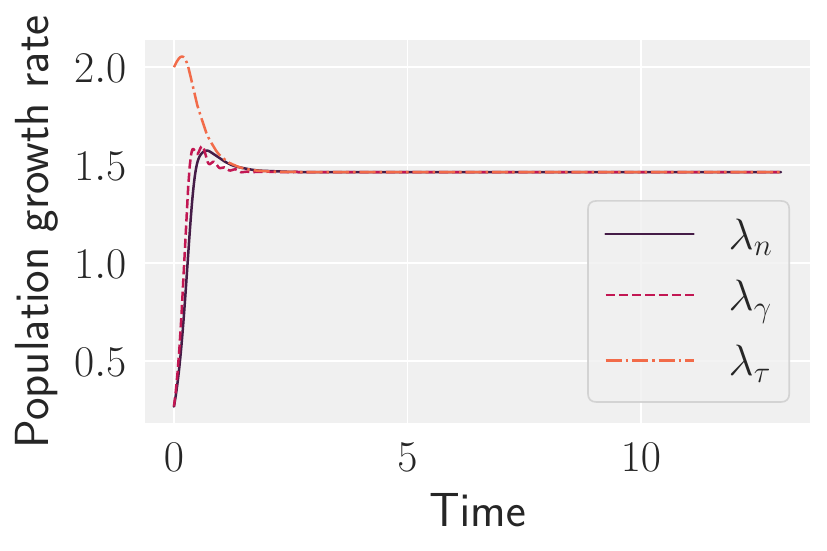}
  ~~~\includegraphics[height=.13\textheight]{%
    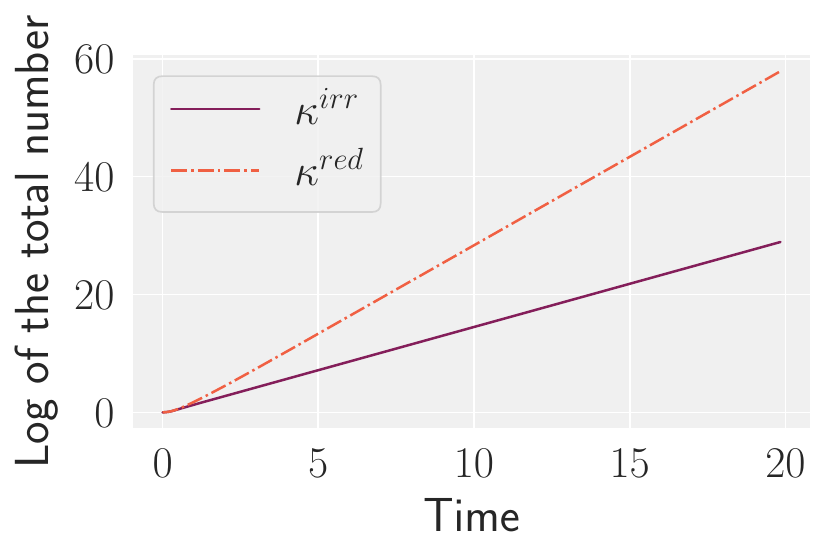}
  \captionsetup{type=figure}%
  \captionof{figure}{Time evolution of the estimates of the
    instantaneous population growth rate obtained with coefficients
    $\tau (v,x) = vx$, $\gamma(v,x) = x^2 \tau(v,x) $ and the two
    variability kernels $\kappa^{red}$ (\textit{left}) and
    $\kappa^{irr}$ (\textit{middle}). For any $t>0$, $\lambda_n(t)$ is
    obtained by linear regression of the log of the total number,
    \smash{$s \mapsto \ln \big( \int \! \! \! \int_{\cS} n(s,v,x) \,
      \dd v \dd x \big)$} (\textit{right}), on the time interval
    $\big[ \frac{t}{2}, t \big]$. \label{fig:estimates_lambda}}
\end{minipage}

\renewcommand{\arraystretch}{1.1}
\begin{minipage}{\textwidth} \centering
  \vspace{.4cm}
  \begin{tabular}{|l|c|c|c|}
    \cline{2-4}
    \multicolumn{1}{c|}{}
    & $\lambda_n$ & $\lambda_\tau$ & $\lambda_\gamma$ \\
    \hline
    \sf{Non-mixing} & 3.005 & 3.000  & 3.007\\
    \hline
    \sf{Mixing} & 1.469 & 1.465 & 1.470\\
    \hline
  \end{tabular}
  \captionsetup{type=table} \captionof{table}{Estimations of the
    Malthus parameters through different
    methods. \label{table:malthus_par}}
\end{minipage}

As expected, we retrieve the Malthus parameter
$\lambda^{red} \approx \lambda_{max} = v_{max}$ that corresponds to a
population where all cells have same feature $v_{max}$ in the case
{of} no mixing; and $\lambda^{irr} \in (v_{min}, v_{max})$,
{with} $v_{min}$ and $v_{max}$ the extremum values of $\cV$, in
the case of mixing.

A deeper study of the variation of the Malthus parameter with respect
to variation in $\kappa$ or $\cV$ would benefit to the understanding
of mixing mechanisms. In agreement with Olivier's numerical
results~\cite{olivier_how_2017}, note that in the case of a
\emph{homogeneous kernel} (such that
\smash{$\kappa_{ij}= \frac{1}{M}$} for all $i,j \in \{1, \ldots M\}$)
and the setting described above for the other parameters, the
estimated Malthus parameter (around $1.47$), is lower that the average
on the features ($\frac{1}{M}\sum_i v_i=2$) --which are the Malthus
parameters of the subpopulations grown independently.

\subsection{Discussion of the irreducibility condition}
\label{ssec:discussion_irr}

The previous examples actually {illustrate two \enquote{extreme} cases
  where either the system is not coupled, or it is \enquote{fully}
  coupled through an irreducible $\kappa$.}  In the present
subsection, we consider cases (in dimension $M=2$ for simplicity) that
are {neither {irreducible} nor completely decoupled}, with variability
kernels of the form:
\begin{equation*}
  \kappa^{_{S t F}} (p) \coloneqq
  \begin{pmatrix}
    p & 1-p \\
    0 & 1
  \end{pmatrix} \quad \text{and} %
  \quad \kappa^{_{F t S}} (p) \coloneqq
  \begin{pmatrix}
    1 & 0  \\
    1-p & p
  \end{pmatrix}, \qquad p \in \cP \subset (0,1),
\end{equation*}
still with $\tau : (v,x) \mapsto vx$ and
$\gamma : (v,x) \mapsto v x^3$.
Both $\kappa^{_{S t F}}$ and $\kappa^{_{F t S}}$ model populations of
two species (with traits $v_1 < v_2$) of which only one can mutate into
the other one; the slowest species for~$\kappa^{_{S t F}}$ (\emph{Slow
  to Fast}) and the fastest for~$\kappa^{_{F t S}}$ (\emph{Fast to
  Slow}).

We start by a useful result that will lead our intuition. Consider the
usual setting with no variability but assume that at birth cells can
either die (or definitely disappear/be removed from the population) or
survive with probability $p \in (0, 1]$:
\begin{equation} \label{eq:GF_with_death}
  \left\{
    \begin{aligned}
      &\tfrac{\dd}{\dd{x}} \bigl( \tau(x) N(x) \bigr)
        + \bigl( \lambda_p + \gamma(x) \bigr) N(x)
        = 4 p \gamma(2x) N(2x),\\
      &\tau(0) N(0) = 0,
        \qquad \mint N  = 1, \qquad  N \geq 0.
    \end{aligned}
  \right.
\end{equation}
How is the Malthus parameter $\lambda_p$ affected by the death rate
$p$? When $p=\frac{1}{2}$ we recover the conservative case: in
average, only one daughter cell is kept at each division hence the
total number of cell is conserved and $\lambda_{\frac{1}{2}}$ is
zero. When $\tau = \tau_0x$ we know that $\lambda_1$ is $\tau_0$. The
following proposition provides a general expression of $\lambda_p$ in
this case.
\begin{prop}
  \label{prop:lambda_with_death}
  Assume that $\tau : x \mapsto \tau_0 x$. Then the Malthus parameter
  $\lambda_p$ solution to~\eqref{eq:GF_with_death} is given by
  \begin{equation*}
    \lambda_p = \tau_0 \big( \log_2(p) + 1 \big),
    \qquad p \in (0, 1].
  \end{equation*}
\end{prop}

\begin{proof}
  Integrating~\eqref{eq:GF_with_death} against $x^\alpha$, for
  $\alpha \in \R$ (this is legitimate
  from~\cite[Theorem~1]{doumic_jauffret_eigenelements_2010}), brings
  us
  \begin{equation*}
    -\alpha \mint_0^\infty x^{\alpha-1} \tau(x) N(x) \dd{x}
    + \lambda_p \mint_0^\infty x^\alpha N(x) \dd{x}
    = \lp \mfrac{p}{2^{\alpha-1}} - 1 \rp \mint_0^\infty x^\alpha
    \gamma(x) N(x) \dd{x}.
  \end{equation*}
  Therefore, when $\tau$ is linear we get
  \begin{equation*}
    \lambda_p = \tau_0 \alpha +
    \lp \mfrac{p}{2^{\alpha-1}} - 1 \rp
    \mfrac{\eint{0}{\infty} \! x^\alpha \gamma(x) N(x) \dd{x}}{
      \eint{0}{\infty} \! x^{\alpha} N(x) \dd{x}}, \qquad
    \alpha \in \R,  
  \end{equation*}
  and it suffices to chose $\alpha$ so that to cancel the second term
  of the right-hand side to obtain the expected result.
\end{proof}

\Cref{prop:lambda_with_death} allows us to properly define the
proportion $p_0$ of newborn cells kept in subpopulation $2$ (of
fast-growing cells with trait $v_2$) that is necessary for it to grow
asymptotically exactly as fast as subpopulation $1$ (with trait $v_1$)
when both are considered independently:
\begin{equation}
  \label{eq:p_0}
  v_1 = v_2 \big( \log_2(p_0) + 1 \big)
  \qquad \Longleftrightarrow \qquad
  p_0 \coloneqq 2^{\frac{v_1}{v_2} - 1}.
\end{equation}

Numerical simulations, presented in
\Cref{fig:approx_n_new,fig:approx_n_at_fixed_x_lim}, suggest that the
long-time asymptotic behavior depends on which species is able to
mutate as follows:
\begin{itemize}[leftmargin=*, wide=0cm, topsep=0cm, itemsep=0cm]
\item When the slowest species is able to mutate --this is
  $\kappa = \kappa^{_{S t F}}$-- then its contribution to the
  exponentially growing subpopulation with trait $v_2 > v_1$ is
  negligible and the system asymptotically behaves as if there was
  only subpopulation~$2$ (exponential growth at rate $\lambda = v_2$
  and cyclic behavior, \Cref{fig:approx_n_new_1}).
\item When the fastest species is the one that can mutate, the
  asymptotics depends on~$p$. There exists $p_0 \in (0.5, 1)$, most
  likely given by~\eqref{eq:p_0}
  (see~\Cref{fig:approx_n_at_fixed_x_lim}),
  such that
  \begin{itemize}[leftmargin=.5cm,parsep=0cm,itemsep=0.1cm,topsep=0cm]
  \item Case $p < p_0$. If $p > 0.5$, subpopulation~$2$ grows but more
    slowly than subpopulation~$1$ alone (that would grow alone
    exponentially at rate $v_1$). Otherwise this is even clearer since
    subpopulation $2$ loses (or conserves if $p=0.5$) cells over time
    up to extinction ($p<0.5$). In both cases, the contribution of
    subpopulation with trait $v_2$ is rapidly negligible and the
    system asymptotically behaves as if there was only
    subpopulation~$1$ (exponential growth at rate $v_1$ and cyclic
    behavior, \Cref{fig:approx_n_at_fixed_x_lim}-\textit{(left)} and
    \ref{fig:approx_n_new_2}).
  \item Case $p > p_0$. Although it loses cells, subpopulation~$2$
    grows exponentially at~rate
    \begin{equation}
      \label{eq:lambda_2_p}
      \lambda_{v_2,p} \coloneqq v_2 \big( \log_2(p) + 1 \big)
    \end{equation}
    superior to the rate at which would grow subpopulation~$1$
    alone. Thus asymptotically, both grow exponentially at the rate
    $\lambda_{v_2,p}$ and oscillate at a frequency imposed by the
    subpopulation with trait $v_2$ (see
    \Cref{fig:approx_n_at_fixed_x_lim}-\textit{(right)} and
    \ref{fig:approx_n_new_3}).
  \end{itemize}
\end{itemize}
Because our scheme is diffusive for the feature $v_1$, it might not be
clear whether a seemingly convergence of $n_1 \exp^{-\lambda t}$
towards a steady state is due to numerical {diffusion} or if it
corresponds to the theoretical behavior. To avoid ambiguity we limit
the diffusion by taking the refined size-grid defined by $N=25001$,
$k = 2000$.

\begin{minipage}{.58\textwidth}
  \centering
  \renewcommand{\arraystretch}{1.}
    \begin{tabular}{|l|c|c|c|c|}
      \cline{2-5}
      \multicolumn{1}{c|}{}
      & $\cP$& $N_1$& $N_2$& $\lambda$\\
      \hline
      $\kappa^{_{S t F}}(p)$
      & $(0, 1)$& $0$& periodic& $v_2$\\
      \hline
      \multirow{2}*{$\kappa^{_{F t S}}(p)$}
      & $(0, p_0)$& periodic& $0$& $v_1$\\
      \cline{2-5}
      & $(p_0, 1)$& periodic & periodic &$\lambda_{v_2,p}$\\
      \hline
    \end{tabular}
\end{minipage}
\begin{minipage}{.39\textwidth}
  \vspace{.4cm}%
  \captionsetup{type=table} \captionof{table}{Conjecture of the
    long-time behavior when $\kappa$ is $\kappa^{_{S t F}}(p)$ or
    $\kappa^{_{F t S}}(p)$ depending on $p \in \cP$. The value of
    $p_0$ and $\lambda_{v_2,p}$ are conjectured to be defined
    by~\eqref{eq:p_0} and \eqref{eq:lambda_2_p}, respectively.
    \label{table1:conjecture}}
\end{minipage}

\begin{minipage}{\textwidth}
  \centering %
  \includegraphics[height=0.17\textheight]{%
    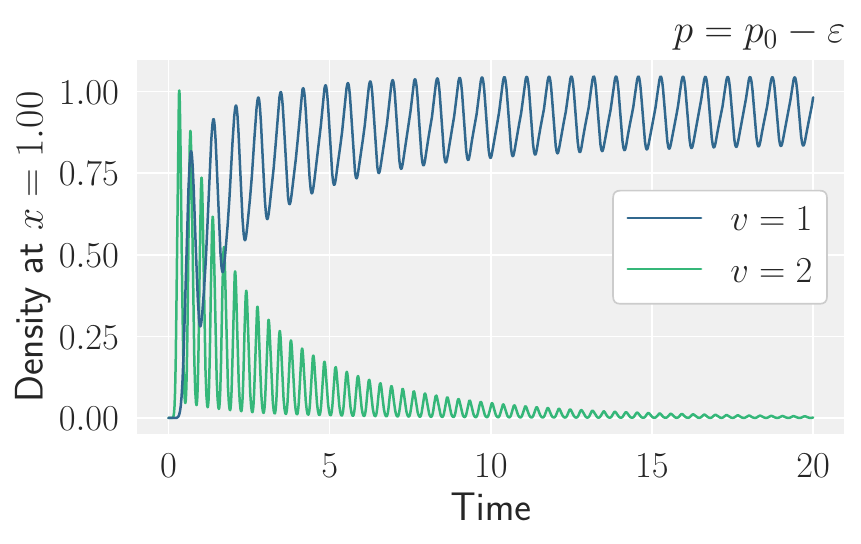}
  ~~~\includegraphics[height=.17\textheight, trim={1cm 0 0 0}, clip]{%
    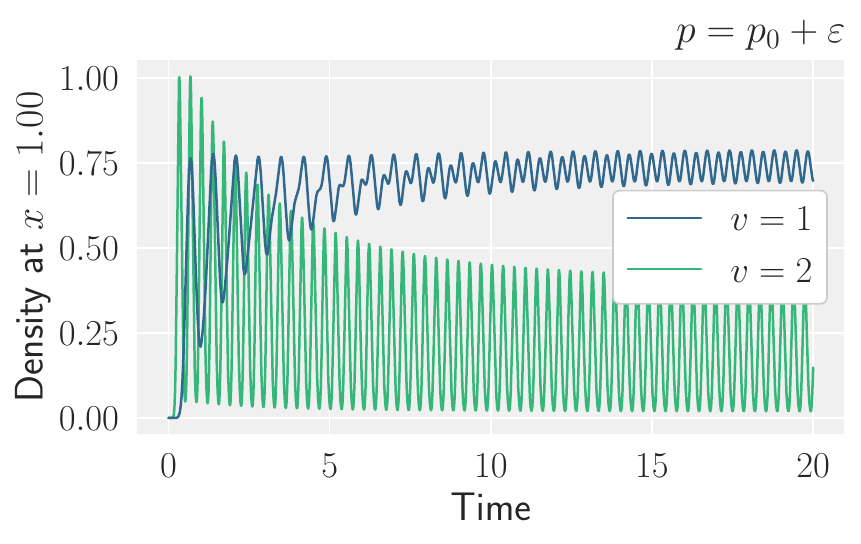}
  \captionsetup{type=figure}%
  \vspace{-.2cm}%
  \captionof{figure}{Time evolution of the size distribution per
    feature at size $x=1$, for the variability kernels
    $\kappa^{_{F t S}}(p_0 \pm \ep)$, with $p_0$ defined
    by~\eqref{eq:p_0} and $\ep = 0.05$ to show that $p_0$ appears as a
    critical value for $p$ in the case \emph{Fastest to
      Slowest}. Obtained with the coefficients $\tau (v,x) = vx$,
    $\gamma(v,x) = x^2 \tau(v,x)
    $. \label{fig:approx_n_at_fixed_x_lim}}
\end{minipage}

\begin{figure}
  \centering
  \begin{subfigure}[b]{0.32\textwidth}
    \centering%
    ~~~\includegraphics[height=0.14\textheight, trim={0cm 0cm 0 0},
    clip]{%
      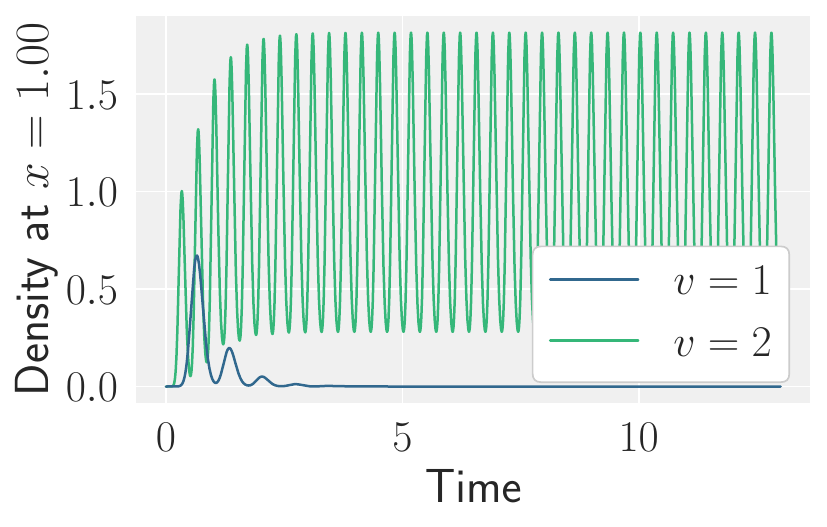}
    
    \includegraphics[height=0.65\textheight]{%
      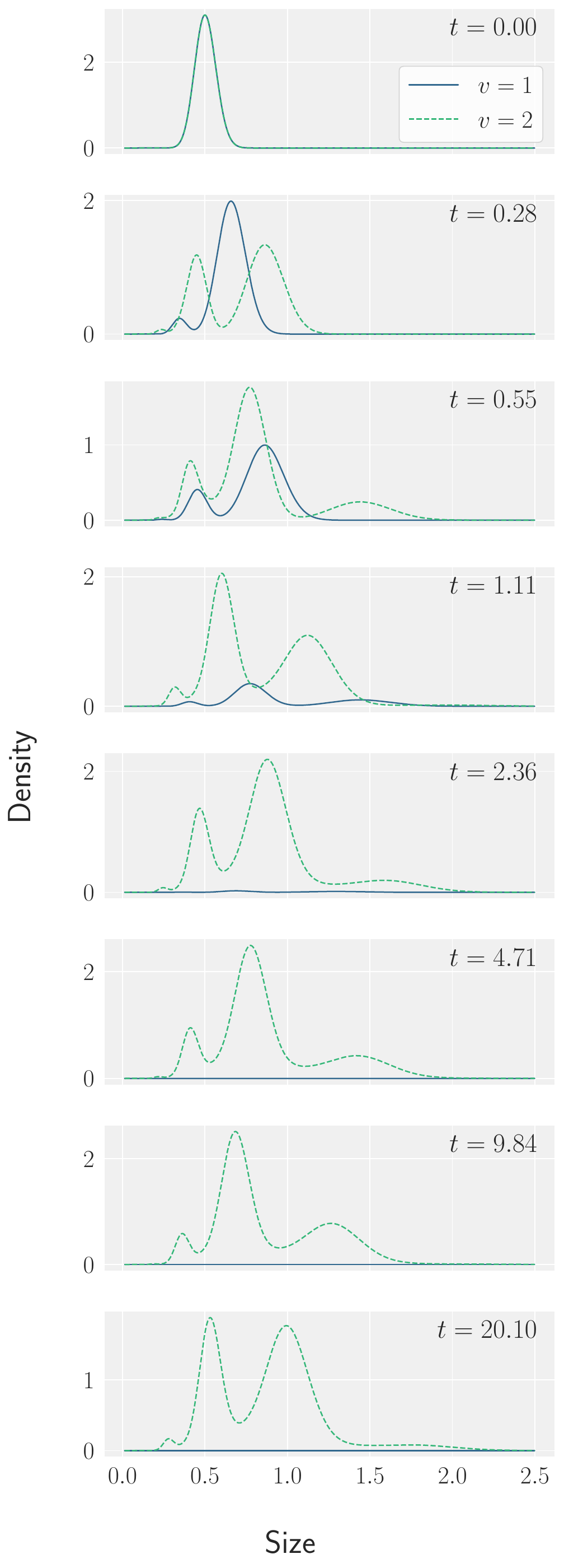}
    \caption{Slowest to fastest \newline $p=0.5$}
    \label{fig:approx_n_new_1}
  \end{subfigure}
 \hfill
  \begin{subfigure}[b]{0.31\textwidth}
    \centering%
    ~~~~\includegraphics[height=.14\textheight, trim={1cm 0cm 0 0},
    clip]{%
      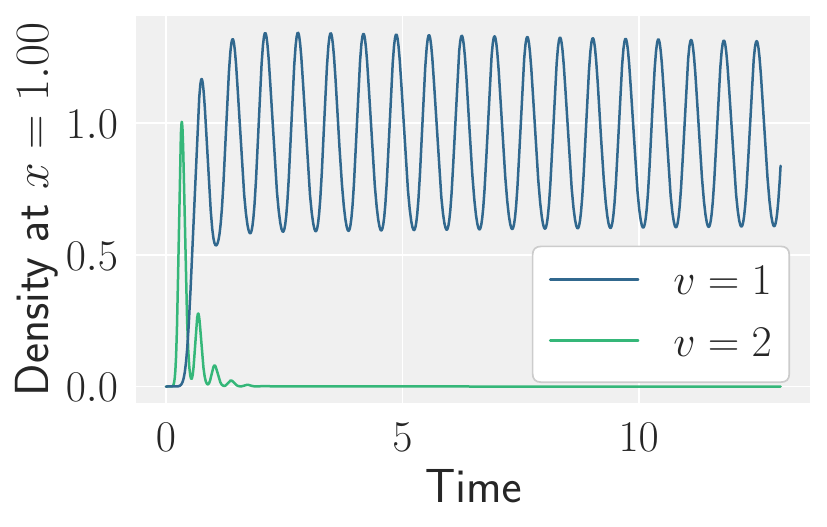}
    
    ~~~\includegraphics[height=.65\textheight,%
    trim={1.5cm 0 0 0}, clip]{%
      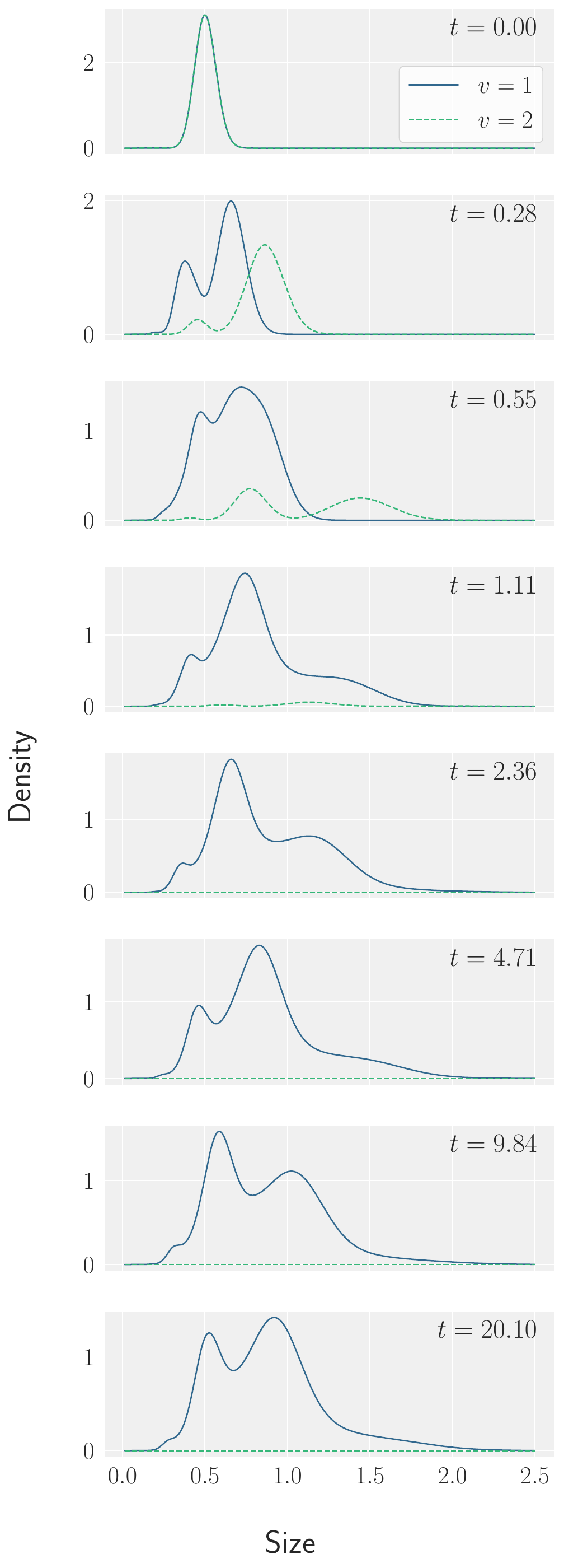}
    \caption{Fastest to lowest \newline $p=0.2$}
    \label{fig:approx_n_new_2}
  \end{subfigure}
  \hfill
  \begin{subfigure}[b]{0.31\textwidth}
    \centering%
    ~~~\includegraphics[height=.14\textheight, trim={1cm 0cm 0 0},
    clip]{%
      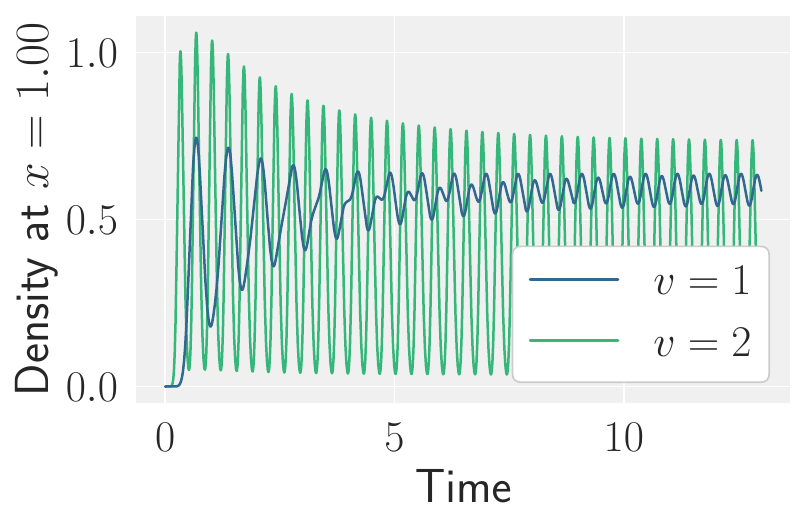}

    \includegraphics[height=.65\textheight,%
    trim={1.5cm 0 0 0}, clip]{%
      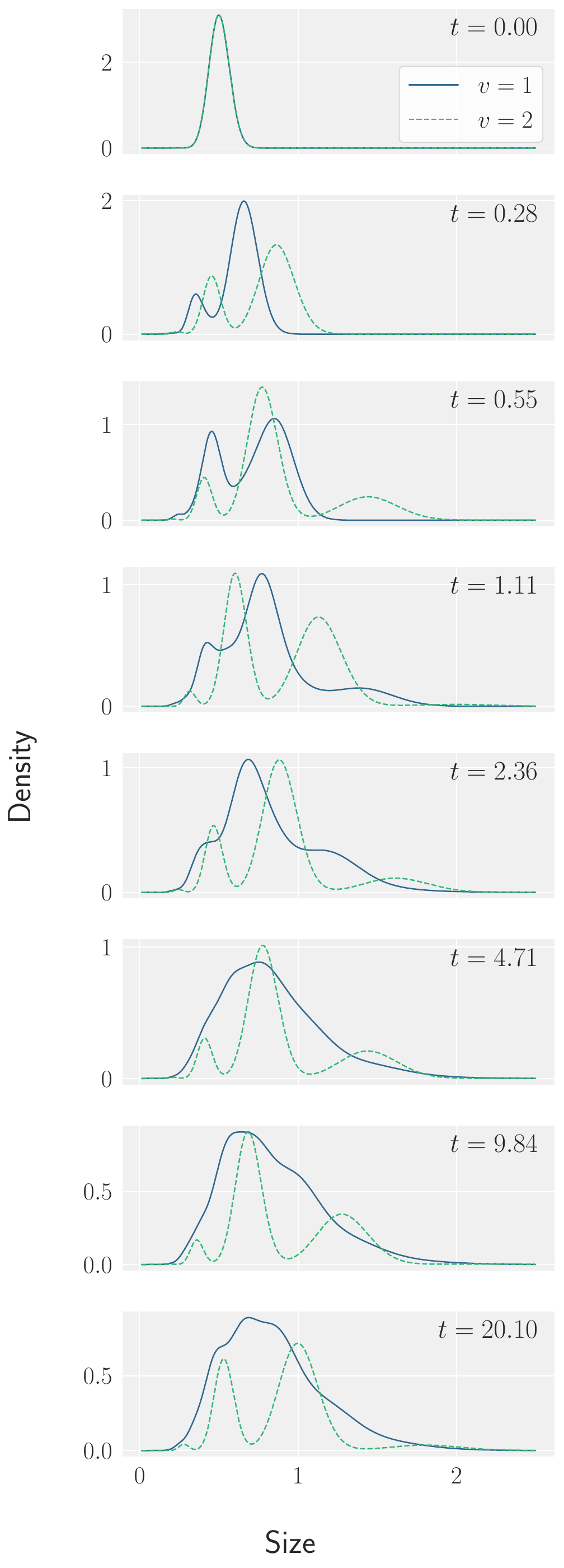}
    \caption{Fastest to lowest \newline $p=0.8$}
    \label{fig:approx_n_new_3}
  \end{subfigure}
  \caption{Time evolution of the size distribution per feature, either
    the whole distribution at a few times (\textit{bottom}) or at many
    times but only at size $x=1$ (\textit{top}). Obtained with
    coefficients $\tau (v,x) = vx$, $\gamma(v,x) = x^2 \tau(v,x) $ and
    the variability kernels %
    $\kappa^{_{S t F}}(0.5)$, 
    $\kappa^{_{F t S}}(0.2)$ 
    and $\kappa^{_{F t S}}(0.8)$. 
    The corresponding Malthus parameters were estimated to be (up to
    $10^{-3}$ precision) $v_2=2$, $v_1=1$, and
    $\lambda_{2, 0.8} \approx 1.356$, respectively, in accordance with
    the conjecture (see \Cref{table1:conjecture}).}
  \label{fig:approx_n_new}
\end{figure}

\FloatBarrier

\section{Proofs}
\label{sec:proof}

\subsection{Proof of Theorem~\ref{thm:existence_eig}}

We adapt the existence proof of~\cite[Theorem
1]{doumic_jauffret_eigenelements_2010} for the setting with no
variability to our case. It basically consists in obtaining uniform
estimates on a truncated problem, whose existence is proved through
regularization and {the} Kre\u{\i}n-Rutman theorem, as follows:
\begin{enumerate}[label={\roman*)}, leftmargin=.7cm, parsep=0cm,
  itemsep=0.1cm, topsep=0cm]
\item\label{enum:proof_eig1} we \enquote{truncate}
  problem~\ref{pb:GFv} to work with a nicer \enquote{truncated
    problem}, \ref{pb:GFv_trunc} defined on the compact set
  $\cV \times (0, R)$ and having positive boundary condition,
  
\item\label{enum:proof_eig2} {regularize~\ref{pb:GFv_trunc}}
  through convolution by a mollifier sequence
  $( \rho_\ep )_{\ep>0}$ so that {the}
  Kre\u{\i}n-Rutman {theorem} can apply and provide existence of
  eigenelements
  $(\lambda^\ep, N^\ep, \phi^\ep)$,
  
\item\label{enum:proof_eig3} derive estimates on these eigenelements
  to be able to pass to the limit $\ep \rightarrow 0$ and
  obtain eigenelements
  $(\lambda^{\eta,\delta}, N^{\eta,\delta} ,\phi^{\eta,\delta})$
  solution to the truncated problem,
    
\item\label{enum:proof_eig4} then bound
    $(\lambda^{\eta,\delta} )_{\eta,\delta>0}$ in $(0, +\infty)$ to
    get a $\lambda>0$ limit when $\eta,\delta \rightarrow 0$ (and
    $R \rightarrow \infty$),

\item\label{enum:proof_eig5} derive estimates on the moments of
  $\tau^\eta N^{\eta,\delta}$, uniform in $\eta$ and $\delta$, to be
  able to pass to the limit the direct problem of~\ref{pb:GFv},
  ~$\delta \rightarrow 0$ first, and then $\eta \rightarrow 0$,
    
\item\label{enum:proof_eig6} and similarly with the adjoint problem by
  uniform estimates on $\phi^{\eta,\delta}$.
\end{enumerate}

\begin{proof}\label{proof:existence_eig}
  \textbf{Step~\ref{enum:proof_eig1}.} In order to get compactness and
  positivity, both required to apply {the} Kre\u{\i}n-Rutman theorem,
  we first consider the problem on a bounded domain
  \[ \cS_R \coloneqq \cV \times [0,R], \qquad R>0,\] and second, endow
  the direct problem with a positive boundary condition $\delta$ and
  set the growth rate~$\tau$ to $\eta>0$ around $x=0$:
  \begin{equation*}
    \tau^\eta \coloneqq \left\{
      \begin{aligned}
        &\eta \quad &&\textrm{on~~} \cV \times [0, \eta],\\
        &\tau \quad &&\textrm{elsewhere},
      \end{aligned} \right.
  \end{equation*}
  to make it bounded from below on $\cS_R$ by the positive constant
  \begin{equation}\label{eq:def_mu}
    \mu = \mu(\eta, R)
    \coloneqq \min_{i \in \cI} \Bigl(
    \underset{(0,R)}{\infess} \tau_i^\eta \Bigr) >0.
  \end{equation}
  Following~\ref{as:tau-gamma}, we also define on $\R_+$ the functions
  $\gamma^\eta_i \coloneqq \beta \tau_i^\eta$, for $i$ in $\cI$.

  The truncation of the adjoint problem simply follows as the adjoint
  eigenproblem of the truncated direct eigenproblem endowed with null
  boundary condition at~\smash{$x=R$}, and finally the whole
  \enquote{truncated} problem states as follow: $\forall i \in \cI$,
  $\nae x \in (0,R)$,
  \begin{equation*}\label{pb:GFv_trunc}
    \tag*{\textnormal{(GF$^{\eta,\delta}_{\! v}$)}}
    \left\{~
      \begin{aligned} 
        &\mpdvx \big[ \tau_i^\eta (x) N_i^{\eta,\delta} (x) \big] +
        \big( \lambda^{\eta,\delta} + \gamma_i^\eta (x) \big)
        N_i^{\eta,\delta}(x) = 4 \msum{j \in \cI} \Big(
        \gamma_j^\eta(2x) N_j^{\eta,\delta}(2x) \kappa_{ji} \Big)
        \mathds{1}_{\tCinterval{0}{R}}(2x),\\[-4pt]
        &\tau_i^\eta(0) N_i^{\eta,\delta}(0) = \delta, \qquad \msum{j
          \in \cI} \Big( \mint_0^R \!  N_j^{\eta,\delta}(s) \dd{s}
        \Big) = 1,
        \qquad N_i^{\eta,\delta}(x) > 0,\\[-1pt]
        &-\! \tau_i^\eta(x) \mpdvx \phi_i^{\eta,\delta}(x) + \big(
        \lambda^{\eta,\delta} \! + \! \gamma_i^\eta (x) \big)
        \phi_i^{\eta,\delta}(x) = 2 \gamma_i^\eta (x) \msum{j \in \cI}
        \Big( \phi_j^{\eta,\delta} \big(\mfrac{x}{2} \big) \kappa_{ij}
        \Big)
        +\delta \phi_i^{\eta,\delta}(0),\\[-4pt]
        &\phi_i^{\eta,\delta}(R)=0, \qquad \msum{j \in \cI} \Big(
        \mint_0^R \!  N_j^{\eta,\delta}(s) \phi_j^{\eta,\delta}(s)
        \dd{s} \Big) = 1, \qquad \phi_i^{\eta,\delta}(x) \geq 0.
      \end{aligned}
    \right.
  \end{equation*}  
  From now on, $\delta$ and $\eta$ being to be brought to zero, we
  choose them lower than $1$.

  \textbf{Step~\ref{enum:proof_eig2}.} Besides compactness and
  positivity, one needs to work in a space whose positive cone has
  non-empty interior to apply the strong form of
  Kre\u{\i}n-Rutman's theorem~\cite{perthame_transport_2007,
    dautray_mathematical_1999}. In particular, $L^p$ spaces are not
  appropriate but the space of continuous functions is.
  Let us thus regularize~\ref{pb:GFv_trunc} in order to work in
  $ \sC ([0,R])^{\texp{M}}$ endowed with the norm
  in~$\ell^\infty \big(\cI ; \,\sC([0,R]) \big)$ or equivalently
  in~$\ell^1 \big(\cI ; \,\sC([0,R]) \big)$:
  \begin{equation}\label{eq:definition_norms}
    \lN n \rN_{\ell^\infty} \coloneqq \max_{i \in \cI} \lN n_i
    \rN_{\infty}, \qquad
    \lN n \rN_{\ell^1} \coloneqq \msum{i \in \cI} \lN n_i \rN_{\infty},
  \end{equation}
  where $\lN \cdot \rN_\infty$ refers to the supremum norm in
  $\sC([0,R])$.  We define, for
  $\lp \rho_\ep \rp_{\ep>0}$ a sequence of mollifiers
  and $\, * \,$ the convolution product\footnote{For any
    $L^1_{loc}(\R)$ function $f$, the convolution product $*$ is
    defined for all $\varphi$ in $\sC_c(\R)$ (the space of continuous
    compactly functions on $\R$), as
    $f * \varphi : x \mapsto \int_{-\infty}^{+\infty} f(s)
    \varphi(s-x) \dd{s}$. For $f$ only defined on $\R_+$,
    $f * \varphi$ makes sense as the truncation to $\R_+$ of the
    convolution between $f$ extended to $\R$ by zero and $\varphi$ in~
    $\sC_c (\R_+)$.}:
  \begin{gather*}
    \beta^\ep \coloneqq \rho_\ep * \beta, \; \qquad
    \mathds{1}^\ep_{\tCinterval{0}{R}} \coloneqq \rho_\ep *
    \mathds{1}_{\tCinterval{0}{R}} \; \qquad \tau_i^\ep \coloneqq
    \rho_\ep * \tau_i^\eta, \quad \gamma^\ep_i
    \coloneqq \beta^\ep \tau^\ep_i, \quad \forall i \in \cI,
  \end{gather*}
  omitting $\eta$ to lighten notations. The sequence is chosen with
  $\textrm{supp}(\rho_\ep) \subset [-\ep, 0]$, for all
  positive $\ep$, such that
  $\mathds{1}^\ep_{\tCinterval{0}{R}}$ cancels on
  $(R, \textstyle{+\infty )}$ (and is $1$ on
  $[0,\textstyle{R-\ep]}$), which gives sense to the
  regularized truncated equation~\ref{pb:GFv_trunc_reg} defined only
  for $x \leq R$. We also mention
  \begin{equation}\label{eq:prop_convolution_cv}
    \beta^\ep \underset{\ep \rightarrow
      0}{\longrightarrow} \beta
    \quad L^1(0,R),
    \qquad \; \tau_i^\ep
    \underset{\ep \rightarrow 0}{\longrightarrow}
    \tau_i^\eta \quad L^1(0,R), \quad \forall i \in \cI,
  \end{equation}
  from the properties of the convolution with
  mollifiers~\cite{brezis_functional_2011}, 
  as well as the inequalities:
  \begin{equation}\label{eq:prop_convolution_ineq}
    \Norm{\beta^\ep}{L^1(0,R)}  \leq 2
    \Norm{\beta}{L^1(0,R)} \text{~$\ep$ small enough}, \qquad
    \mu \leq \tau^\ep_i \leq \Norm{\tau^\eta}{L^
      \infty(\cS_R)}, \; \forall i \in \cI.
  \end{equation}
  In particular $(\tau^\ep)_{\ep>0}$, bounded in
  $L^\infty([0,R])$, converges in $L^\infty([0,R])$ for the weak*
  topology (towards $\tau^\eta$) so that we can also prove
  $L^1$-convergence for $(\gamma^\ep)_{\ep > 0}$:
  $\forall i \in \cI$,
  \begin{equation*}
    \Norm{\gamma_i^\ep-\gamma_i}{L^1(0,R)} \leq
    \Norm{\tau^\eta}{L^\infty(\cS_R)}
    \Norm{\beta^\ep-\beta}{L^1(0,R)} + \Norm{\beta
      (\tau_i^\ep-\tau_i^\eta)}{L^1(0,R)}
    \underset{\ep \rightarrow 0}{\longrightarrow} 0.
  \end{equation*}
  Conditions are now gathered to use {the} Kre\u{\i}n-Rutman theorem
  and derive the existence of a solution
  $(\lambda^\ep,N^\ep,\phi^\ep)$ to the following regularized problem
  (written omitting $\eta$ and $\delta$, again for simplicity). The
  proof is left to Subsection~\ref{ssec:proof_existence}.
  \begin{theorem}\label{thm:existence_trunc_reg} Under
    assumptions~\ref{as:tau}~\ref{as:gamma}~\ref{as:beta}
    and~\ref{as:kappa_cv}, for $\,\delta M R < \mu$ and
    $\ep >0$, there exists a unique solution
    $(\lambda^\ep,N^\ep,\phi^\ep) \in \R
    \times \sC^1 ([0,R])^{\texp{M}}\times \sC^1([0,R])^{\texp{M}} $ to
    the regularized eigenproblem: $\forall i \in \cI$,
    $\forall x \in [0,R)$,
    \begin{equation*}\label{pb:GFv_trunc_reg}
      \tag*{\textnormal{(GF$^{\ep}_{\!\scriptscriptstyle v}$)}}
      \left\{~
        \begin{aligned}
          &\mpdvx \big[ \tau_i^\ep(x) N_i^\ep (x)
          \big] + \big( \lambda^\ep + \gamma_i^\ep (x)
          \big) N_i^\ep(x) = 4 \msum{j \in \cI} \big(
          \gamma_j^\ep(2x) N_j^\ep (2x) \kappa_{ji}
          \big) \mathds{1}^\ep_{\tCinterval{0}{R}}(2x),\\[-4pt]
          &\tau_i^\ep N_i^\ep(0) = \delta, \qquad
          \msum{j \in \cI} \Big( \mint_0^R N_j^\ep(s) \dd{s}
          \Big) = 1,
          \qquad N_i^\ep(x) > 0,\\[-1pt]
          &-\tau_i^\ep(x) \mpdvx \phi_i^\ep(x) + \big(
          \lambda^\ep + \gamma_i^\ep(x) \big)
          \phi_i^\ep(x) = 2 \gamma_i^\ep (x) \msum{j
            \in \cI} \phi_j^\ep \big( \mfrac{x}{2} \big)
          \kappa_{ij}
          +\delta \phi_i^\ep(0),\\[-4pt]
          &\phi_i^\ep(R)=0, \qquad \msum{j \in \cI} \Big(
          \mint_0^R N_j^\ep(s) \phi_j^\ep(s) \dd{s}
          \Big) = 1, \qquad \phi_i^\ep(x) > 0.
        \end{aligned}
      \right.    
    \end{equation*}
  \end{theorem}

  From now on, we arbitrarily impose on the truncation parameters to
  satisfy
  \begin{equation}\label{eq:cond_parametres_trunc}
    2 \delta M R = \mu
  \end{equation}
  so that the condition $\,\delta M R < \mu$ holds.  As a result, it
  should be noted that for $\eta$ fixed and $\delta$ tending toward
  zero, $R$ tend to infinity. This is not completely direct since
  $\mu = \mu(\eta, R)$ does not only depend on $\eta$. To be
  convinced, one can notice that~\ref{as:tau_in_inf} ensures:
  \begin{equation*}
    \exists \omega \geq 0, \;  C>0, \; A > \eta \quad :
    \quad \min_{i \in \cI} \Bigl( \infess_{(A,R)}
    \tau_i^\eta \Bigr) \geq \frac{C}{R^{\omega}}, \quad \forall R>A,
  \end{equation*}
  while~\ref{as:tau_positive} provides $\tau=\tau^\eta$ with a lower
  bound on $\cV \times [\eta,A]$, say
  \smash{$m_{\tCinterval{\eta}{A}}$}. At the end, together with
  condition~\eqref{eq:cond_parametres_trunc} we have
  $\, 2\delta M R = \mu > \min \bigl( \eta, m_{\tCinterval{\eta}{A}},
  \frac{C}{R^\omega} \bigr)$ making things clear.
  
  It follows that $R$ can be seen not as a parameter anymore but as
  functions of parameters~$\eta$ and $\delta$, and then $\mu$ as well
  from~\eqref{eq:def_mu}: $R = R(\eta,\delta)$,
  $\mu = \mu(\eta,\delta)$.
  
  \textbf{Step~\ref{enum:proof_eig3}.} To find a solution to
  problem~\ref{pb:GFv_trunc} it remains to bring~$\ep$ to zero
  in the weak formulation of~\ref{pb:GFv_trunc_reg}. We rely on the
  following estimates, derived in
  Subsection~\ref{ssec:proof_KR_estimates}.
  \begin{lemma}\label{lemma:unif_bound_eigenels_varep}
    For $\ep$, $\delta$ and $\eta$ fixed positive, there
    exists constants $\lambda_{up}$, $N_{low}$, $N_{up}$, and
    $\phi_{up}$, depending on $\eta$ and $\delta$ but
    \emph{independent of $\ep$}, such that:
    $\forall i \in \cI$, $\forall x \in [ 0, R ]$,
    \begin{equation*}
      0 < \lambda^\ep \leq \lambda_{up},
      \qquad  0 < N_{low} \leq N_i^\ep(x) \leq N_{up},
      \quad 0 \leq \phi_i^\ep(x)
      \leq \phi_{up}.
    \end{equation*}
  \end{lemma}

  \textit{Eigenvalue.} From $(\lambda_\ep)_\ep$
  bounded in $(0, +\infty)$ we extract a subsequence, still denoted
  $(\lambda_\ep)_\ep$ for simplicity, that converges
  to some non-negative $\lambda^{\eta, \delta}$.

  \textit{Eigenvectors.} Similarly, since for all $i$ in $\cI$
  families $(N_i^\ep)_{\ep>0}$ and
  $(\phi_i^\ep)_{\ep>0}$ are bounded in
  $L^\infty(0,R)$ 
  we can extract from $(N^\ep)_{\ep>0}$
  and~$(\phi^\ep)_{\ep>0}$ (through $M$ successive
  extractions) subsequences, denoted the same, that converge
  component-wise weakly* in $L^\infty (0,R)$:
  \begin{equation}\label{eq:CV_faible-*_N-phi}
    \mint_0^R  N_i^{\ep}\varphi  \;
    \underset{\ep \rightarrow 0}{ \longrightarrow}
    \; \mint_0^R N_i \varphi,
    \qquad  \mint_0^R \phi_i^{\ep} \varphi
    \; \underset{\ep \rightarrow 0}{ \longrightarrow} \;
    \mint_0^R \phi_i \varphi, \qquad \forall \varphi \in L^1(0,R),
    \quad \forall i \in \cI.
  \end{equation}

  Let us check that $(\lambda, N, \phi)$ is (weak) solution
  to~\ref{pb:GFv_trunc}:
  \begin{enumerate}[leftmargin=.7cm, parsep=0cm, itemsep=0.1cm,
    topsep=0cm]
  \item From the bound from below on
    $(N^{\ep})_{\ep>0}$ and
    $(\phi^{\ep})_{\ep>0}$ we have that $N$ positive,
    bounded from below by $N_{low}$, and $\phi$ non-negative.
 
  \item We prove that $N$ satisfies the direct equation
    of~\ref{pb:GFv_trunc}. As classical solution $N^\ep$~is
    weak solution to~\ref{pb:GFv_trunc_reg}:
    \smash{$\forall \varphi \in \sC_c^\infty\big( \lbrack 0,R )
      \big)^{\texp{M}}$},
    \begin{multline}\label{pb:GFv_trunc_rg_WF}
      \msum{i \in \cI} \mint_0^R N_i^\ep (x) \Big(-
      \tau_i^\ep(x) \mpdvx \varphi_i(x) +
      \big(\lambda^\ep + \gamma_i^\ep(x) \big)
      \varphi_i(x) \Big) \dd{x}
      - \delta \msum{i \in \cI} \varphi_i (0)\\
      = 2 \msum{i \in \cI} \mint_0^R \gamma_i^\ep(x)
      N_i^\ep (x)
      \mathds{1}^{\ep}_{\tCinterval{0}{R}}(x) \bigg( \msum{j
        \in \cI}\kappa_{ij} \varphi_j\Big(\mfrac{x}{2} \Big) \bigg)
      \dd{x},
    \end{multline}
    and we show that the limit $\ep \rightarrow 0$ exits and
    is nothing but the weak formulation of the direct problem
    of~\ref{pb:GFv_trunc}. For the first term we have:
    \begin{multline*}
      \bigg\lvert \msum{i \in \cI} \mint_0^R \Big(
      \tau_i^{\ep} N_i^{\ep} - \tau_i^\eta N_i \Big)
      \mpdvx \varphi_i \,\bigg\rvert \leq \msum{i \in \cI} \Big(
      \Norm{\tau_i^{\ep}-\tau_i^\eta}{L^1(0,R)} \Big) \lN
      \mpdvx \varphi \rN_{\ell^\infty} N_{up} \\ + \bigg\lvert \msum{i
        \in \cI} \mint_0^R \tau_i^\eta \big( N_i^{\ep} - N_i
      \big) \mpdvx \varphi_i \,\bigg\rvert,
    \end{multline*}
    with $\tau_i^\eta \mpdvx \varphi_i \ \in L^1(0,R)$, making clear,
    with~\eqref{eq:prop_convolution_cv}
    and~\eqref{eq:CV_faible-*_N-phi}, that right-hand side tends
    towards zero as $\, \ep$ does.
    The convergence also holds for the other terms of the left-hand
    side of~\eqref{pb:GFv_trunc_rg_WF} since
    $(\lambda^{\ep})_{\ep>0}$ converges towards
    $\lambda^{\eta,\delta}$
    and~$(\gamma^{\ep})_{\ep>0}$ towards
    $\gamma^\eta$. 
    As for the right-hand term, we perform the same kind of
    computations relying on~\ref{as:kappa_cv}:
    \begin{multline*}
      \bigg\lvert \msum{i \in \cI} \mint_0^R \Big(
      \gamma_i^\ep(x) N_i^\ep (x)
      \mathds{1}^{\ep}_{\tCinterval{0}{R}}(x) -
      \gamma_i^\eta(x) N_i (x) \Big) \msum{j \in \cI}\kappa_{ij}
      \varphi_j\Big(\mfrac{x}{2} \Big) \dd{x} \bigg\lvert\\
      \leq \lN \varphi \rN_{\ell^\infty} N_{up} \msum{i \in \cI} \Big(
      \Norm{\gamma_i^{\ep} - \gamma_i^\eta}{L^1(0,R)} +
      \Norm{\gamma_i^\eta}{L^1 (R-\ep, R)} \Big) + \msum{i \in
        \cI} \mint_0^R ( N_i^{\ep} - N_i )
      \tilde{\varphi}_i^\eta,
    \end{multline*}
    with
    \smash{$\tilde{\varphi}_i^\eta \coloneqq \gamma_i^\eta
      \sum_{j}\kappa_{ij} \varphi_j \big(\frac{\cdot}{2} \big) \in
      L^1(0,R)$}, and again all terms tend to zero when~$\ep$
    does thanks to~\eqref{eq:prop_convolution_cv},
    $\gamma_i^\eta \in L^1_{loc}(\R_+^*)$
    and~\eqref{eq:CV_faible-*_N-phi}.
    
  \item Similarly one can prove that $\phi$ satisfies the adjoint
    equation of~\ref{pb:GFv_trunc}.

  \item Plus, testing for all $i$, $\varphi \equiv 1$ against
    $N_i^\ep$ in~\eqref{eq:CV_faible-*_N-phi}, and summing for
    $i$ in $\cI$, brings the expected normalization condition on $N$.

    As for $\phi$, we have
    $1 = \sum_i \int N_i^\ep \phi_i^\ep \leq N_{up}
    \sum_i \int \phi_i^\ep $ so $\phi$ is non-zero and
    satisfies the normalization condition up to renormalization.
  \end{enumerate}

  \textbf{Step \ref{enum:proof_eig4}. \textit{Limit as
      $\delta, \eta \rightarrow 0$ for $\lambda^{\eta,\delta}$. }}
  \hspace{-3pt}We compare $\lambda^{\eta,\delta}$ with the
  \emph{positive} eigenvalues~$\lambda_{min}^{\eta,\delta}$ and
  $\lambda_{max}^{\eta,\delta}$ solution to the reformulation
  of~\ref{pb:GFv_trunc} in {the} absence of variability, when all
  cells have feature $v_{min} \coloneqq \min (\cV)$ or
  $v_{max} \coloneqq \max(\cV)$, respectively.
  The existence and positivity of such eigenvalues in ensured by the
  proof of~\cite[Theorem~1]{doumic_jauffret_eigenelements_2010} in the
  mitosis case. Mimicking the proof of~\Cref{thm:monotonicity_V}
  (working on the truncated domain $\cS_R$ with a slightly different
  formulation in $\delta$ and $\eta$ does not change computations) we
  get:
  \begin{equation}\label{eq:bounds_lambda_eta_delta}
    0 < \lambda_{min}^{\eta,\delta} \leq \lambda^{\eta,\delta} \leq
    \lambda_{max}^{\eta,\delta}, \qquad  \eta, \delta >0.
  \end{equation}
  The proof of~\cite[Theorem~1]{doumic_jauffret_eigenelements_2010}
  also provides positive $\lambda_{min}^\eta$, $\lambda_{max}^\eta$,
  $\lambda_{min}$, and~$\lambda_{max}$, such that
  \begin{equation*}
    \lambda_{min}^{\eta,\delta} \underset{\delta \rightarrow
      0}{\longrightarrow} \lambda_{min}^\eta,
    \quad \lambda_{max}^{\eta,\delta} \underset{\delta \rightarrow
      0}{\longrightarrow} \lambda_{max}^\eta, \quad  \eta>0,
    \qquad \lambda_{min}^{\eta} \underset{\eta \rightarrow
      0}{\longrightarrow} \lambda_{min},
    \quad \lambda_{max}^{\eta} \underset{\eta \rightarrow
      0}{\longrightarrow} \lambda_{max}.
  \end{equation*}
  Passing to the limit $\delta \rightarrow 0$ and $\eta \rightarrow 0$
  successively in~\eqref{eq:bounds_lambda_eta_delta} yields to the
  existence of subsequences $(\lambda^{\eta,\delta})_{\eta,\delta>0}$
  and $(\lambda^\eta)_{\eta>0}$, denoted the same, and $\lambda$ such
  that:
  \begin{equation*}
    \lambda^{\eta,\delta} \underset{\delta \rightarrow
      0}{\longrightarrow} \lambda^\eta>0, \quad  \eta >0,
    \qquad \, \lambda^{\eta} \underset{\eta \rightarrow
      0}{\longrightarrow} \lambda >0.
  \end{equation*}

  \textbf{Step \ref{enum:proof_eig5}. \textit{Limit as
      $\delta \rightarrow 0$ for $N^{\eta,\delta}$}.} Let us fix
  $\eta >0$ and recall that $\delta$ brought to zero brings $R$ to
  infinity. To pass to the limit $\delta \rightarrow 0$, we bound the
  moments of $\tau^\eta N^{\eta,\delta}$
  in~$\ell^1 \big(\cI ;\, W^{1,1}([0,R]) \big)$, uniformly in
  $\delta$.
  To do so, we start bounding
  $(x^\alpha \gamma^\eta N^{\eta,\delta})_{\delta>0}$, 
  $\alpha \geq 0$:
  \begin{itemize}[leftmargin=.5cm, parsep=0cm, itemsep=0.1cm,
    topsep=0cm]
  \item First for $\alpha \geq m \coloneqq \max(2, \omega_0+1)$, with
    $\omega_0$ defined by~\ref{as:tau_in_0}. Multiplying the direct
    equation in~\ref{pb:GFv_trunc} by $x^\alpha$ and integrating on
    $\cS_R$ brings with a change of variables:
    \begin{equation}\label{eq:SV_xalpha_int_SR}
      \begin{aligned}
        \msum{i \in \cI} R^\alpha \tau_i^\eta(R) &
        N_i^{\eta,\delta}(R) - \alpha \msum{i \in \cI} \mint_0^R
        x^{\alpha-1} \tau_i^\eta(x) N_i^{\eta,\delta} (x) \dd{x} +
        \lambda^{\eta,\delta} \msum{i \in \cI} \mint_0^R x^\alpha
        N_i^{\eta,\delta}
        (x) \dd{x} \\
        &+ \msum{i \in \cI} \mint_0^R x^\alpha \gamma_i^\eta(x)
        N_i^{\eta,\delta}(x) \dd{x} = \mfrac{1}{2^{\alpha-1}} \msum{i
          \in \cI} \mint_0^R x^\alpha \gamma_i^\eta (x)
        N_i^{\eta,\delta}(x) \dd{x},
      \end{aligned}
    \end{equation}
    hence the inequality
    \begin{equation*}
      \Big( 1- \mfrac{1}{2^{\alpha-1}} \Big)
      \msum{i \in \cI}
      \mint_0^R  x^\alpha \gamma_i^\eta(x)  N_i^{\eta,\delta}(x) 
      \dd{x} \leq \alpha  \msum{i \in \cI} \mint_0^R
      x^{\alpha-1} \tau_i^\eta(x) N_i^{\eta,\delta} (x) \dd{x}.
    \end{equation*}
    According to assumption~\ref{as:beta_in_inf} there exists
    $A_{\alpha}\geq \eta$ such that
    \begin{equation*}
      \tau_i^\eta(x) = \tau_i(x) \leq \frac{x}{\alpha 2^\alpha}
      \gamma_i^\eta (x),
      \qquad  x \geq A_\alpha,
      \quad \forall i \in \cI.
    \end{equation*}
    By definition $x \mapsto x^{\omega_0} \tau_i(x)$ is essentially
    bounded from above on a neighborhood of $0$ and so is
    $\tau^\eta$. Combined with assumption~\ref{as:tau_in_0}, we obtain
    that for $\eta$ small enough
    \begin{equation*}
      \bigNorm{x^{\omega_0} \tau_i^\eta}{L^\infty(0,A_\alpha)}
      \leq \bigNorm{x^{\omega_0} \tau_i}{L^\infty(0,A_\alpha)}
      < + \infty, \qquad \forall i \in \cI.
    \end{equation*}
    From all these considerations, it follows that for all $i$ in
    $\cI$:
    \begin{multline*}
      \hspace{-10pt}
      \mint_0^R x^{\alpha-1} \tau_i^\eta(x) N_i^{\eta,\delta}(x) \dd{x}
      \leq A^{\alpha-1-\omega_0}_\alpha \mint_0^{A_\alpha} \! 
      x^{\omega_0} \tau_i^\eta(x) N_i^{\eta,\delta}(x) \dd{x} +
      \mint_{A_\alpha}^R \! 
      x^{\alpha-1} \tau_i^\eta(x) N_i^{\eta,\delta}(x) \dd{x}\\
      \leq A_\alpha^{\alpha-1-\omega_0} \, \bigNorm{x^{\omega_0}
      \tau_i }{L^\infty(0,A_\alpha)}
      \mint_0^{A_\alpha} \! N_i^{\eta,\delta}(x) \dd{x} +
      \mfrac{1}{\alpha 2^{\alpha}} \mint_{A_\alpha}^R x^{\alpha}
      \gamma_i^\eta(x) N_i^{\eta,\delta}(x) \dd{x}.
    \end{multline*}
    We conclude using the normalization condition on
    $N^{\eta,\delta}$: for all $\eta > 0$ small and $\alpha \geq m$
    \begin{equation}\label{eq:moment_gammaN_1}
      \msum{i \in \cI} \mint_0^R x^\alpha
      \gamma_i^\eta(x) N_i^{\eta,\delta}(x) \dd{x}
      \leq \alpha \Big( 1- \mfrac{3}{2^{\alpha}} \Big)^{-1}
      A^{\alpha-1-\omega_0}_\alpha \,
      \bigNorm{ x^{\omega_0} \tau }{\ell^1 ( \cI; L^\infty
        (0,A_\alpha))}
      \coloneqq B_\alpha.
    \end{equation}
   
  \item Then for $0 \leq \alpha < m$. To extend estimates to smaller
    $\alpha$ we make sure there is no problem around $x=0$ focusing on
    bounding
    \smash{$( \sum_{i} \tau_i^\eta N_i^{\eta,\delta})_{\delta>0}$}
    essentially around zero so the moments of
    $\sum_{i} \gamma_i^\eta N_i^{\eta,\delta}=\beta \sum_{i}
    \tau_i^\eta N_i^{\eta,\delta}$ can be bounded as well around $0$.
    Let fix $\rho$ in $(0,\tfrac{1}{2})$ and define $x_\rho>0$ (lower
    than $R$ for $R$ big) as the unique point such that
    \begin{equation}\label{eq:def_rho}
      \int_0^{x_\rho} \beta(x) \dd{x} \coloneqq \rho, 
    \end{equation}
    which is well defined since $\beta$ is non-negative integrable
    around zero from~\ref{as:beta_L0}. Integrating~\ref{pb:GFv_trunc}
    on sizes lower than any $x \in (0, x_\rho)$ and traits, gives:
    \begin{equation*}
      \begin{aligned}
        \msum{i \in \cI} &\tau_i^\eta(x) N_i^{\eta,\delta} (x) \leq
        \delta M + 2 \msum{i \in \cI} \mint_0^{2x} \gamma_i^\eta(s)
        N_i^{\eta,\delta}(s)
        \mathds{1}_{\tCinterval{0}{R}}(s) \dd{s}\\
        &\leq \delta M + 2 \, \BigNorm{ \msum{i \in \cI} \tau_i^\eta
          N_i^{\eta,\delta}}{L^\infty(0,x_\rho)} \mint_0^{x_\rho} \!
        \beta(s) \dd{s} + \frac{2}{x_\rho^m} \msum{i \in \cI}
        \mint_{x_\rho}^R s^m \gamma_i^\eta(s) N_i^{\eta,\delta}(s)
        \dd{s}.
      \end{aligned}
    \end{equation*}
    Remembering that $\delta$ has been taken inferior to $1$, it
    follows that
    \begin{equation*}
      \BigNorm{ \msum{i \in \cI}
        \tau_i^\eta N_i^{\eta,\delta}}{L^\infty (0,x_\rho)}
      \leq \frac{1}{1-2 \rho} \lp M + \frac{2 }{x_\rho^m} B_m 
      \rp \coloneqq D_0.
    \end{equation*}
    Now let us go back to
    $\big(\sum_i \int_0^R x^\alpha \gamma_i^\eta N_i^{\eta,\delta}
    \big)_{\delta, \eta}$ for $0\leq \alpha < m$. We have
    \begin{equation}\label{eq:moment_gammaN_2}
      \begin{aligned}
        \msum{i \in \cI} \mint_0^R x^\alpha \gamma_i^\eta(x)
        N_i^{\eta,\delta}(x) \dd{x} &\leq D _0\mint_0^{x_\rho} x^\alpha
        \beta(s) \dd{s} + x_\rho^{\alpha-m} \msum{i \in \cI}
        \mint_{x_\rho}^R x^{m}
        \gamma_i^\eta(x) N_i^{\eta,\delta}(x) \dd{x} \\
        &\leq D_0 \rho \, x_\rho^\alpha + x_\rho^{\alpha-m} \coloneqq
        B_\alpha.
      \end{aligned}
    \end{equation}
  \end{itemize}
  At the end, combining~\eqref{eq:moment_gammaN_1}
  and~\eqref{eq:moment_gammaN_2} brings
  \begin{equation*}
    \forall \alpha \geq 0, \; \exists B_\alpha :
    \quad \msum{i \in \cI} \mint_0^R x^\alpha
    \gamma_i^\eta(x) N_i^{\eta,\delta} (x) \dd{x} \leq B_\alpha,
    \quad \delta, \eta > 0.
  \end{equation*}
  
  Finally, we can control
  $\big( x^\alpha \sum_i\tau_i^\eta N_i^\eta \big)_{\delta}$, with
  $\alpha >-1$, in $L^1(0,R)$. Using, from the definition of
  $\gamma^\eta$ and~\ref{as:beta_in_inf}, the fact that there exists
  $\hat{x} >\eta$ such that:
  \begin{equation*}
    \tau_i^\eta(x) \leq x \, \gamma_i^\eta(x),
    \qquad {\nae} x \geq \hat{x},
    \quad \forall i \in \cI,
  \end{equation*}
  we have that for $R>\hat{x}$ (which is satisfied for $\delta$ small
  enough)
  \begin{equation*}
    \begin{aligned}
      \msum{i \in \cI} \mint_0^R x^\alpha \tau_i^\eta(x)
      N_i^{\eta,\delta} (x) \dd{x} &\leq \msum{i \in \cI}
      \int_0^{\hat{x}} x^\alpha \tau_i^\eta(x) N_i^{\eta,\delta}(x)
      \dd{x} + \msum{i \in \cI} \int_{\hat{x}}^R x^{\alpha+1}
      \gamma_i^\eta(x) N_i^{\eta,\delta}(x) \dd{x}\\
      &\leq D_0 \mfrac{\hat{x}^{\alpha+1}}{\alpha+1}+ B_{\alpha+1}
      \coloneqq C_\alpha,
    \end{aligned}
  \end{equation*}
  hence a $l^1 (\cI ; \, L^1)$-bound for
  $\big( x^\alpha \tau^\eta N^\eta \big)_{\eta,\delta}$ :
  \begin{equation}\label{eq:moment_tau_N}
    \forall \alpha > -1, \; \exists C_\alpha :
    \quad \msum{i \in \cI} \mint_0^R
    x^\alpha \tau_i^\eta(x) N_i^{\eta,\delta}(x) \dd{x}
    \leq C_\alpha, \quad \delta, \eta > 0.
  \end{equation}
  
  To conclude to a $l^1(\cI;\,W^{1,1})$ bound we need an estimate on
  the derivative. Relying on equation~\ref{pb:GFv_trunc} again
  and~\eqref{eq:SV_xalpha_int_SR}, we get: for all $\alpha > 0$,
  \begin{equation*}
    \begin{aligned}
      \msum{i \in \cI} &\mint_0^R \Big\lvert \mpdvx \Big( x^\alpha
      \tau_i^\eta(x) N_i^{\eta,\delta} (x) \Big) \Big\rvert \dd{x}
      \leq \alpha \msum{i \in \cI} \mint_0^R \!  \Big( x^{\alpha-1}
      \tau_i^\eta(x) N_i^{\eta,\delta}(x) + x^{\alpha} \Big\lvert
      \mpdvx \big( \tau_i^\eta N_i^{\eta,\delta}\big)(x)
      \Big\rvert \Big) \dd{x}\\
      &\leq \alpha \bigg( C_{\alpha-1} + \msum{i \in \cI} \mint_0^R
      x^\alpha \big( \lambda^{\eta,\delta}+ \gamma_i^\eta(x) \big)
      N^{\eta,\delta}_i(x) \dd{x} + \mfrac{1}{2^{\alpha-1}} \msum{i
        \in \cI} \mint_0^Rx^\alpha
      \gamma_i^\eta(x) N_i^{\eta,\delta}(x) \dd{x} \bigg)\\
      &\leq \alpha \bigg( C_{\alpha-1} + \msum{i \in \cI} \mint_0^R
      x^{\alpha-1} \tau_i^\eta(x) N_i^{\eta,\delta} (x) \dd{x} +
      \mfrac{1}{2^{\alpha}} \msum{i \in \cI} \mint_0^Rx^\alpha
      \gamma_i^\eta(x) N_i^{\eta,\delta}(x) \dd{x} \bigg)\\
      & \leq
      \alpha \big( 2 C_{\alpha-1} + 2^{-\alpha}B_\alpha \big),
    \end{aligned}
  \end{equation*}
  and a similar control holds when $\alpha$ is zero according the
  direct equation of~\ref{pb:GFv_trunc}, so that at the end, together
  with~\eqref{eq:moment_tau_N}, we have:
  \begin{equation}\label{eq:moment_tau_N_W11}
    \forall \alpha \geq 0,
    \quad  \big( x^\alpha \tau^\eta N^{\eta,\delta} \big)_{\eta,\delta>0}
    \text{~bounded in~} \ell^1 \big( \cI; \, W^{1,1}(\R_+) \big).
  \end{equation}

  We deduce that
  $( x^\alpha \tau^\eta N^{\eta, \delta} )_{\delta >0}$, for every
  $\alpha \geq 0$, belongs to a compact set of
  $ \ell^1 \big( \cI; \, L^1(\R_+) \big)$. Indeed by a diagonal
  argument we can extract from it (after M successive extractions) a
  subsequence (denoted identically) such that for every $i \in \cI$
  and all positive~$X$
  \smash{$(x^\alpha \tau_i^\eta N_i^{\eta, \delta})_{\delta >0}$}
  converges strongly in $L^1([0, X])$. Denote by
  $H^{\eta} \in \ell^1 \big( \cI; \, L^1(\R_+) \big)$ the limit when
  $\alpha =0$, then
  \begin{equation*}
    \begin{aligned}
      \msum{i \in \cI} \int_0^{\infty} \big| x^\alpha \tau_i^\eta
      N_i^{\eta,\delta} - x^\alpha H_i^{\eta} \big| &\leq \msum{i \in
        \cI} \Big( \int_0^X x^\alpha \big| \tau_i^\eta
      N_i^{\eta,\delta} -H_i^{\eta} \big| + \int_X^{\infty} x^\alpha
      \tau_i^\eta N_i^{\eta,\delta} +
      \int_X^{\infty} x^\alpha  H_i^{\eta} \Big)\\
      &\leq \bigNorm{ x^\alpha \big( \tau^\eta N^{\eta,\delta}
        -H_i^{\eta} \big) }{\ell^1 \lp \cI; \, L^1([0, X]) \rp} +
      \mfrac{2}{X} C_{\alpha + 1},
    \end{aligned}
  \end{equation*}
  where $\int_X^\infty x^\alpha H^\eta$ has be bounded by
  $ \frac{1}{X} C_{\alpha + 1}$ thanks to Fatous's lemma. Thus, for
  any $\ep >0$ we can find $X$ big enough such that the last
  term of the right-hand side is less than
  $\frac{\ep}{2}$. The first term will be as well for all
  small $\delta$ since convergence is strong in
  $\ell^1 \big( \cI; \, L^1([0, X]) \big)$ and we conclude to the
  strong $\ell^1 \big( \cI; \, L^1(\R_+) \big)$-convergence.

  We have now everything gathered to pass to the limit, first in
  $\delta \rightarrow 0$ and $\eta >0$ fixed to get rid of the
  positive boundary condition.  The same argument holds for
  \smash{$\big((1+x^\alpha) \tau^\eta N^{\eta,\delta}\big)_{\delta
      >0}$}: we can extract a subsequence that converges
  component-wise in $L^1(\R_+)$ to some $H^{\eta,\alpha}$.  However,
  for some $\omega \geq 0$ and all~$i \in \cI$,
  $x \mapsto (1+x^\omega)\tau_i^\eta(x)$ is bounded from below by a
  positive constant (from $\tau^\eta_i$ being equal to $\eta>0$ on
  $[0,\eta]$ and~\ref{as:tau_positive} and~\ref{as:tau_in_inf}) and
  therefore we deduce
  \begin{equation*}
    N_i^{\eta,\delta}  \overset{L^1}{\underset{\delta \rightarrow
        0}{\longrightarrow}} N_i^\eta
    \coloneqq \frac{H^{\eta,\omega}_i}{(1+x^{\omega})
      \tau_i^\eta}, \qquad \forall i \in \cI.
  \end{equation*}
  Passing to the limit (in the weak sense) in the
  equation~\ref{pb:GFv_trunc} on~$N^{\eta,\delta}$ we find:
  $\forall i \in \cI$, $x \in (0,+\infty)$,
  \begin{equation*}\label{pb:GFv_trunc_limdelta}
    \tag*{\textnormal{(GF$^{\eta}_{\! v}$)}}
    \left\{~
      \begin{aligned}
        &\mpdvx \big[ \tau_i^\eta (x) N_i^\eta (x) \big] + \big(
        \lambda^{\eta} + \gamma_i^\eta (x) \big) N_i^{\eta}(x) = 4
        \msum{j \in \cI} \Big( \gamma_j^\eta(2x)
        N_j^{\eta}(2x) \kappa_{ji} \Big),\\
        & N_i^{\eta}(0) = 0, \qquad \msum{j \in \cI} \Big( \mint_0^R
        \!  N_j^{\eta}(s) \dd{s} \Big) = 1, \qquad N_i^{\eta}(x) \geq
          0.
      \end{aligned}
    \right.
  \end{equation*}

  \textbf{\textit{Limit as $\eta \rightarrow 0$ for $N^\eta$.}}  All
  estimates~\eqref{eq:SV_xalpha_int_SR}-\eqref{eq:moment_tau_N_W11}
  remain true for delta $\delta = 0$. If they ensure that
  \smash{$( x^\alpha \tau_i^\eta N_i^{\eta} )_{\eta >0}$}, for all
  $i \in \cI$, belongs to a compact set of $L^1(\R_+)$ (same arguments
  than for
  \smash{$( x^\alpha \tau_i^\eta N_i^{\eta, \delta} )_{\delta >0}$})
  not necessarily $(N_i^{\eta})_{\eta >0}$ anymore, since the limit
  $\tau_i$ of $(\tau_i^\eta)_{\eta>0}$ can vanish at zero.
  Let us focus on proving the $L^1$-weak convergence of
  $(N_i^{\eta})_{\eta >0}$ first, which is equivalent by {the}
  Dunford-Pettis theorem to prove that $(N_i^\eta)_{\eta>0}$ is
  equi-integrable, bounded in $L^1(\R_+)$.
  
  \begin{itemize}[leftmargin=.5cm, parsep=0cm, itemsep=0.1cm,
    topsep=0cm]
  \item \textit{Around x=0.} We establish $L^\infty_0$-bounds for
    $( \sum_i x^{\nu_0} \tau_i^\eta N_i^{\eta} )_{\eta}$, and thus
    $( x^{\nu_0} \tau_i^\eta N_i^{\eta} )_{\eta}$,
    with $\nu_0 \geq 0$ defined by~\ref{as:tau_in_0}.
    Integrating~\ref{pb:GFv_trunc_limdelta} on $(0, x')$, for
    $x' \leq x \in \R_+$, and summing on $i \in \cI$ yields:
    \begin{equation*}
      \msum{i \in \cI} \tau_i^\eta(x') N_i^\eta(x')
      \leq 2 \msum{i \in \cI} \mint_0^{2x'}
      \gamma_i^\eta(s) N_i^\eta(s) \dd{s} \leq 2
      \msum{i \in \cI} \mint_0^{2x}
      \gamma_i^\eta(s) N_i^\eta(s) \dd{s} \leq 2 B_0.
    \end{equation*}
    We introduce
    $f^\eta \colon x \mapsto \Norm{\tau^\eta N^\eta}{\ell^1(\cI;
      L^\infty(0,x))}$. For $x \in \big(0, \frac{x_\rho}{2} \big)$
    ($x_\rho$ defined by~\eqref{eq:def_rho}), we find
    \begin{equation*}
      \begin{aligned}
        f^\eta(x)
        &\leq 2 \msum{i \in \cI} \Big( \mint_0^x \beta(s)
          \tau_i^\eta(s) N_i^\eta(s) \dd{s} + (2x)^{\nu_0} \mint_x^{2x}
          s^{-\nu_0} \beta(s) \tau_i^\eta(s)
          N_i^\eta(s) \dd{s} \Big)\\
        &\leq 2 \rho f^\eta(x) + 2^{\nu_0+1} x^{\nu_0} \mint_x^{x_\rho}
          s^{-\nu_0} \beta(s) f^\eta(s) \dd{s}.
      \end{aligned}
    \end{equation*} 
    Therefore, for $x$ in $\big(0, \frac{x_\rho}{2} \big)$ and for
    $c_0 \coloneqq \mfrac{2^{\nu_0+1}}{1-2\rho} >0$,
    \begin{equation*}
      x^{-\nu_0} f^\eta(x) \leq c_0
      \mint_x^{x_\rho} \beta(s)
      s^{-\nu_0} f^\eta(s) \dd{s} \coloneqq c_0 F^\eta(x),
    \end{equation*}
    and applying Gr\"{o}nwall's lemma to
    $x \mapsto x^{-\nu_0}f^\eta(x)$ finally brings: for
    $ x \in \big(0, \frac{x_\rho}{2} \big)$,
    \begin{equation*}
      x^{-\nu_0} \msum{i \in \cI} \tau_i^\eta(x) N_i^\eta(x)
      \leq c_0 F^\eta \big( \mfrac{x_{\rho}}{2} \big)
      \e^{ c_0  \scaleobj{0.65}{
          \mint_{\scaleobj{1.4}{x}}^{\scaleobj{1.2}{\frac{x_{\texp{\rho}}}{2}}}}
        \!  \beta(s) \dd{s}}
      \leq c_0 \mfrac{2^{\nu_0+1} B_0 \rho}{x_\rho^{\nu_0}}
      \e^{c_0 \rho} \coloneqq C.
    \end{equation*}
    Therefore, noticing that for all $i \in \cI$, and for $x > 0$,
    \begin{equation*}
      \mfrac{x^{\nu_0}}{\tau^\eta_i(x)}
      \leq \max \Bigl( 1, \mfrac{x^{\nu_0}}{\tau_i(x)} \Bigr)
      \coloneqq  f_i(x),
    \end{equation*}
    we find that for all $i$ in $\cI$, $N_i^\eta$ is controlled on
    $(0,\frac{x_\rho}{2})$ by $Cf_i$ that is $L^1_0$ and independent
    of $\eta$. We conclude that $(N_i^\eta)_{\eta >0}$, $i \in \cI$,
    is equi-integrable equi-bounded around size $0$.
    
    \emph{Note that from the bound on $x^{-\nu_0} f^\eta$, a bound on
      $x^{-2\nu_0} f^\eta$ can similarly be obtained, and so on for
      any $x^{-n\nu_0} f^\eta$, $n \in \N$. At the end, what we
      actually have is a uniform $L^\infty_0$-bound on
      $x^\alpha \tau_i^\eta N_i^\eta$ for all $\alpha \in \R$. The
      continuous embedding
      \smash{$W^{1,1}(\R_+) \hookrightarrow L^\infty(\R_+)$} besides
      ensures that for all $\alpha \in \R$,
      $x^\alpha \tau_i^\eta N_i^\eta$ is uniformly bounded in
      $L^\infty(\R_+)$.}

  \item \textit{Elsewhere, on intervals $[\ep, +\infty)$.} The fact
    that \smash{$( x^\alpha \tau_i^\eta N_i^{\eta} )_{\eta >0}$}, for
    all $i \in \cI$, belongs to a compact set of $L^1(\R_+)$ combined
    to \ref{as:tau_positive} and~\ref{as:tau_in_inf}, ensures that up
    to extraction $(N_i^\eta)_{\eta>0}$ converges strongly in
    $L^1([\ep, + \infty))$ for any $\ep > 0$.  In particular the
    assumptions of the Dunford-Pettis theorem are verified.
  \end{itemize}

  By Dunford-Pettis theorem we deduce that $(N^\eta)_{\eta >0}$
  converges (after $M$ successive extractions) to some $N$
  component-wise weakly in $L^1(\R_+)$.
  However for all $i \in \cI$, $(N_i^\eta)_{\eta >0}$ converges
  strongly on any interval $[\ep,+\infty)$, $\ep > 0$, thus we have:
  \begin{equation*}
    \begin{aligned}
      \mint_0^{+\infty} \abs{N_i^\eta(x)-N_i(x)} \dd{x} &=
      \mint_0^\ep \abs{N_i^\eta(x)-N_i(x)} \dd{x} +
      \mint_\ep^{+\infty} \abs{N_i^\eta(x)-N_i(x)} \dd{x}\\
      &\leq 2 {C} \mint_0^\ep f_i(x) \dd{x} +
      \mint_\ep^{+\infty} \abs{N_i^\eta(x)-N_i(x)} \dd{x}\\
      & \underset{\eta \rightarrow 0}{\longrightarrow} 2 {C}
      \mint_0^\ep f_i(x) \dd{x},
    \end{aligned}
  \end{equation*}
  with the right-hand side arbitrarily small for $\ep, \eta$
  small since $f_i$ belongs to $L^1_0$. We conclude that for all $i$
  in $\cI$, $(N_i^\eta)_{\eta>0}$, converges to $N_i$ strongly in
  $L^1(\R_+)$, and then, passing to the limit in the weak formulation
  of~\ref{pb:GFv_trunc_limdelta}, that $N$ satisfies the direct problem
  of~\ref{pb:GFv}.
  
  \textbf{Step \ref{enum:proof_eig6}. \textit{Limit as
      $\eta,\delta \rightarrow 0$ for $\phi^{\eta,\delta}$}.} We want
  to derive for some positive $k$ and all $i \in \cI$, a
  \smash{$L^\infty( \R_+, \frac{\dd{x}}{1+x^k} )$}-bound for
  \smash{$(\phi_i^{\eta,\delta})_{\eta,\delta>0}$}. Let fix $\eta$,
  $\delta$ positive such as to
  guarantee~\eqref{eq:cond_parametres_trunc}. We start by controlling
  \smash{$(\phi_i^{\eta,\delta})_{\eta,\delta>0}$} on any
  interval~$[0,A]$, for any positive $A$:
  \begin{itemize}[leftmargin=.5cm, parsep=0cm, itemsep=0.1cm,
    topsep=0cm]
  \item \textit{Bound on $[0, x_\rho]$}, with $x_\rho$ defined
    by~\eqref{eq:def_rho}.  Consider equation~\ref{pb:GFv_trunc}
    on~$\phi^{\eta,\delta}_i$, divide by $\tau^\eta_i$ and integrate
    on $[x, x_\rho]$: for all $i \in \cI$, and for
    $x \in [0, x_\rho]$,
    \begin{equation*}
      \begin{aligned}
        \phi_i^{\eta,\delta}(x)
        &\leq \phi_i^{\eta,\delta}(x_\rho) +2
          \mint_x^{x_\rho} \beta(s) \Big( \msum{j \in \cI}
          \phi_j^{\eta,\delta} \big( \mfrac{s}{2} \big) \kappa_{ij}
          \Big) \dd{s} +
          \mfrac{\delta x_\rho}{\mu} \phi_i^{\eta,\delta}(0)\\
        &\leq \max_{j \in \cI} \bigl( \phi_j^{\eta,\delta}(x_\rho)
          \bigr) + \Big( 2 \rho + \mfrac{x_\rho}{2MR} \Big) \max_{j \in
          \cI} \Bigl( \bigNorm{\phi_j^{\eta,\delta}}{L^\infty(0,
          x_\rho)} \Bigr),
      \end{aligned}
    \end{equation*}
    and thus for $R= R(\eta, \delta)$ greater than
    $R_0 \coloneqq \frac{x_\rho}{2M(1-2\rho)}$, i.e. for $\delta$ and
    $\eta$ small enough we have: for all $i \in \cI$, for
    $x \in [0,x_\rho]$,
    \begin{equation*} \phi_i^{\eta,\delta}(x) \leq C(x_\rho)
      \big\lVert \phi^{\eta,\delta} (x_\rho)
      \big\rVert_{\ell^\infty(\cI)}, \qquad C(x_\rho) \coloneqq
      \Big(1- 2 \rho - \mfrac{x_\rho}{2MR_0} \Big)^{-1} .
    \end{equation*}
  \item \textit{Bound on $[x_\rho, A]$.} The map
    \smash{$G_i^{\eta, \delta} : x \mapsto \e^{-\eint{x_\rho}{x}
        \frac{\lambda^{\eta,\delta} + \gamma_i^\eta}{\tau_i^\eta}} \!
      \phi_i^{\eta,\delta}(x)$}, $i \in \cI$, decays on~$\R_+$ from
    \ref{pb:GFv_trunc}. So for any $A> x_\rho>\eta$, we can find
    $C(A) > 0$ (independent from $\eta, \delta$) s.t.
    \begin{equation*}
      \phi_i^{\eta,\delta}(x) \leq \e^{\eint{x_\rho}{x}
        \frac{\lambda^{\eta,\delta} + \gamma_i(s)}{\tau_i(s)} \dd{s}}
      \phi_i^{\eta,\delta}(x_\rho)
      \leq C(A) \big\lVert \phi^{\eta,\delta} (x_\rho)
      \big\rVert_{\ell^\infty(\cI)},
      \qquad \forall i \in \cI, \; x \in [x_\rho,A].
    \end{equation*}
  
  \item \textit{Uniform bound on $[0, A]$.} To conclude we need a
    uniform bound for $\phi^{\eta,\delta}(x_\rho)$. Using the decay of
    $G_i^{\eta, \delta}$ and the normalization condition on
    $\phi^{\eta, \delta}$ we get: $\forall i \in \cI$,
    $x \in [0, x_\rho]$,
    \begin{equation*}
      \begin{aligned}
        1 \geq \msum{i \in \cI} \Big( \mint_0^{x_\rho}
        \phi_i^{\eta,\delta}(x) N_i^{\eta,\delta}(x) \dd{x} \Big)
        &\geq \msum{i \in \cI} \Big( \phi_i^{\eta,\delta}(x_\rho)
        \mint_0^{x_\rho} \e^{-\eint{x}{x_\rho}
          \frac{\lambda^{\eta,\delta} +
            \gamma_i^\eta(s)}{\tau_i^\eta(s)} \dd{s}}
        N_i^{\eta,\delta}(x) \dd{x} \Big)\\
        &\geq \msum{i \in \cI} \Big( \phi_i^{\eta,\delta}(x_\rho)
        \e^{- \rho} \mint_{\frac{x_\rho}{2}}^{x_\rho}
        \e^{-\eint{x}{x_\rho}\frac{2
            \lambda}{\tau_i(s)} \dd{s}} N_i^{\eta,\delta}(x) \dd{x} \Big),\\
      \end{aligned}
    \end{equation*}
    where the integral term converges towards a positive quantity as
    $\delta, \eta$ go to zero since $x_\rho> b$ and $N_i$ is positive
    on $(\frac{b}{2}, +\infty)$ as proved in
    proposition~\eqref{prop:positivity}. Thus
    \smash{$\Norm{ \phi^{\eta,\delta} (x_\rho)}{\ell^\infty(\cI)}$} is
    bounded and we can conclude
    \begin{equation}\label{eq:Linfty_0_phi_etadelta}
      \forall A >0, \; \exists C_0(A) :
      \quad \bigNorm{\phi^{\eta,\delta}}{\ell^\infty (\cI; L^\infty(0,A))}
      \leq C_0(A) \quad \eta, \delta > 0. 
    \end{equation}
  \end{itemize}

  It remains to bound \smash{$\phi^{\eta, \delta}$} uniformly by
  $C(1+x^k)$ on $\cV \times [A, +\infty)$. One prove that the adjoint
  problem of~\ref{pb:GFv_trunc} satisfies a maximum principle of the
  form of~\cite[Lemma 4.]{doumic_jauffret_eigenelements_2010}, the
  proof follows the same steps. Therefore, building for some $A_0>0$ a
  supersolution $\bar{\phi}$ of any problem~\ref{pb:GFv_trunc},
  $\eta, \delta >0$, on $\cV \times [A_0, +\infty)$, greater than
  \smash{$\phi^{\eta,\delta}$} on \smash{$\cV \times [0, A_0]$} and
  positive at $x=R(\eta,\delta)$, yields
  \smash{$\phi^{\eta,\delta} \leq \bar{\phi}$} everywhere. We look for
  a supersolution of the form
  $\bar{\varphi}_i \colon x \mapsto x^k+\theta$, $i \in \cI$, with
  positive $k, \theta$ to be determined. It must satisfy on $[A_0, R]$
  \begin{equation}\label{eq:supersol}
    - \tau_i^\eta(x) \mpdvx \bar{\varphi}_i(x) + \big(
    \lambda^{\eta,\delta} \! + \! \gamma_i^\eta (x) \big)
    \bar{\varphi}_i(x) \geq 2 \gamma_i^\eta (x)
    \msum{j \in \cI} \Big( \bar{\varphi}_j
    \bigl(\tfrac{x}{2} \bigr) \kappa_{ij} \Big)
    +\delta \phi_i^{\eta,\delta}(0), \quad i \in \cI.
  \end{equation}
  Since \smash{$\phi_i^{\eta,\delta}(0) \leq C_0(1)$} for
  $\eta, \delta >0$ (see~\eqref{eq:Linfty_0_phi_etadelta}) it is
  enough to find $k$ and $A_0 \geq \eta$ such that
  \begin{equation*}
    -k \tau_i(x)x^{k-1}  + \big(
    \lambda^{\eta,\delta} \! + \! \gamma_i (x) \big)
    \big(x^k+\theta \big) \geq 2 \gamma_i (x)
    \Big( \theta + \mfrac{x^k}{2^k} \Big) 
    + \delta C_0(1), \quad i \in \cI,
  \end{equation*}
  holds on $[A_0, +\infty)$. Dividing by $x^{k-1} \tau_i(x)$ we find
  that if
  \begin{equation}\label{eq:condition_supersol}
    \Big( 1-\mfrac{1}{2^{k-1}} \Big)x \beta (x)
    \geq k +\frac{\theta \gamma_i(x)}{x^{k-1}\tau_i(x)}
    + \frac{\delta C_0(1)}{x^{k-1} \tau_i(x)}
  \end{equation}
  is satisfied on $[A_0, +\infty)$ then~\eqref{eq:supersol}
  holds. Since $x \mapsto x \beta(x)$ goes to infinity as $x$ does and
  both $\tau_i$, $\gamma_i$ belong to $\cP_\infty$
  (assumptions~\ref{as:beta_in_inf} and~\ref{as:tau_in_inf}), there
  exists $k>0$ such that for any $\theta>0$, there is a $A_0>0$ for
  which~\eqref{eq:condition_supersol} holds true on $[A_0,+\infty)$.
  We can apply the maximal principle to
  $\bar{\phi} \coloneqq \frac{A_0}{\theta} \bar{\varphi}$ (that
  satisfies $\bar{\phi} \geq \phi_i^{\eta,\delta}$ on
  $\cV \times [0,A_0]$,~\eqref{eq:supersol} on $\cV \times [A_0, R]$
  and $\bar{\phi}_i(R) >0$, uniformly in $\eta, \delta$) to finally
  conclude:
  \begin{equation*}
    \exists k, C, \theta > 0 : \; \forall \eta, \delta >0
    \text{~small},
    \quad \phi_i^{\eta,\delta}(x)  \leq Cx^k + \theta
    \quad \forall i \in \cI, \;x \in [0, +\infty).
  \end{equation*}
  
  We obtained that $\phi_i^{\eta,\delta}$, for all $i \in \cI$, is
  uniformly bounded in $L^{\infty}_{loc}(\R_+)$, therefore
  $ \tau_i^\eta \p_x \phi_i^{\eta,\delta}$ is uniformly bounded in
  $L^{\infty}_{loc}(\R_+^*)$ (from~\ref{pb:GFv_trunc} and
  $ \gamma \in L^\infty_{loc}(\R_+^*)$ and so is
  \smash{$(\p_x \phi_i^{\eta,\delta})_{\delta, \eta}$} thanks
  to~\ref{as:tau_positive}. We can thus extract (again after M
  successive diagonal extractions) a subsequence still denoted
  \smash{$(\phi^{\eta,\delta})_{\eta,\delta>0}$} converging in
  \smash{$\sC^0(\R_+^*)^{\texp{M}}$} towards some $\phi$ and such that
  for every $i \in \cI$,
  \smash{$(\mpdvx \phi_i^{\eta,\delta})_{\eta,\delta}$} converges
  \smash{$L^\infty_{loc}(\R_+^*)$-weakly*} towards
  \smash{$\mpdvx \phi_i \in L^\infty_{loc}(\R_+^*)$} (and
  \smash{$\frac{\phi_i^{\eta,\delta}}{1 + x^k}$} towards
  \smash{$\frac{\phi_i}{1 + x^k}$} strongly in $L^1(\R_+)$).

  It remains to check that $\phi$ satisfies the adjoint equation
  of~\ref{pb:GFv} in~$L^1_{loc}(\R_+^*)$. It is clear
  from~\ref{as:gamma} that all the terms in~$\phi^{\eta,\delta}$
  converge to the expected limit in $L^1_{loc}(\R_+^*)$. As for the
  derivative term, we can use the weak* convergence to derive: for all
  $\varphi \in \cC^\infty_c( \R_+^* )$, with $\supp(\varphi) =K$ and
  $\eta \leq \min(K)$:
  \begin{equation*}
    \mint_K \abs{ \tau^\eta \mpdvx \phi_i^{\eta,\delta}
      - \tau \mpdvx \phi_i } \varphi
    \leq \mint_K \big( \mpdvx \phi_i^{\eta,\delta}
    - \mpdvx \phi_i \big) \tilde{\varphi}
    \underset{\eta, \delta \rightarrow 0}{\longrightarrow}  0,
  \end{equation*}
  for
  \smash{$\tilde{\varphi} = \tau \sgn \big(\mpdvx \phi_i^{\eta,\delta}
    - \mpdvx \phi_i \big) \varphi \in L^1(K)$}, so that the
  convergence holds in \smash{$L^1_{loc}(\R_+^*)$} as well.
  
  At the end, the normalization condition holds as a consequence of
  the $L^\infty - L^1$ convergence written as:
  \begin{equation*}
      1 = \msum{i \in \cI} \mint_0^\infty
      \mfrac{\phi_i^{\eta,\delta}}{1+x^k} (1+x^k) N_i^{\eta,\delta}
      \underset{\eta,\delta \rightarrow 0}{\longrightarrow} \, \msum{i
        \in \cI} \mint_0^\infty \mfrac{\phi_i}{1+x^k} (1+x^k) N_i =
      \msum{i \in \cI} \mint_0^\infty \phi_i N_i.
    \end{equation*}
\end{proof}

\subsection{Proof of Proposition~\ref{thm:uniqueness_eig}}

\begin{proof}
  \textit{Eigenvalue.} Consider
  $(\lambda_1, N^{\ttexp{(1)}}, \phi^{\ttexp{(1)}})$ and
  $(\lambda_2, N^{\ttexp{(2)}}, \phi^{\ttexp{(2)}})$ two solutions
  to~\ref{pb:GFv}. We have by duality that:
  \begin{equation*}
    \lambda_1 \big\la N^{\ttexp{(1)}} , \phi^{\ttexp{(2)}} \big\ra
    = \big\la -\tfrac{\dd}{\dd{x}} \big( \tau N^{\ttexp{(1)}} \big)
    + \cF N^{\ttexp{(1)}}, \phi^{\ttexp{(2)}} \big\ra
    = \big\la N^{\ttexp{(1)}}, \tau \tfrac{\dd}{\dd{x}}
    \phi^{\ttexp{(2)}} + \cF^* \phi^{\ttexp{(2)}} \big\ra
    = \lambda_2 \big\la N^{\ttexp{(1)}} , \phi^{\ttexp{(2)}} \big\ra,
  \end{equation*}
  with $\la N^{\ttexp{(1)}}, \phi^{\ttexp{(2)}} \ra$ positive thanks to
  \Cref{prop:positivity}, so that
  $\lambda_1 = \lambda_2 \coloneqq \lambda$.

  \textit{Direct eigenvector.} Now, we find that the entropy of
  $n \coloneqq N^{\ttexp{(1)}} \e^{\lambda t}$ with respect to
  $p \coloneqq N^{\ttexp{(2)}} \e^{\lambda t}$, written for
  $\psi \coloneqq \phi^{\ttexp{(1)}} \e^{- \lambda t}$ and some
  strictly convex $H$ is independent on time so that for all $t>0$,
  $D^H[n|p](t) = D^H[N^{\ttexp{(1)}} | N^{\ttexp{(2)}}] = 0$, and we
  deduce with \Cref{lemma:GRE_minimizer} that
  $N^{\ttexp{(1)}} = C N^{\ttexp{(2)}}$, for some $C >0$.  Thanks to
  the normalizing condition on $N^{\ttexp{(1)}}$ and $N^{\ttexp{(2)}}$
  we conclude to $C = 1$.

  \textit{Adjoint eigenvector.} Let $x_0$ be any positive real number
  and define
  $\phi \coloneqq C \big( \phi^{\ttexp{(2)}} - \phi^{\ttexp{(1)}}
  \big)$ with
  $C \coloneqq - \sgn \! \big( \phi^{\ttexp{(2)}}(x_0) -
  \phi^{\ttexp{(1)}}(x_0) \big)$. By linearity, $\phi$ satisfies the
  adjoint equation of~\ref{pb:GFv}. Thus for every $i \in \cI$,
  the map
  \begin{equation*}
    x \mapsto \phi_i(x) \e^{- \Lambda_i(x)}, \qquad
    \Lambda_i(x) \coloneqq \mint_{x_0}^x \mfrac{\lambda +
      \gamma_i(s)}{\tau_i(s)} \dd{s}
  \end{equation*}
  is decreasing on $[x_0, +\infty)$ (see computations of the proof
  of~\Cref{prop:positivity}) which brings
  \begin{equation*}
    \phi_{i}(x) \leq \phi_{i}(x_0) \e^{ \Lambda_i(x)}
    = - \abs{\phi_{i}(x_0)} \e^{ \Lambda_i(x)} \leq 0,
    \qquad \forall i \in \cI, \; {\forall} x \in [x_0, +\infty].
  \end{equation*}
  Besides, from the normalization condition satisfied by
  $\phi^{\ttexp{(1)}}$ and $\phi^{\ttexp{(2)}}$ we have that
  $\sum_i \int N_i \phi_i$ is zero and thus:
  \begin{equation*}
    \msum{i \in \cI } \mint_0^\infty N \abs{ \phi }
    = \msum{i \in \cI } \Big( \mint_0^{x_0} N \abs{ \phi }
    - \mint_{x_0}^\infty N \phi \Big) \underset{x_0 \rightarrow
      0}{\longrightarrow} 0,
  \end{equation*}
  which allows to conclude that $\phi$ is zero almost everywhere on
  $\supp N$, i.e.  $\phi^{\ttexp{(2)}} \equiv \phi^{\ttexp{(1)}}$ on
  $\cV \times \big[\tfrac{b}{2}, +\infty \big)$. If $b>0$,
  $\beta \equiv 0$ on $[0, b]$ thus for $j = 1, 2$:
  \begin{equation*}
    \phi_i^{\ttexp{(j)}}(x) = \phi_i^{\ttexp{(j)}} \big( \tfrac{b}{2}
    \big) \e^{- \eint{x}{b/2}
      \frac{\lambda}{\tau_i(s)} \dd{s}}, \qquad \forall i \in \cI, \;
    {\forall} x \in  \big[ 0, \tfrac{b}{2} \big],
  \end{equation*}
  and we have $\phi^{\ttexp{(2)}} \equiv \phi^{\ttexp{(1)}}$ as well
  on $\cV \times \big[ 0, \tfrac{b}{2} \big]$.
\end{proof}

\appendix

\phantomsection
\addcontentsline{toc}{section}{Appendix}
\section*{Appendix}

\renewcommand\thesubsection{\Alph{subsection}}

\subsection{Proof of Theorem~\ref{thm:existence_trunc_reg} --
  Kre\u{\i}n-Rutman}
\label{ssec:proof_existence}

The proof relies on {the} Kre\u{\i}n-Rutman theorem which generalizes
Peron-Frobenius theorem for matrices to the infinite dimension. We
refer to~\cite[Section 6.6]{perthame_transport_2007} for a similar
application,~\cite{krein_linear_1948} for the original paper and
\cite[Chapter VIII, Appendix]{dautray_mathematical_1999} for a more
recent proof and applications.

\begin{proof}
  Let $\eta$, $\delta$ and $\ep$ be fixed positive real numbers and
  omit from notations the truncation parameters $\eta$ and
  $\delta$. We work on the Banach space
  \begin{equation*}
    \big(X_{\! R}, \, \lN \cdot \rN_{X_R} \big) \coloneqq
    \big(\sC ([0,R])^{\texp{M}}, \, \lN \cdot \rN_{\ell^\infty} \big)
  \end{equation*}
  whose positive cone
  $ \big\{ f \in X_{\! R} \mid f_i \geq 0, \; \; \forall i \in \{ 1,
  \ldots, M \} \big\} $ has non-empty interior. We recall that
  $\lN \cdot \rN_{\ell^\infty}$ is defined
  by~\eqref{eq:definition_norms}.

  \textbf{Direct problem.} In the following we consider equivalently
  the quantity $u^\ep=\tau^\ep N^\ep$ to hide the growth rate in the
  transport term. The objective is to apply {the} Kre\u{\i}n-Rutman
  theorem to the linear operator $\cG_\ep \colon f \mapsto u$ defined
  on $X_{\! R}$, for $\alpha >0$, as solution to
  \begin{equation}\label{eq:SV_KR}
    \left\{~
      \begin{aligned}
        &\mpdvx u_i (x) + \mfrac{\alpha + \gamma_i^\ep
          (x)}{\tau_i^\ep(x)} u_i(x) = 4 \msum{j \in \cI} \lp
        \mfrac{\gamma_j^\ep(2x)}{\tau_j^\ep(2x)}
        u_j(2x) \kappa_{ji} \rp \mathds{1}^\ep
        _{\tCinterval{0}{R}}(2x)
        +\mfrac{f_i(x)}{\tau_i^\ep(x)},\\
        & u_i(0) = \delta \msum{j \in \cI} \lp \mint_0^R \mfrac{u_j
          (x)}{\tau_j^\ep(x)} \dd{x} \rp,
      \end{aligned}
    \right.
  \end{equation}
  with semi-explicit expression:
  \begin{multline}\label{eq:proof_u_semi_explicit}
      u_i(x) = u_i(0) \e^{-\eint{0}{x} \frac{\alpha +
          \gamma_i^\ep(y)}{\tau_i^\ep(y)} \dd{y}} + \,
      2 \int_0^{2x} \e^{-\eint{\frac{s}{2}}{x} \frac{\alpha +
          \gamma_i^\ep(y)}{\tau_i^\ep(y)} \dd{y}}
      \beta^\ep(s) \msum{j \in \cI} \big( u_j(s) \kappa_{ji}
      \big) \mathds{1}^\ep
      _{\tCinterval{0}{R}}(s) \dd{s}\\
      + \int_0^{x} \e^{-\eint{s}{x} \frac{\alpha +
          \gamma_i^\ep(y)}{\tau_i^\ep(y)} \dd{y}}
      \mfrac{f_i(s)}{\tau_i^\ep(s)} \dd{s}, \qquad \forall i
      \in \cI, \; \forall x \in [0,R].
  \end{multline}
  Let us verify that $\cG_\ep$ fulfills the assumptions of the
  theorem.
    
  \begin{enumerate}[label=\textit{\roman*)}, leftmargin=*, wide=0pt]
  \item\textit{$\cG_\ep$ is well defined from $X_{\! R}$ into itself.}
    We aim at proving that for any $f$ in $X_{\! R}$,~\eqref{eq:SV_KR}
    admits a solution in $X_{\!  R}$. Relying on {the} Banach-Picard
    fix-point theorem (that also provides uniqueness), we define for
    $f$ in $X_{\! R}$, the operator $T_f$ such that for $w$
    in~$X_{\! R}$, $u=T_f(w)$ is the (explicit) solution to:
    $\forall i \in \cI$, $\forall x \in (0,R)$,
    \begin{equation*}
      \left\{~
        \begin{aligned}
          &\mpdvx u_i (x) + \mfrac{\alpha + \gamma_i^\ep
            (x)}{\tau_i^\ep(x)} u_i(x) = 4
          \beta^\ep(2x) \msum{j \in \cI} \big( w_j(2x)
          \kappa_{ji} \big) \mathds{1}^\ep
          _{\tCinterval{0}{R}}(2x)
          +\mfrac{f_i(x)}{\tau_i^\ep(x)}, \\
          & u_i(0) = \delta \msum{j \in \cI} \lp \mint_0^R
          \mfrac{w_j (s)}{\tau_j^\ep(s)} \dd{s} \rp,
        \end{aligned}
      \right.
    \end{equation*}
    and prove that $T_f$ is a contraction from $X_{\! R}$ to
    $X_{\! R}$.
    
    Let $w^1$, $w^2$ be two elements of $X_{\! R}$. We define
    $u^i=T_f(w^i)$, $i \in \lb 1,2 \rb$, and $w=w^2-w^1$,
    $u=u^2-u^1$. Then $u$ satisfies: $\forall i \in \cI$,
    $\forall x \in (0,R)$,
    \begin{equation*}
      \left\{~
        \begin{aligned}
          &\mpdvx u_i (x) + \mfrac{\alpha + \gamma_i^\ep
            (x)}{\tau_i^\ep(x)} u_i(x) = 4
          \beta^\ep(2x) \msum{j \in \cI} \big( w_j(2x)
          \kappa_{ji} \big) \mathds{1}^\ep
          _{\tCinterval{0}{R}}(2x), \\
          & u_i(0) = \delta \msum{j \in \cI} \lp \mint_0^R \mfrac{w_j
            (s)}{\tau_j^\ep(s)} \dd{s} \rp,
        \end{aligned}
      \right.
    \end{equation*}
  whose explicit solution is given, for all $i$ in $\cI$,
  all $x$ in $[0,R]$, by:
  \begin{equation*}
    u_i(x) = u_i(0) \e^{-\eint{0}{x} \frac{\alpha +
        \gamma_i^\ep (y)}{\tau_i^\ep(y)} \dd{y}} + 4
    \int_0^x \e^{-\eint{s}{x} \frac{\alpha + \gamma_i^\ep
        (y)}{\tau_i^\ep(y)} \dd{y}} 
    \beta^\ep(2s) \msum{j \in \cI} \big( w_j(2s) \kappa_{ji} \big)
    \mathds{1}^\ep_{\tCinterval{0}{R}}(2s) \dd{s}.
  \end{equation*}    
  A change of variables and rough estimates,
  using~\eqref{eq:prop_convolution_ineq} to bound
  $\tau^\ep \!$, bring:
  $\forall i \in \cI, \forall x \in [0,R]$
  \begin{equation}\label{eq:in_pt_fixe_contraction}
    \begin{aligned}
      \abs{u_i(x)} &\leq \delta \msum{j \in \cI} \lp \mint_0^R
      \mfrac{\abs{w_j (s)}}{\tau_j^\ep(s)} \dd{s} \rp + 4
      \int_0^{x} \e^{-\eint{s}{x}
        \frac{\alpha+\gamma_i^\ep(y)}{\tau_i^\ep(y)}
        \dd{y}} \beta^\ep(2s) \msum{j \in \cI} \big(
      \abs{w_j(2s)} \kappa_{ji} \big)
      \mathds{1}^\ep_{\tCinterval{0}{R}}(2s) \dd{s}\\
      &\leq \bigg( \underbrace{ \mfrac{\delta M R}{\mu} + 2 \max_{i
          \in \cI} \Bigl( \msum{j \in \cI} \kappa_{ji} \Bigr)
        \int_0^{2x} \e^{-\frac{\alpha
            (x-\frac{s}{2})}{\tNorm{\tau^\eta}{L^{\! \infty}}}}
        \beta^\ep(s) \mathds{1}^\ep_{\tCinterval{0}{R}}(s)
        \dd{s}}_{ \coloneqq k_\ep(x) } \bigg) \XRnorm{w}.
    \end{aligned}
  \end{equation}
  We need $k_\ep(x)$ to be strictly less than $1$ uniformly in
  $x$. However, for later purpose we prove a bound uniform in
  $\ep$ as well, which requires a little finer work.

  First, let us impose ${\delta M R}<{\mu}$, say
  $2{\delta M R} \leq {\mu}$. Then, let us consider some small
  $\tilde{\ep} > 0$ to be fixed, and focus on the integral
  part of $k_\ep$: first for $x$ in
  \smash{$\big[0, \frac{\tilde{\ep}}{2} \big]$}
  \begin{equation*}
    \begin{aligned}
      I_\ep(x) \coloneqq \int_0^{2x} \e^{-\frac{\alpha
          (x-\frac{s}{2})}{\tNorm{\tau^\eta}{L^{\! \infty}}}}
      \beta^\ep(s)
      \mathds{1}^\ep_{\tCinterval{0}{R}}(s) \dd{s} &\leq
      \int_0^{\tilde{\ep}} \beta^\ep(s) \dd{s} = \lN
      \beta^\ep \rN_{L^1(0, \tilde{\ep})} \leq 2 \lN
      \beta \rN_{L^1(0, \tilde{\ep})}
    \end{aligned}
  \end{equation*}
  (using that $\beta^\ep$ converges to $\beta$ in
  $L^1(0, R)$), and then for $x$ in
  \smash{$\big[ \frac{\tilde{\ep}}{2},R \big]$}
  \begin{equation*}
    \begin{aligned}
      I_\ep(x) & \leq \int_0^{\min(2x,R)-\tilde{\ep}}
      \e^{-\frac{\alpha (x-\frac{s}{2})}{\tNorm{\tau^\eta}{L^{\!
              \infty}}}} \beta^\ep(s) \dd{s} + \!
      \int_{\min(2x, R) -\tilde{\ep}}^{\min(2x, R)}
      \beta^\ep(s) \dd{s}\\
      &\leq \e^{-\frac{\alpha}{\tNorm{\tau^\eta}{L^{\! \infty}}}
        \big(x-\frac{\min(2x,R)-\tilde{\ep}}{2} \big)}
      \Norm{\beta^\ep}{L^1(0,R)} +
      \Norm{\beta^\ep}{L^1(\min(2x,R)-\tilde{\ep}, \,
        \min(2x, R) )},
    \end{aligned}
  \end{equation*}
  which brings, after distinguishing whether $x$ is smaller or larger
  than \smash{$\ttfrac{R}{2}$} and using
  properties~\eqref{eq:prop_convolution_ineq} of the convolution
  product, that for any $\ep$ small enough and $x$ in
  \smash{$\big[ \frac{\tilde{\ep}}{2}, R \big]$}:
  \begin{equation*}
    I_\ep(x) \leq 2 \Big(
    \e^{-\frac{\alpha}{\tNorm{\tau^\eta}{L^{\! \infty}}}
      \tilde{\ep}} \Norm{\beta}{L^1(0, R)} + \sup_{s \in
      [\tilde{\ep}, R]} \! \big(
    \Norm{\beta}{L^1(s-\tilde{\ep}, \, s)} \big) \Big)
    \coloneqq I(\alpha, \tilde{\ep}).
  \end{equation*}
  Therefore, $\beta$ being $L^1_{loc}\big( [0, +\infty) \big)$ (see
  Assumption~\ref{as:beta_L0}), we can fix
  $\tilde{\ep} = \tilde{\ep}_{_0}$ such that
  \begin{equation*}
    \sup_{s \in [0, R ]} \! \big(
    \Norm{\beta}{L^1(s, \, s+
      \tilde{\ep}_{\texp{0}})} \big) < \mfrac{1}{4}
    \bigg( 4 \max_{i \in \cI} \Bigl( \msum{j \in \cI}
    \kappa_{ji} \Bigr)  \bigg)^{-1}.
  \end{equation*}
  This way, for any small $\ep$, $k_\ep$ is bounded by $\frac{3}{4}$
  on $\big[0, \frac{\tilde{\ep}_{\texp{0}}}{2} \big]$ and to have
  $k_\ep$ strictly less than~$1$ on
  $[\frac{\tilde{\ep}_{\texp{0}}}{2}, R]$ as well, it only remains to
  fix $\alpha=\alpha_{\texp{0}}$ such that
  \begin{equation*}
    \e^{-\frac{\alpha_{\texp{0}}}{\tNorm{\tau^\eta}{L^{\! \infty}}}
      \tilde{\ep}_{\texp{0}}} \Norm{\beta}{L^1(0, R)}
    < \mfrac{1}{4} \bigg( 4 \max_{i \in \cI} \Bigl( \msum{j \in \cI}
    \kappa_{ji} \Bigr)  \bigg)^{-1}.
  \end{equation*}
  Going back to~\eqref{eq:in_pt_fixe_contraction}, we are now able to
  conclude:
  \begin{equation*}
    \lN u \rN_{\XRindex} \leq k \lN w \rN_{\XRindex}, \qquad
    k \coloneqq \mfrac{\delta M R}{\mu} 
    + 4 \max_{i \in \cI} \Bigl( \msum{j \in \cI}
    \kappa_{ji} \Bigr)  I (\alpha_{\texp{0}},
    \tilde{\ep}_{\texp{0}}) < 1,
  \end{equation*}  
  for ${2 \delta M R} \leq {\mu}$ and $I$, $\alpha_{\texp{0}}$,
  $\tilde{\ep}_{\texp{0}}$ independent of $\ep$,
  defined by the three previous equations.  As strict contraction,
  $T_f$ thus admits a unique fix-point that is precisely the solution
  to~\eqref{eq:SV_KR}.

  To ensure $k<1$ we impose from now one to the end of the proof that
  \begin{equation}\label{eq:condition_k_less_1}
    2 \delta M R \leq {\mu},
    \qquad \alpha_{\texp{0}} \leq \alpha.
  \end{equation}
  
\item\label{item:proof_th2_2} \textit{$\cG_\ep$ is continuous.} For
  any $i$ in $\cI$, $x$ in $[0,R]$, taking the absolute value
  in~\eqref{eq:proof_u_semi_explicit} we obtain
  inequality~\eqref{eq:in_pt_fixe_contraction} with at the right-hand
  side: terms in $u$ instead of $w$, similarly controlled by
  $k \lN u \rN_{\XRindex}$, and an additional term in $f$, handled as
  follows:
  \begin{equation*}
    \int_0^{x} \e^{-\eint{s}{x} \frac{\alpha +
        \gamma_i^\ep(y)}{\tau_i^\ep(y)} \dd{y}}
    \mfrac{\abs{f_i(s)}}{\tau_i^\ep(s)} \dd{s} \leq
    \lc \mfrac{1}{\alpha} \e^{-\eint{s}{x}
      \frac{\alpha}{\tau_i^\ep(y)} \dd{y}} \rc_0^x \lN f
    \rN_{\XRindex} \leq
    \frac{1}{\alpha} \lN f \rN_{\XRindex}.
  \end{equation*}
  This provides $\abs{u_i (x)}$ with a bound uniform in $i$ and $x$
  allowing to conclude to
  \begin{equation}\label{eq:fix_point_continuity}
    \lN u \rN_{\XRindex} \leq
    \frac{1}{\lp 1-k \rp \alpha} \lN f \rN_{\XRindex}.
  \end{equation}

\item\label{item:proof_th2_3} \textit{$\cG_\ep$ is strongly
      positive.}  Let $f$ be a non-negative function in $X_{\! R}$,
  we first want to show that $u=\cG_\ep(f)$ is non-negative as
  well. It is easy to check that for any non-negative~$w$, $T_f(w)$ is
  also non-negative. Therefore, recalling that $u$ is defined as the
  fix-point of $T_f$, we have $u$ non-negative. Combined
  with~\eqref{eq:proof_u_semi_explicit} we thus have
  \begin{equation*}
    u_i(x) \geq u_i(0) \e^{-\eint{0}{x} \frac{\alpha +
        \gamma_i^\ep(y)}{\tau_i^\ep(y)} \dd{y}}
    + \int_0^x \e^{-\eint{s}{x} \frac{\alpha + \gamma_i^\ep
        (y)}{\tau_i^\ep(y)} \dd{y}}
    \mfrac{\abs{f_i(s)}}{\tau_i^\ep(s)} \dd{s},
  \end{equation*}
  with
  \begin{equation*}
    u_i(0) = \delta \msum{j \in \cI} \lp
    \mint_0^R \mfrac{u_j (s)}{\tau_j^\ep(s)} \dd{s} \rp,
  \end{equation*}
  and we deduce that if additionally $f$ is not the null function,
  then $u$ is positive on $\cS_R$.

\item \textit{$\cG_\ep$ is compact.} We aim at showing that
  the image $A$ by $\cG_\ep$ of the unit ball $\cB_{X_{\! R}}$ of
  $X_{\! R}$ is relatively compact in $X_{\!
    R}$. From~\ref{item:proof_th2_2},
  inequality~\eqref{eq:fix_point_continuity}, we know that for any $f$
  in $\cB_{X_{\! R}}$, $u=\cG_\ep(f)$ is bounded in $X_{\! R}$,
  uniformly in $f$, thus so is $\mpdvx u$ from the equation:
  $\forall i \in \cI$
  \begin{equation}\label{eq:proof_reg_dx}
    \mpdvx u_i (x) = -\mfrac{\alpha + \gamma_i^\ep
      (x)}{\tau_i^\ep(x)} u_i(x)
    + 4 \beta^\ep(2x) \msum{j \in \cI} \big(
    u_j(2x) \kappa_{ji} \big) \mathds{1}^\ep
    _{\tCinterval{0}{R}}(2x)
    +\mfrac{f_i(x)}{\tau_i^\ep(x)}.
  \end{equation}
  Therefore, by Ascoli-Arzel\`{a} theorem,
  $ A_i \coloneqq \{u_i \mid u = \cG_\ep (f), \, f \in \cB_\cX \}$,
  $i \in \cI$, is relatively compact in $\sC([0,R])$ for the supremum
  norm.  Thus from any sequence $(u^n)_{n\in \N}$ in $A$ we can
  extract, by $M$ successive extractions, a subsequence
  $(u^{\varphi(n)})_{n \in \N}$ such that every sequence
  \smash{$(u_i^{\varphi(n)})_{n \in \N}$} of $A_i$, $i$ in $\cI$,
  converges in $\sC([0,R])$. Then $(u^{\varphi(n)})_{n \in \N}$
  converges in $X_{\! R}$ and we have proved that $A$ is relatively
  compact in $X_{\! R}$.
\end{enumerate}

We can finally apply the Kre\u{\i}n-Rutman theorem that gives us the
existence and uniqueness of a positive eigenvalue $\alpha^\ep$ and a
\emph{positive} eigenvector $u^\ep$ in $X_{\! R}$ solution to
$\cG_\ep(u^\ep)= \alpha^\ep u^\ep$. Denoting $\frac{u^\ep}{\tau^\ep}$
the $X_{\! R}$ vector with components $\frac{u_i^\ep}{\tau_i^\ep}$,
$i \in \cI$, we got the existence of
\begin{equation}\label{eq:proof_lambda_N}
  \lambda_{\textit{di}}^\ep \coloneqq
  \alpha-\frac{1}{\alpha^\ep} < \alpha, \qquad N^\ep =
  \Bigg( \msum{i \in \cI} \lp \mint_0^R
  \mfrac{\abs{u_i^\ep (x)}}{\tau_i^\ep(x)} \dd{x} \rp
  \Bigg)^{-1} \frac{u^\ep}{\tau^\ep},
\end{equation}
solution in $\R \times X_{\! R}$ to the direct problem
of~\ref{pb:GFv_trunc_reg}. The $\sC^1([0,R])$ continuity of the
$N_i^\ep$, $i$ in $\cI$, is a direct consequence
of~\eqref{eq:proof_reg_dx} and the continuity of $\tau_i^\ep>0$ and
$\gamma_i^\ep$.

\textbf{Adjoint problem.} A function $\phi^\ep$ is solution to the
adjoint equation of~\ref{pb:GFv_trunc_reg} if and only if for all $i$
in $\cI$, $\hat{\phi}_i^\ep \colon x\mapsto \phi_i^\ep(R-x)$ satisfies
for all $x$ in $(0,R)$
\begin{equation}\label{pb:GFv_trunc_reg_ad_reverse}
  \left\{
    \begin{aligned}
      &\hat{\tau}_i^\ep(x) \mpdvx \hat{\phi}_i^\ep(x)
      + \big( \lambda^\ep \! + \! \hat{\gamma}_i^\ep
      (x) \big) \hat{\phi}_i^\ep(x) = 2
      \hat{\gamma}_i^\ep (x) \msum{j \in \cI} \Big(
      \hat{\phi}_j^\ep \big(\mfrac{x}{2} \big) \kappa_{ij}
      \Big)
      +\delta \hat{\phi}_i^\ep(R),\\
      &\hat{\phi}_i^\ep(0)=0, \qquad \msum{j \in \cI} \Big(
      \mint_0^R \!  N_j^\ep(s) \hat{\phi}_j^\ep(R-s)
      \dd{s} \Big) = 1, \qquad \hat{\phi}_i^\ep (x)> 0,
    \end{aligned}
  \right.       
\end{equation}
where $\hat{\tau}_i^\ep \colon x\mapsto \tau_i^\ep(R-x)$ and
$\hat{\gamma}_i^\ep \colon x\mapsto \gamma_i^\ep(R-x)$. This brings us
to a problem similar to the direct problem. By the same method we
obtain the existence of $(\lambda_{\textit{ad}}^\ep, \hat{\phi}^\ep)$
solution to~\eqref{pb:GFv_trunc_reg_ad_reverse} and thus a solution to
the adjoint problem of~\ref{pb:GFv_trunc_reg}.

\textbf{Complete problem.} We have proven the existence of
$(\lambda_{\textit{di}}^\ep,N^\ep)$ and
$(\lambda_{\textit{ad}}^\ep,\phi^\ep)$ solution to the direct and
adjoint equations of~\ref{pb:GFv_trunc_reg}. It remains to check that
$\lambda_{\textit{di}}^\ep=\lambda_{\textit{ad}}^\ep$. This is
straightforward if we integrate the direct equation against the
adjoint eigenvector:
  \begin{equation*}
    \langle \cG_\ep N^\ep,\phi^\ep \rangle
    = \langle N^\ep, {\cG^*_\ep \phi^\ep}
    \rangle,
    \qquad \ie \quad \lambda_{\textit{di}}^\ep
    \langle N^\ep,\phi^\ep \rangle
    = \lambda_{\textit{ad}}^\ep
    \langle N^\ep,\phi^\ep \rangle,
  \end{equation*}
  with compact bracket notation
  $\la f, g \ra \coloneqq \sum_i \lp \int f_i g_i \rp$. We conclude
  thanks to the normalization condition
  $\langle N^\ep,\phi^\ep \rangle =1$.
\end{proof}

\subsection{Proof of
  Lemma~\protect\ref{lemma:unif_bound_eigenels_varep} -- Estimates}
\label{ssec:proof_KR_estimates}

Relying on estimate~\eqref{eq:fix_point_continuity} and on the
expression of $\lambda^\ep$ given by~\eqref{eq:proof_lambda_N}
we derive bounds for the family
$(\lambda^\ep, N^\ep,
\phi^\ep)_{\ep>0}$.
\begin{proof}
  Let us fix $\ep>0$ and take $\alpha = \alpha_{\texp{0}}$, which
  do not contradict~\eqref{eq:condition_k_less_1}.

  \textit{Uniform upper bound for $\lambda_\ep$.}
  From~\eqref{eq:fix_point_continuity} applied to
  $f= u^\ep >0$ we have
  \smash{$\alpha^\ep \leq \frac{1}{(1-k)
      \alpha_{\texp{0}}}$}. The bound then comes directly from the
  expression of $\lambda^\ep$ given
  by~\eqref{eq:proof_lambda_N}:
  \begin{equation*}
    \lambda^\ep \leq k \alpha_{\texp{0}} \leq \alpha_{\texp{0}} =
    \lambda_{up}.
  \end{equation*}
  
  \textit{Null lower bound for $\lambda_\ep$.}
  We assume by contradiction that $\lambda^\ep$ is
  non-positive. Integrating for $i \in \cI$  and 
  $x$ in $\lp 0,R \, \rc$, the direct equation
  of~\ref{pb:GFv_trunc_reg} brings with a change of variables:
  \begin{multline*}
      \msum{i \in \cI} \Big( \tau_i^\ep(x) N_i^\ep
      (x)\Big) -\delta M+ \lambda^\ep \msum{i \in \cI}
      \mint_{0}^x N_i^\ep(s) \dd{s} + \msum{i \in \cI}
      \mint_0^x
      \gamma_i^\ep(s) N_i^\ep(s) \dd{s}\\
      = 2 \msum{i \in \cI} \mint_0^{2x} \gamma_i^\ep (s)
      N_i^\ep (s) \mathds{1}^{\ep}_{\tCinterval{0}{R}}(s)
      \dd{s}.
  \end{multline*}
  Then, $\lambda^\ep$ non-positive would imply:
  $ \forall x \in [ 0, R-\ep]$,
  \begin{equation*}
    \msum{i \in \cI} \Big( \tau_i^\ep(x) N_i^\ep
    (x)\Big) \geq~ \delta M + \mint_0^x \beta^\ep (s)
    \msum{i \in \cI} \Big(
    \tau_i^\ep(s) N_i^\ep(s) \Big) \dd{s}.
  \end{equation*}
  We can apply Gr\"{o}nwall's lemma to
  $x \mapsto \sum_{i \in \cI} \big( \tau_i^\ep(x)
  N_i^\ep (x)\big)$, that is continuous together
  with~$\beta^\ep$, and get:
  $ \forall x \in [ 0, R-\ep]$,
  \begin{equation*}
    \msum{i \in \cI} \Big(
    \tau_i^\ep(x) N_i^\ep (x) \Big) \geq \delta M
    \e^{ \,\eint{0}{x}  \beta^\ep (s)  \dd{s}}.
  \end{equation*}
  However \ref{as:beta_in_inf} and~\ref{as:tau_in_inf} (inherited by
  $\beta^\ep$ and $\tau^\ep$, resp.)  successively
  guarantee: $\forall i \in \cI$
  \begin{itemize}[leftmargin=.7cm, parsep=0cm, itemsep=0.1cm,
    topsep=0cm]
  \item the existence for all $k \in \N$ of $A>0$ such that
    \smash{$\beta^\ep(s) \geq \frac{k}{s}$} for $s \geq A$,
    which implies
    \begin{equation*}
      \msum{i \in \cI} \Big(
      \tau_i^\ep(x) N_i^\ep (x) \Big) \geq \delta M
      e^{ \eint{0}{x} \beta^\ep (s) \dd{s}}
      \geq \delta M \e^{ \eint{A}{x} \frac{k}{s} \dd{s} } = \delta M
      \Big( \mfrac{x}{A}
      \Big)^k, \qquad x \in [A, R],
    \end{equation*}
    
  \item and providing $k$ big enough, the existence of $B>A$
    s.t. $x^{-k} \tau_i^\ep(x) \leq \frac{\delta M}{A^k}$, for
    $x \geq B$.
  \end{itemize}
  Together with the normalization condition of~\ref{pb:GFv_trunc_reg},
  we find:
  \begin{equation*}
    1 = \msum{i \in \cI} \mint_0^R N_i^\ep (x) \dd{x} \geq
    \mfrac{A^k}{\delta M} \msum{i \in \cI} \mint_B^{R- \ep} 
    x^{-k} \tau_i^\ep(x) N_i^\ep (x) \dd{x} \geq
    \mint_B^{R-\ep} \dd{x} =  R - \ep - B,
  \end{equation*}
  which is contradictory as soon as $R$ gets big enough. Thus
  $\lambda^\ep$ is positive for $R$ big enough.
  
  \textit{Uniform bounds for $N^\ep$.} We fix $i$ arbitrarily
  in $\cI$, multiply the direct equation of~\ref{pb:GFv_trunc_reg} by
  \smash{$\e^{-\int_{0}^{x} \scriptscriptstyle{ (\lambda^\ep +
        \gamma_i^\ep)/\tau_i^\ep}}$} and integrate in
  size: $\forall i \in \cI$, $\forall x \in [0,R]$,
  $\forall \ep_{\texp{0}} >0$
  \begin{equation*}
    \begin{aligned}
      \tau_i^\ep (x) N_i^\ep(x) &= \delta
      \e^{-\scaleobj{.75}{\mint_0^x} \frac{\lambda^\ep +
          \gamma_i^\ep(y)}{\tau_i^\ep(y)} \dd{y}} + \,
      4 \int_0^x \e^{-\scaleobj{.75}{\mint_{\scaleobj{1.2}{s}}^x}
        \frac{\lambda^\ep +
          \gamma_i^\ep(y)}{\tau_i^\ep(y)} \dd{y}}
      \msum{j \in \cI} \Big( \gamma_j^\ep(2s)
      N_j^\ep(2s) \kappa_{ji} \Big)
      \mathds{1}^\ep_{\tCinterval{0}{R}}(2s) \dd{s}\\
      &\leq \delta + 2 \max_{i, j}( \kappa_{ji}) \Big(
      \Norm{\beta^\ep}{L^1 ([0, \ep_{\texp{0}}])}
      \Norm{\tau^\ep N^\ep}{\ell^\infty(\cI; L^\infty
        ([0, \ep_{\texp{0}}]))} +
      \Norm{\gamma^\ep}{\ell^\infty(\cI;
        L^\infty(\ep_0,R))} \Big).
    \end{aligned}
  \end{equation*}
  Using~\eqref{eq:prop_convolution_ineq}, the fact that $N^\ep$
  is positive for the lower bound, and that for some positive
  $\ep_{\texp{0}}$,
  $\Norm{\beta}{L^1 ([0, \ep_{\texp{0}}])}$ is less than
  \smash{$(8 \max_{i, j}( \kappa_{ji}))^{-1}$} such that on
  $[0, \ep_{\texp{0}}]$
  \smash{$\tau_i^\ep N_i^\ep$} is less than
  $K \coloneqq 2 (\delta + 2 \max_{i, j}( \kappa_{ji})
  \Norm{\gamma^\ep}{\ell^\infty(\cI;
    L^\infty(\ep_0,R))})$ for the upper bound, we conclude to
  what expected:
  \begin{equation*}
    N_{low}(\delta,\eta)  
    \coloneqq \mfrac{\delta}{\Norm{\tau^\eta}{L^{\! \infty}}} \e^{- \lp
      \frac{R \lambda_{up}}{\mu}+ \tNorm{\beta}{L^1(0, R)} \rp} >0,
    \quad N_{up}(\delta,\eta) \coloneqq \mfrac{\delta
      + \frac{K}{2}+ 2 \max_{i, j}( \kappa_{ji})
      \Norm{\gamma^\ep}{\ell^\infty(\cI;
        L^\infty(\ep_0,R))}}{\mu}.
  \end{equation*}
  As for the $\phi^\ep$, $\ep >0$, we can proceed the
  same way --using the adjoint equation of~\ref{pb:GFv_trunc_reg} to
  derive a $l^\infty\big(\cI; \, L^\infty(0,R) \big)$-bound uniform in
  $\ep$ as well.
\end{proof}

\begin{acknowledgements}
  We thank Marie Doumic for bringing to us this nice problem, for her
  help and corrections.
  A.R. has been partially supported by the ERC Starting Grant SKIPPER
  AD (number 306321).
\end{acknowledgements}

\newpage

\phantomsection
\addcontentsline{toc}{section}{References}
\section*{References}

\bibliographystyle{plain}
\renewcommand{\bibsection}{}

\footnotesize

  
\end{document}